\documentclass{article}

\usepackage{xr}
\makeatletter
\usepackage{authblk}
\usepackage{dsfont}
\usepackage{hyperref}

\newcommand*{\addFileDependency}[1]{
\typeout{(#1)}
\@addtofilelist{#1}
\IfFileExists{#1}{}{\typeout{No file #1.}}
}\makeatother

\usepackage{comment}
 \usepackage{amsmath, amsthm, amssymb, bbm, setspace,bigints}
 \usepackage{mathrsfs}
  \usepackage{bm}
 \usepackage[margin=1 in]{geometry}
\usepackage{caption}
\usepackage{subcaption}
\usepackage[toc,page]{appendix}     
\usepackage{pdfpages}
\usepackage{epstopdf}
\usepackage{booktabs}
\usepackage{tabularx, multirow}
\usepackage{tikz, pgfplots}
\usetikzlibrary{positioning}
\usepackage{algorithm}
\usepackage{algpseudocode}
\usepackage{adjustbox}
\usepackage{booktabs}
\usepackage{siunitx}
\usepackage[round, sort]{natbib}
\usepackage{float}

\usepackage{empheq}
\usepackage{accents}
\usepackage{bm}
\newcommand*{\B}[1]{\ifmmode\bm{#1}\else\textbf{#1}\fi}

\newcommand{\be} {\begin{eqnarray*}}
\newcommand{\ee} {\end{eqnarray*}}

\newcommand{\argmin}{\mathop{\rm argmin~}}

\newcommand{\dd}{{\rm d}}

\newcommand{\wht}{\widehat}

\newcommand{\wt}{\widetilde}

\newcommand{\KL}{D_{\mbox{\scriptsize \rm KL}} }

\newcommand{\ri}{\textrm{(i)}}
\newcommand{\rii}{\textrm{(ii)}}
\newcommand{\riii}{\textrm{(iii)}}

\DeclareMathOperator{\osc}{Osc}

\DeclareMathOperator{\eot}{EOT}
\DeclareMathOperator{\W}{W}
\DeclareMathOperator{\LSI}{LSI}
\DeclareMathOperator{\rhs}{RHS}
\DeclareMathOperator{\lhs}{LHS}
\DeclareMathOperator{\vol}{Vol}
\DeclareMathOperator{\proj}{Proj}

\def\m{\mathcal}
\def\ms{\mathscr}
\def\mb{\mathbb}

\def\mx{\mbox}

\newcommand{\matnorm}[1]{{\left\vert\kern-0.25ex\left\vert\kern-0.25ex\left\vert #1 
    \right\vert\kern-0.25ex\right\vert\kern-0.25ex\right\vert}}
\newcommand{\norm}[1]{{\left\vert\kern-0.25ex\left\vert #1 
    \right\vert\kern-0.25ex\right\vert}}

\newtheorem{theorem}{Theorem}
\newtheorem{lemma}[theorem]{Lemma}

\newtheorem{proposition}[theorem]{Proposition}

\newtheorem{rem}{Remark}
\newtheorem{definition}{Definition}

\newcommand{\ry}[1]{\textcolor{blue}{#1}}
\newcommand{\an}[1]{\textcolor{purple}{#1}}

\title{Learning Density Evolution from Snapshot Data}

\author[1]{Rentian Yao\thanks{rentian2@math.ubc.ca}}
\author[2,3]{Atsushi Nitanda\thanks{atsushi\_nitanda@cfar.a-star.edu.sg}}
\author[4]{Xiaohui Chen\thanks{xiaohuic@usc.edu}}
\author[5]{Yun Yang\thanks{yy84@umd.edu}}
\affil[1]{Department of Mathematics, University of British Columbia}
\affil[2]{CFAR and IHPC, Agency for Science, Technology and Research (A$\star$STAR)}
\affil[3]{College of Computing and Data Science, Nanyang Technological University}
\affil[4]{Department of Mathematics,
University of Southern California}
\affil[5]{Department of Mathematics, University of Maryland}

\begin{document}
\maketitle

\begin{abstract}
Motivated by learning dynamical structures from static snapshot data, this paper presents a distribution-on-scalar regression approach for estimating the density evolution of a stochastic process from its noisy temporal point clouds. We propose an entropy-regularized nonparametric maximum likelihood estimator (E-NPMLE), which leverages the entropic optimal transport as a smoothing regularizer for the density flow. We show that the E-NPMLE has almost dimension-free statistical rates of convergence to the ground truth distributions, which exhibit a striking phase transition phenomenon in terms of the number of snapshots and per-snapshot sample size. To efficiently compute the E-NPMLE, we design a novel particle-based and grid-free coordinate KL divergence gradient descent (CKLGD) algorithm and prove its polynomial iteration complexity. Moreover, we provide numerical evidence on synthetic data to support our theoretical findings. This work contributes to the theoretical understanding and practical computation of estimating density evolution from noisy observations in arbitrary dimensions.


\end{abstract}

\section{Introduction}

Learning dynamical structures from multiple snapshot data has received increasing attention in various scientific fields such as bioinformatics and social networks~\citep{schiebinger2019optimal,mackey2021can,leskovec2007graph}. For instance, in single-cell RNA sequencing data analysis~\citep{klein2015droplet, macosko2015highly}, the trajectory inference problem aims to reconstruct the evolution of gene expression in cells using static snapshot data, where each snapshot consists of (often high-dimensional) gene expression profiles captured at a single time point representing a population of cells in various states. In this paper, our primary goal is to provide a general statistical and computational framework for simultaneous inference of many marginal distributions of a stochastic process from noisy snapshot data. 


We begin with the problem setup. Let $Z := \{Z_t: t\in[0, T]\}$ be a stochastic process evolving from $t=0$ to $t=T$ on a state space $\m X \subset \mb{R}^d$ and $R_t^\ast$ denote the marginal distribution of $Z_t$. Consider $m$ \emph{fixed} time points $0\leq t_1 < \dots < t_m\leq T$. At each time point $t_j$, we have an independent sample of random $N$-point cloud from the snapshot distribution $R_{t_j}^\ast$, i.e., $Z_{t_j,i} \sim R_{t_j}^\ast$ are independent for all $j \in [m] := \{1, \dots, m\}$ and $i \in [N]\coloneqq\{1, \dots, N\}$. In reality, we allow measurement errors when the snapshot data $\{X_{t_j}^i: i\in[N]\}$ are observed using the following standard statistical model 
\begin{equation}
\label{eqn:data_gen_mech}
X_{t_j}^i = Z_{t_j, i} + \sigma_{j, i},\quad i \in[N],
\end{equation}
where $\{\sigma_{j, i}: j\in[m], i\in[N]\}$ are $Nm$ i.i.d. mean-zero Gaussian noise with density $\m K_\sigma$ on $\m X$ and variance $\sigma^2 > 0$. In the single-cell RNA application, one can think of $Z_{t_j, i}$ is the $i$-th realization of the biological process $Z$ at time $t_j$ such that $Z_{t_j, i}$ and $Z_{t_{j'}, i}$ represent distinct realizations at different time points, resulting in a total $Nm$ samples derived from realizations of $Z$. Our primary goal of this paper is to recover the evolution dynamics of the distributions $R_{t_1}^\ast, \dots, R_{t_m}^\ast$ from their associated noisy temporal marginal snapshots $\big\{\wht\mu_{t_j} = \frac{1}{N}\sum_{i=1}^N\delta_{X_{t_j}^i}: j\in[m]\big\}$. For this purpose, there are two central questions: (i) Can we build a statistically efficient estimator with sample complexity that recovers certain conventional nonparametric density estimation approaches and meanwhile sheds light on the experimental design of $(m, N)$ given limited total sample availability? (ii) How can we design a tractable algorithm to compute this estimator with guaranteed iteration complexity?


To address question (i), we observe that model~\eqref{eqn:data_gen_mech} can be naturally treated as a nonparametric distribution-on-scalar regression problem, where the response variable takes value in the space of probability distributions $\m P(\m X)$ on the state space $\m X$ with a temporal predictor. Inspired by the maximum likelihood approach for (classical) nonparametric regression problems~\citep{kiefer1956consistency}, we propose the following \emph{entropy-regularized nonparametric maximum likelihood estimator} (E-NPMLE) to estimate the discretized density flow map $t \mapsto R^*_t$ at time points $\{t_1,\dots,t_m\}$:
\begin{equation}\label{eqn: E-NPMLE}
    (\wht{R}_{t_1}, \dots, \wht{R}_{t_m}) = \argmin_{(\rho_1,\dots, \rho_m) \in \ms P(\m X)^{\otimes m}} \m F_{N, m}(\rho_1,\dots, \rho_m),
\end{equation}
where the objective functional is defined as
\begin{align}\label{eqn: obj_func_traj}
\m F_{N, m}(\rho_1,\dots, \rho_m) := -\sum_{j=1}^m\frac{t_{j+1} - t_j}{N\lambda}\sum_{i=1}^N\log\big[\m K_\sigma\ast \rho_j(X_{t_j}^i)\big] + \sum_{j=1}^{m-1}\frac{\eot_{\tau^j}^{c_j}(\rho_j, \rho_{j+1})}{t_{j+1}-t_j} + \tau\sum_{j=1}^m\int\rho_j\log\rho_j.
\end{align}
Here in~\eqref{eqn: obj_func_traj}, $\tau > 0$ is a fixed temperature parameter, $\lambda > 0$ is the coefficient of regularization, $\tau^j = (t_{j+1}-t_j)\tau$, $\eot^{c_j}_{\tau^j}(\mu, \nu)$ is the entropic optimal transport (EOT) cost between two probability densities $\mu$ and $\nu$ over the state space $\m X$ with the cost function $c_j(x, x') = -\tau^j\log\m K_{\tau^j}(x-x')$, and $\m K_\sigma\ast\mu$ denotes the convolution of the distribution $\mu$ with a Gaussian distribution with variance $\sigma^2$. We highlight that all three terms in~\eqref{eqn: obj_func_traj} have statistically meaningful interpretations for estimating $R_{t_1}^\ast, \dots, R_{t_m}^\ast$. Specifically, the first term in~\eqref{eqn: obj_func_traj} is the negative log-likelihood that quantifies the closeness between $R_{t_j}$ and $\wht\mu_{t_j}$ based on independent noisy samples~\citep{koenker2014convex, polyanskiy2020self, soloff2024multivariate, yan2024learning, yao2024minimizing, saha2020nonparametric, zhang2009generalized, aragam2024model}. The second term imposes a piecewise penalty based on the entropic optimal transport (cf. Section~\ref{subsec: EOT}), thus ensuring the smoothness of the estimated marginal distributions along time, while the third term (i.e., total negative self-entropy) ensures that all estimated marginal distributions have density functions and are non-degenerate to the observed point clouds $X^i_{t_j}$.

Our E-NPMLE perspective provides a general statistical framework for estimating the (marginal) density flow that is related to recent progress in trajectory inference when the trajectory data $X_t$ process follows a noisy contaminated stochastic differential equation (SDE) with a gradient drift vector field and a constant diffusion parameter~\citep{lavenant2024toward,chizat2022trajectory}. For more detailed comparisons, please refer to Section~\ref{subsec:related_works}.

To tackle question (ii), we first note that computation of $\wht{R}_{t_1}, \dots, \wht{R}_{t_m}$ is a challenging optimization problem because the objective functional $\m F_{N, m}$ in~\eqref{eqn: obj_func_traj} is not a (jointly) geodesically convex in a product Wasserstein space. Thus, various existing convergence results in discretization of the Wasserstein gradient flow, no matter explicit or implicit, are no longer applicable to yield an algorithmically efficient solution to our current problem~\citep{chizat2022trajectory,yao2024wasserstein,ZhuChen2025_convergence-nonconvex}. In this work, we design a new algorithm by fully harnessing the joint convexity of $\m F_{N, m}$ in the linear structure (cf. ahead Definition~\ref{def:convex}). Our key idea is to combine the gradient descent in the linear geometry with respect to the relative entropy structure in~\eqref{eqn: obj_func_traj} (see also the related minimum entropy estimator in~\eqref{eqn: obj_func} below) and estimate the (exponential of) gradient for the next iterate 
by locally and iteratively sampling in the Wasserstein geometry. It turns out that the judicious integration of the two optimization geometries leads to a much faster convergence rate of our proposed algorithm than the existing mean-field Langevin dynamics.




\subsection{Our contributions}
The main contributions of this paper are to propose a novel nonparametric approach for learning the evolution of probability density functions from snapshot data and to equip the method with a convex algorithm in a proper optimization geometry to find the estimator. Theoretically, we demonstrate superior statistical and algorithmic convergence rates compared to the existing literature, which we elaborate on in the following paragraphs.


Statistically, we determine the \emph{almost dimension-free} non-asymptotic rates of convergence for the E-NPMLE estimated marginal distributions to the ground truth distributions in the full regime of scaling behavior $(m, N)$, which exhibits a striking \emph{phase transition} phenomenon in terms of snapshot/sample frequency. Statistical rates and their transitions are summarized in Table~\ref{tab:statistical_rates}, which is a consequence of Theorem~\ref{thm: stat_rate} and Theorem~\ref{coro: cts_stat_rate}. To interpret the rates, we may regard $m$ and $N$ as the temporal and spatial resolutions, respectively.

\begin{table}[ht]
    \begin{center}
\begin{tabular}{|c|c|c|}
\hline
& Low frequency ($m\lesssim N$) & High frequency ($m\gtrsim N$) \\
\hline
Fixed design (Theorem~\ref{thm: stat_rate}) & $\frac{(\log N)^{\frac{d+1}{2}}}{\sqrt{N}}$  &  $\frac{(\log m)^{\frac{d+1}{2}}}{N^{1/3}m^{1/6}}$\\
\hline
Density flow map (Theorem~\ref{coro: cts_stat_rate}) & $\frac{(\log m)^{\frac{d+1}{2}}}{\sqrt{m}}$ & $\frac{(\log m)^\frac{d+1}{2}}{N^{1/3}m^{1/6}}$\\
\hline
\end{tabular}
\caption{Statistical rates of E-NPMLE in different regimes of snapshot/sample frequency. Fixed design refers to the regression problem at $t_1, \dots, t_m$ and density flow map refers to $t \mapsto R_t$.}
\label{tab:statistical_rates}
\end{center}
\end{table}

Consider first the low frequency setting $m \lesssim N$. When the performance is evaluated at the observed time points $t_1, \dots, t_m$, the estimation error of the E-NPMLE is solely determined by the per-snapshot sample size $N$. Thus, estimating marginal distributions at $t_1,\dots,t_m$ with snapshot observations can be treated as estimating $m$ marginal distributions separately. In particular, our rate $N^{-1/2} (\log N)^{\frac{d+1}{2}}$ matches the known convergence rate of (unregularized) NPMLE in the classical case $m = 1$ for estimating one marginal distribution based on all samples at a single time point~\citep{saha2020nonparametric}. On the other hand, if one wants to estimate all marginal distributions in $[0,T]$ (i.e., the whole density flow trajectory), our rate becomes $m^{-1/2} (\log m)^{\frac{d+1}{2}}$ reflects the statistical bottleneck due to the lack of snapshots. In the high frequency regime $m\gtrsim N$, both the temporal and spatial resolutions will affect the statistical rate of the estimator in the same way for observed densities at $t_1,\dots,t_m$ (i.e., fixed design) and all densities for $t \in [0, T]$ (i.e., density flow map).

For tasks such as reconstruction of the continuous-time density flow map $t \mapsto R^\ast_t$ in real-world applications such as single-cell data analysis~\citep{klein2015droplet, macosko2015highly}, our rate $O(\max\{m^{-\frac{1}{2}}, m^{-\frac{1}{6}}N^{-\frac{1}{3}}\})$ (up to poly-log factor) in the bottom row of Table~\ref{tab:statistical_rates} is particularly relevant in scenarios with limited sample availability, a common constraint in practice due to fixed total sample budgets. Under such constraints, our statistical findings provide practical guidance for experimental design---recommending that the number of snapshots $m$ and the sample size per snapshot $N$ be of the same order to achieve the rate $O((mN)^{-\frac{1}{4}})$, which depends on the total sample size $mN$. Notably, this rate matches the optimal rate of nonparametric estimation of a one-dimensional $1/2$-H\"{o}lder smooth function, which corresponds to the regularity of the realized trajectory from an SDE and is therefore generally unimprovable.


Computationally, we propose a particle-based algorithm called \emph{coordinate KL divergence gradient descent} (CKLGD). By leveraging the convexity of the objective functional, our algorithm achieves an algorithmic convergence rate of $O(\frac{\log k}{\sqrt{k}})$ at the $k$-th iteration. This rate substantially improves upon the $O(\frac{\log\log k}{\log k})$ rate for the mean-field Langevin algorithm, where gradient flow is applied within the framework of Wasserstein geometry~\citep{chizat2022trajectory}. As a consequence, our CKLGD algorithm has a \emph{polynomial} iteration complexity, 
in sharp contrast with the exponential iteration complexity by using the mean-field Langevin algorithm. Furthermore, our specially tailored  particle-based algorithm demonstrates better efficiency compared to the implicit KL divergence gradient descent algorithm~\citep{yao2024minimizing}, which relies on normalizing flows to approximate transport maps between two consecutive iterates. 

To prove the algorithmic convergence, we invent a new technique for analyzing the accumulation of KL-type numerical error across iterations. Previous approaches either characterize numerical error using the $L^2$-norm of the first variation~\citep{yao2024minimizing, cheng2024convergence} or introduce an additional entropy term to control numerical error during iterations~\citep{nitanda2021particle, oko2022particle}. However, neither approach addresses the KL type numerical error that arises during the sampling procedure used to numerically compute each iterate, particularly given that the objective functional loses convexity without the self-entropy term~\citep[Proposition 3.2,][]{chizat2022trajectory}. Our innovative analysis employs an interpolation between the numerical solution and the ideal (exact) solution of the subproblem in each iteration. This approach effectively reduces the error from KL divergence to the Fisher–Rao distance by using the convexity of the objective functional in the algorithmic analysis.

\subsection{Related works}\label{subsec:related_works}

\paragraph{Learning density evolution.}
\cite{lu2022learning} modeled the evolution of probability density functions by mapping from a fixed reference measure using a temporal normalizing flow~\citep{both2019temporal}. This normalizing flow model is trained by minimizing the negative log-likelihood function of all observations.
\citet{lavenant2024toward} employed a similar $M$-estimator to ours based on the minimum entropy principle, where their focus was on \emph{trajectory inference}, namely estimating the distribution of the stochastic process $Z = \{Z_t : t \in [0, T]\}$ in the path space $\Omega := \m C([0, T]; \m X)$ (i.e., continuous mappings from the time interval $[0, T]$ to the state space $\m X$) under the assumption that the process $Z$ follows a particular class of SDEs. More precisely in our setting, they considered the special case with $\sigma = 0$ (noiseless setting) and 
\begin{equation}\label{eqn:Z_SDE}
    \dd Z_t = \nabla \Psi(t, Z_t) \; \dd t + \tau \; \dd B_t,
\end{equation}
where $\Psi : [0, T] \times \m X \to \mb{R}$ is an unknown potential, $B_t$ is the standard reversible Brownian motion on $\m X$, and $\tau > 0$ is a constant temperature parameter. Let $R^\ast\in\m P(\Omega)$ be the distribution of the process $Z$, where $\m P(\Omega)$ denotes the set of all probability distributions over the path space $\Omega$. Despite the noiseless setting,~\citet{chizat2022trajectory} (cf.~also~\citet{lavenant2024toward}) propose the following minimum entropy estimator, where a Gaussian convolution is applied to each marginal $R_{t_j}$ for computational reasons, 
\begin{align}\label{eqn: obj_func}
\wht R = \argmin_{R\in\ms P(\Omega)} 
\Big\{ -\sum_{j=1}^m\frac{t_{j+1} - t_j}{N}\sum_{i=1}^N \log \big[\m K_\sigma\ast R_{t_j}(X_{t_j}^i)\big] +  \lambda\tau\KL(R\,\|\, W^\tau) \Big\},
\end{align}
where $W^\tau\in\ms P(\Omega)$ is the distribution of the process $\{\tau B_t\}$ and $\KL$ is the KL divergence of $R$ relative to $W^\tau$. The objective functional~\eqref{eqn: obj_func_traj} is the reduced formulation of~\eqref{eqn: obj_func} into a multivariate Wasserstein functional, and therefore when $Z$ is of the form~\eqref{eqn:Z_SDE}, the minimizer $\wht R$ in the path space can be reconstructed from $\wht R_{t_1}, \dots, \wht R_{t_m}$ in~\eqref{eqn: obj_func_traj}; see Proposition~\ref{prop:representer} for the connection between the E-NPMLE and minimum entropy estimator. In contrast, \emph{current work targets the simultaneous inference of the many densities in the noisy setting ($\sigma > 0$) for a more general class of stochastic processes without such restrictions.} In particular, our E-NPMLE formulation in~\eqref{eqn: E-NPMLE} and~\eqref{eqn: obj_func} maintains statistical validity for estimating the (marginal) density flow of an SDE with a curl-free component in the drift vector field. However, recovering its path-space distribution is impossible due to identifiability issues, as adding any divergence-free component to the drift vector field changes the path measure while preserving the marginals. Moreover, while their results addressed the estimation consistency, our finite-sample analysis derives an explicit rate of convergence and provides much deeper statistical insights and guidance of the experiment design in different $(m, N)$ regimes. For precise statements and implications, we refer to Theorem~\ref{thm: stat_rate} and the follow-up discussions in Section~\ref{subsec:stat_prop}.

Under the assumption that $Z$ follows an SDE, another approach is to model the velocity vector field using neural networks~\citep{neklyudov2023action, sha2024reconstructing, shen2024learning, chen2021solving, yeo2020generative, tong2020trajectorynet}.
Alternatively, instead of regressing all snapshots on the probability space, another line of research focuses on interpolating observed probability distributions. We reference several works in this direction~\citep{chewi2021fast, chen2018measure, schiebinger2019optimal, botvinick2023generative}.

\paragraph{Computation of the regularized E-NPMLE estimator.}
\cite{chizat2022trajectory} provided a grid-free mean-field Langevin algorithm for computing an $M$-estimator similar to that of~\citet{lavenant2024toward} since the computation of the latter must be restricted to grid points. The key observation of~\citet{chizat2022trajectory} was that the objective functional in~\eqref{eqn: E-NPMLE} is the sum of a total negative self-entropy and a Fr\'echet smooth functional (for the rest two terms) such that the mean-field Langevin sampling can be deployed to approximate the Wasserstein gradient flow of $\m F_{N,m}(\wht R_{t_1}, \dots, \wht R_{t_m})$. Nonetheless, since the smooth functional is not geodesically convex in the Wasserstein space, simulated annealing on the sampling step size has to be incorporated to yield a notably slow logarithmic convergence rate of $O(\frac{\log\log k}{\log k})$ in the $k$-th iteration---an algorithmic rate that has been demonstrated to be fundamentally unimprovable in the worst-case scenario for more general geodesically nonconvex optimization problems~\citep{chizat2022mean} (see also the discussion below in a general context).

\paragraph{Optimization algorithms in Wasserstein space.}
Previous algorithmic approaches for minimizing univariate and multivariate functionals over the space of probability distributions typically relied on discretizing Wasserstein gradient flows~\citep{chizat2022trajectory, yao2024wasserstein, ZhuChen2025_convergence-nonconvex, yao2022meana}. However, when the objective functional lacks (joint) convexity along geodesics, directly discretizing its corresponding Wasserstein gradient flow usually fails to produce an explicit algorithmic convergence rate. To address the lack of convexity along geodesics such as $\m F_{N,m}$ in~\eqref{eqn: obj_func_traj},~\citet{chizat2022mean} proposed an annealing scheme, while~\citet{nitanda2022convex} introduced a non-vanishing $\ell^2$ regularization to ensure a uniform log-Sobolev inequality. The former approach reduces the algorithmic convergence rate to $O(\frac{\log\log k}{\log k})$ at the $k$-th iteration, which is proven to be unimprovable in the worst-case scenario~\citep{chizat2022mean} as we discussed above, while the latter introduces an irreducible bias to the objective functional. Neither method fully exploits the joint linear convexity of the objective functional.

For training a mean-field two-layer neural network, \citet{nitanda2021particle} introduced the particle dual averaging algorithm to minimize a regularized $\ell^2$-loss, achieving an algorithmic convergence rate of $O(k^{-1})$. A stochastic variant of this algorithm was later proposed by~\citet{oko2022particle} to optimize the same objective functional.
\citet{chizat2022convergence} and~\citet{aubin2022mirror} analyzed the convergence rate of mirror descent for minimizing linearly convex functionals on the probability space under certain smoothness assumptions. \citet{yao2024minimizing} proposed the implicit KL divergence proximal descent algorithm, which achieves a convergence rate of $O(k^{-1})$ without requiring smoothness assumptions. However, their algorithm relies on normalizing flows for implementation, which can lead to inefficiencies during training.

\subsection{Organization of the paper}

The remainder of the paper is organized as follows. In Section~\ref{subsec:stat_prop}, we first provide a finite-sample statistical analysis of the proposed E-NPMLE procedure in the fixed design setting and then extend to the estimation problem of the continuous-time density flow map. In Section~\ref{sec: algo}, we present the CKLGD algorithm and its convergence results for minimizing general linearly convex functionals in the space of probability distributions. In Section~\ref{sec: main_results}, we tailor the general-purpose CKLGD algorithm to solve the E-NPMLE objective, resulting an inexact CKLGD algorithm with a polynomial iteration complexity guarantee. Section~\ref{sec: numerical_results} demonstrates the practical utility of our approach through a simulation study. Section~\ref{sec: proof} highlights the key innovations and technical challenges in the proofs of the main results. Finally, we summarize our work in Section~\ref{sec: summary}. Detailed proofs for all theoretical results are provided in the Appendix.
\section{Statistical Convergence of E-NPMLE}\label{subsec:stat_prop}
In this section, we will determine the statistical sample complexity of the E-NPMLE on both fixed design and the density flow map. We first briefly review the background of the entropic optimal transport problem.

\subsection{Background: entropic optimal transport}\label{subsec: EOT}
The entropic optimal transport (EOT) cost between two absolutely continuous probability distributions $\mu$ and $\nu$ over the space $\m X$ (denoted as $\mu, \nu \in \ms P^r(\m X)$) with cost function $c$ is
\begin{align}\label{eqn: eot}
\eot_{\varepsilon}^c(\mu, \nu) := \min_{\gamma\in\Pi(\mu, \nu)} \int_{\m X\times\m X} c(x, x')\,\dd\gamma(x, x') + \varepsilon\KL(\gamma\,\|\,\mu\otimes\nu),
\end{align}
where $\Pi(\mu, \nu)$ denotes the set of all probability distributions over $\m X\times \m X$ with marginal distributions $\mu$ and $\nu$, and $\varepsilon > 0$ is the coefficient of the entropic regularization. When $\varepsilon = 0$, the EOT cost degenerates to the Wasserstein distance between $\mu$ and $\nu$ with cost function $c$. It is known that there exists a unique optimal coupling $\gamma^\ast$ that solves the EOT problem~\eqref{eqn: eot}. Moreover, the solution $\gamma^\ast$ satisfies
\begin{align*}
\frac{\dd\gamma^\ast}{\dd \mu\otimes\nu}(x, x') = e^{\frac{\varphi(x) + \psi(x') - c(x, x')}{\tau}},
\end{align*}
where $\varphi\in L^1(\mu)$ and $\psi\in L^1(\nu)$ are the Schr\"{o}dinger potential functions satisfying the following first-order optimality condition of~\eqref{eqn: eot}, also known as the Schr\"{o}dinger system
\begin{align}\label{eqn: schrodinger}
\varphi(x) = -\tau\log\Big(\int_{\m X}e^{\frac{\psi(x') - c(x, x')}{\tau}}\,\dd\nu(x')\Big)
\quad\mx{and}\quad
\psi(x') = -\tau\log\Big(\int_{\m X}e^{\frac{\varphi(x) - c(x, x')}{\tau}}\,\dd\mu(x)\Big).
\end{align}
Though the explicit solution of~\eqref{eqn: schrodinger} is typically intractable, numerous efficient algorithms for numerically computing the solution have been developed. The most notable include the Sinkhorn algorithm~\citep{cuturi2013sinkhorn} or the iterative proportional fitting algorithm~\citep{kullback1968probability, ruschendorf1995convergence}. For more details of entropic optimal transport, we refer to the review papers~\citet{leonard2014survey} and \citet{chen2021stochastic}.

\subsection{Statistical rates}\label{subsec:stat_rates}

In this section, we provide a non-asymptotic convergence analysis of the E-NPMLE in the full regime of $(m, N)$, suggesting an explicit choice of the number of snapshots to fully leverage all available samples. Recall that $m$ is the number of snapshots and $N$ is the sample size per snapshot. Our results offer valuable guidance for researchers in designing and conducting experimental studies. Let $d^2_{\rm H}(p, q) = \int(\sqrt{p} - \sqrt{q})^2\,\dd x$ denote the squared Hellinger distance between two probability distributions. Without loss of generality, we consider the unit time interval with $T = 1$.





Our first main result presents an almost dimension-free statistical rate of convergence for E-NPMLE estimated marginals to the ground true marginal distributions on the fixed time points $t_1, \dots, t_m$.

\begin{theorem}[Statistical rate of convergence: fixed design]\label{thm: stat_rate}
Assume there is a constant $E > 0$ such that $E^{-1} \leq \tau\KL(R^\ast\,\|\,W^\tau) \leq E$, and the time step satisfies $\Delta_m := \max_j \{t_{j+1}-t_j\} = O(m^{-1})$.
Let $C_\delta, C_\lambda > 0$ be two sufficiently large constants, and
\begin{align*}
\delta_{N, m} = C_\delta \min\Big\{\frac{1}{N^{1/2}}, \frac{1}{N^{1/3}m^{1/6}}\Big\}\big(\log\max\{m, N\}\big)^{\frac{d+1}{2}}
\quad\mx{and}\quad
\lambda_{N, m} = C_\lambda \delta_{N, m}^2.
\end{align*}
Then, with the choice of $\lambda=\lambda_{N,m}$, it holds with probability at least $1 - 2e^{-\frac{N\delta_{N, m}^2}{2\Delta_m}}$ that
\begin{align}\label{eqn: finite_marginal_converge}
\sum_{j=1}^m (t_{j+1} - t_j)d_{\rm H}^2\big(\m K_\sigma\ast R_{t_j}^\ast, \m K_\sigma\ast \wht R_{t_j}\big) \lesssim \delta_{N, m}^2.
\end{align}
\end{theorem}

\begin{rem}[Choice of $\tau$]
Previously,~\citet{lavenant2024toward} and~\citet{chizat2022trajectory} choose $\tau$ same as the temperature coefficient in the SDE~\eqref{eqn:Z_SDE} to estimate the path-space distribution of the SDE. In practice, when $\tau$ is unknown, they recommend using a plug-in type minimum entropy estimator by replacing $\tau$ in~\eqref{eqn: obj_func} with an estimated value.
As our goal is to estimate the marginal distributions $R_{t_1}^\ast, \dots, R_{t_m}^\ast$ from noisy snapshots, we simply assume that $\tau$ is known in our theoretical analysis.
\end{rem}

\begin{rem}[Extreme case $m=1$: connection with unregularized NPMLE]
When only $m=1$ snapshot is available, the problem reduces to estimating the marginal distribution of all samples at a single time point, which aligns with the definition of the (unregularized) NPMLE problem~\citep{kiefer1956consistency}. In this setting, our result implies the convergence rate
\begin{align*}
d_{\rm H}(\m K_\sigma\ast R_{t_1}^\ast, \m K_\sigma\ast \wht R_{t_1}) \lesssim \delta_{N, 1} = O\Big(\frac{(\log N)^{\frac{d+1}{2}}}{\sqrt{N}}\Big),
\end{align*}
which precisely matches the existing statistical convergence rate of NPMLE derived in Corollary 2.2 by~\citet{saha2020nonparametric} with the same poly-log factors.
\end{rem}

\begin{rem}[Extreme case $N=1$]
When only $N=1$ sample can be observed at each time point, it is still possible to estimate all the $m$ marginal distributions with the statistical rate $O\big(\frac{(\log m)^\frac{d+1}{2}}{m^{1/6}}\big)$, due to the smoothness of the marginal density evolution guaranteed by the regularization terms in~\eqref{eqn: obj_func_traj}.
\end{rem}

Next, we consider the problem of extending the estimated marginal distributions $(\wht{R}_{t_1}, \dots, \wht{R}_{t_m})$ to the whole density flow map on $t \in [0,1]$. Following~\cite{chizat2022trajectory} in the trajectory inference problem, our density flow estimator is defined as
\begin{equation}\label{eqn:density_flow_estimator}
    \wht R(\cdot) = \int_{\m X^{\otimes m}} W^\tau(\cdot\,|\,X_{t_1} = x_1, \dots, X_{t_m} = x_m)\,\dd R_{t_1,\dots, t_m}(x_1, \dots, x_m),
\end{equation}
where $R_{t_1,\dots, t_m}(\dd x_1, \dots, \dd x_m) = \gamma_{1, 2}(\dd x_1, \dd x_2)\gamma_{2, 3}(\dd x_3\,|\,x_2)\cdots \gamma_{m-1, m}(\dd x_m\,|\,x_{m-1})$ and $\gamma_{j, j+1}$ is the optimal coupling of the entropic optimal transport problem $\eot_{\tau^j}^{c_j}(\wht{R}_{t_{j}}, \wht{R}_{t_{j+1}})$. In practice, $\wht R$ in~\eqref{eqn:density_flow_estimator} can be computed through a simulation-based method by first sampling from the couplings $\gamma_{1,2}, \dots, \gamma_{m-1, m}$ and then simulating the Brownian bridges of these samples with sufficiently small time steps. Our density flow map estimator~\eqref{eqn:density_flow_estimator} is the same as the reduced formulation of the minimum entropy estimator~\citep{chizat2022trajectory} when the data process $\{X_t\} = \{Z_t\}$ in the noiseless setting $\sigma = 0$ has an SDE of form~\eqref{eqn:Z_SDE} with a gradient drift vector field and constant diffusion coefficient. The following proposition, due to Theorem 3.1 by~\cite{chizat2022trajectory} and Proposition 3.2 by~\cite{lavenant2024toward}, characterizes the precise relationship that converts a minimization problem on $\ms P(\Omega)$ to a minimization problem on $\ms P(\m X)^{\otimes m}$.

\begin{proposition}[Connection between E-NPMLE and minimum entropy estimator]\label{prop:representer}
The density flow map estimator $\wht{R}$ constructed in~\eqref{eqn:density_flow_estimator} is a solution of~\eqref{eqn: obj_func}, and vice versa.
\end{proposition}

We shall highlight that, when $Z$ is not an SDE of form~\eqref{eqn:Z_SDE}, Proposition~\ref{prop:representer} has an \emph{implicit bias} property, implying that the minimum entropy estimator still yields a Markovian random process with the same marginal distributions as the solution of E-NPMLE. On the other hand, even in the case when $Z$ is indeed an SDE, our inexact CKLGD is a different algorithm from the mean-field Langevin algorithm~\citep{chizat2022trajectory} for solving the same estimator.


The following result presents the statistical convergence rate of the estimated density flow map $t\mapsto \wht R_t$ towards the ground-truth map $t\mapsto R_t^\ast$ on the time interval $[0, 1]$.

\begin{theorem}[Statistical rate of convergence: density flow map]\label{coro: cts_stat_rate}
Under the same assumptions as in Theorem~\ref{thm: stat_rate}, it holds that
\begin{align}\label{eqn: estim+discrete}
\begin{aligned}
\int_0^1 d_{\rm H}^2(\m K_\sigma\ast R_t^\ast, \m K_\sigma\ast\wht R_t)\,\dd t
&\lesssim \sum_{j=1}^m(t_{j+1}-t_j)d_{\rm H}^2(\m K_\sigma\ast R_{t_j}^\ast, \m K_\sigma\ast \wht R_{t_j})\\
&\qquad\qquad\qquad+ \bigg[1 + \sqrt{m\int_0^1 d_{\rm H}^2(\m K_\sigma\ast R_t^\ast, \m K_\sigma\ast\wht R_t)\,\dd t}\bigg]\Delta_m.
\end{aligned}
\end{align}
When $\Delta_m = O(m^{-1})$, the above inequality implies that
\begin{align}\label{eqn: cts_marginal_converge}
\int_0^1 d_{\rm H}^2(\m K_\sigma\ast R_t^\ast, \m K_\sigma\ast \wht R_t)\,\dd t 
\lesssim \max\{\delta_{N, m}^2, m^{-1}\} 
\lesssim \max\Big\{\frac{1}{m}, \frac{1}{N^{2/3}m^{1/3}}\Big\}\big(\log m\big)^{d+1}
\end{align}
holds with probability at least $1 - 2e^{-\frac{N\delta_{N, m}^2}{2\Delta_m}}$.
\end{theorem}

\begin{rem}[Time discretization error]
The upper bound~\eqref{eqn: estim+discrete} demonstrates the impact of estimation error and time discretization error when estimating the density flow map $t\mapsto R_t^\ast$.
Generally, using Riemannian sum to approximate the corresponding Riemannian integral of a $1/2$-H\"{o}lder smooth function will cause discretization error of order $O(m^{-\frac{1}{2}})$. Fortunately, the leading term in our problem is also proportional to the integral of the squared Hellinger distance, improving the order of time discretization error to $O(m^{-1})$. We refer to the proof in Appendix~\ref{app: coro_pf} and Lemma~\ref{lem: diff_H2} for more details.
\end{rem}

\begin{rem}[Optimality of the rate]
To estimate the marginal distribution of the entire process, it is easy to verify that the presented rate achieves the minimal value $O\big(\frac{(\log m)^\frac{d+1}{2}}{(Nm)^\frac{1}{4}}\big)$ with respect to the total sample size $n = Nm$.
We conjecture that this rate is optimal due to the connection of our problem with nonparametric regression.
Specifically, in nonparametric regression, the optimal statistical rate for estimating an $\alpha$-H\"{o}lder smooth function in $d$-dimensional Euclidean space is $n^{-\frac{\alpha}{2\alpha+d}}$. In our problem, it is shown that $\m K_\sigma\ast R_t(x)$ is $1/2$-H\"{o}lder smooth with respect to time $t$ (one-dimensional) but infinitely smooth with respect to the state $x$. Therefore, the optimal rate should be $n^{-\frac{1/2}{2\cdot 1/2+1}} = n^{-\frac{1}{4}}$, up to logarithmic factors.
\end{rem}

\section{Coordinate KL Divergence Gradient Descent Algorithm}\label{sec: algo}
In this section, we will first provide a brief overview of the KL divergence gradient flow in Section~\ref{subsec:convex_fun}. With this background knowledge, in Section~\ref{sec: exact CKLGD}, we will present a general-purpose \emph{coordinate KL divergence gradient descent} (CKLGD) algorithm for minimizing jointly linearly convex multivariate functionals on the probability space and establish its convergence. 

\subsection{Background: linearly convex functionals in Wasserstein space}\label{subsec:convex_fun}
Let $\m F:\ms P^r(\m X)^{\otimes m}\to\mb R$ be a lower semi-continuous multivariate functional. The \emph{first variation} of $\m F$ at $(\rho_1, \dots, \rho_m)$ with respect to the $j$-th coordinate is defined as a map $\frac{\delta\m F}{\delta\rho_j}(\rho):\m X\to\mb R$, such that for any perturbation $\chi_j = \rho_j' - \rho_j$ with $\rho_j'\in\ms P^r(\m X)$, the directional derivative satisfies
\begin{align*}
\frac{\dd}{\dd\varepsilon}\m F(\rho_1, \dots, \rho_j + \varepsilon\chi_j,\dots \rho_m)\bigg|_{\varepsilon = 0} = \int_{\m X}\frac{\delta\m F}{\delta\rho_j}(\rho)\,\dd\chi_j.
\end{align*}
The first variation $\frac{\delta\m F}{\delta\rho_j}$ can be treated as the generalization of gradient on the Euclidean space to the space of probability distributions.

\begin{definition}[Linear convexity]\label{def:convex}
    A multivariate functional $\m F:\ms P^r(\m X)^{\otimes m}\to\mb R$ is said to be \textbf{jointly linearly convex} if $\forall \rho_1, \dots, \rho_m, \rho_1', \dots, \rho_m' \in \ms P^r(\m X)$, we have
\begin{align}\label{eqn: convex}
\m F(\rho_1', \dots, \rho_m') \geq \m F(\rho_1, \dots, \rho_m) + \sum_{j=1}^m\int_{\m X}\frac{\delta\m F}{\delta\rho_j}(\rho)(x_j)\,\dd(\rho_j' - \rho_j). 
\end{align}
\end{definition}
We emphasize that the above definition is distinct from the well-known geodesic convexity in the Wasserstein space. In fact, linear convexity and geodesic convexity are not directly comparable, e.g., see Remark 1 in~\citet{yao2024minimizing}.

The key idea of our algorithm is to discretize the KL divergence gradient flow~\citep{yao2024minimizing} coordinately for minimizing the jointly linearly convex functional $\m F:\ms P^r(\m X)^{\otimes m}\to\mb R$. The following result presents the advantage of using the (continuous-time) KL divergence gradient flow to minimize a jointly linearly convex functional on the probability space---the functional value converges to its minimum at a polynomial rate.

\begin{proposition}\label{prop:KLflow}
For a multivariate functional $\m F:\ms P^r(\m X)^{\otimes m}\to\mb R$, its KL divergence gradient flow $\rho(t) = \big(\rho_1(t), \dots, \rho_m(t)\big)$ is defined by the following ordinary differential equation (ODE) system
\begin{align}\label{eqn: KLflow}
\frac{\dd}{\dd t}\log\rho_j(t) = -\frac{\delta\m F}{\delta\rho_j}\big(\rho(t)\big) + \int_{\m X}\frac{\delta\m F}{\delta\rho_j}\big(\rho(t)\big)\,\dd\rho_j(t).
\end{align}
If $\m F$ is jointly linearly convex with global minimum $\m F^\ast$ and minimizer $\rho^\ast$, for any $T > 0$ it holds that
\begin{align}\label{eqn: continuous_rate}
\min_{0\leq t\leq T} \m F\big(\rho(t)\big) - \m F^\ast \leq \frac{1}{T}\KL\big(\rho^\ast\,\big\|\,\rho(0)\big).
\end{align}
\end{proposition}
Proposition~\ref{prop:KLflow} simply extends Theorem 1 by~\cite{yao2024minimizing} from univariate case to multivariate case. We thus omit the proof.

\subsection{CKLGD for general convex functionals}\label{sec: exact CKLGD}

Given a generic multivariate functional $\m F:\ms P^r(\m X)^{\otimes m}\to\mb R$ with joint linear convexity in the sense of~\eqref{eqn: convex}, we aim to derive a practical algorithm to minimize $\m F$. In view of Proposition~\ref{prop:KLflow}, a natural approach is to consider the discretization of the continuous KL divergence gradient flow dynamics~\eqref{eqn: KLflow}. Here, we choose an explicit discretization scheme for its computational tractability and efficiency which are important considerations in minimizing the regularized E-NPMLE functional in~\eqref{eqn: obj_func_traj} (cf. Section~\ref{sec: inexact CKLGD}). Specifically, in the $k$-th iteration, we update all coordinates in parallel by discretizing the ODE~\eqref{eqn: KLflow} as
\begin{align*}
\frac{1}{\eta_k}\big[\log\rho_j^k - \log\rho_j^{k-1}\big] = C_{j,k} - \frac{\delta{\m F}}{\delta\rho_j}(\rho_j^{k-1}), \quad j\in[m]
\end{align*}
where $\eta_k$ is the step size of discretization, and $C_{j, k}$ is the normalizing constant ensuring that $\rho_j^k$ is a probability distribution. This discretization is also equivalent to solving the following minimization problem:
\begin{align}\label{eqn: CKLGD}
\rho_j^k = \argmin_{\rho_j\in\ms P^r(\m X)} \int_{\m X}\frac{\delta\m F}{\delta\rho_j}(\rho^{k-1})\,\dd(\rho_j - \rho_j^{k-1}) + \frac{1}{\eta_k}\KL(\rho_j\,\|\,\rho_j^{k-1}), \quad j\in[m].
\end{align}
The first term in~\eqref{eqn: CKLGD} corresponds to a linearization around the previous iterate $\rho^{k-1}$, while the second term penalizes the difference between $\rho^k$ and $\rho^{k-1}$, preventing the new iterate $\rho^{k}$ from deviating too far from $\rho^{k-1}$.
Algorithm~\ref{alg: coordinate_KL} provides the pseudocode for our CKLGD algorithm, where~\eqref{eqn: CKLGD} is used as its subproblem to define the current iterate.
\begin{algorithm}[ht]
\caption{Coordinate KL divergence gradient descent (CKLGD) algorithm}\label{alg: coordinate_KL}
\begin{algorithmic}
\Require objective functional $\m F$; initialization $\rho^0 = (\rho_1^0, \dots, \rho_m^0)$; number of iterations $K$; a sequence of step sizes $\{\eta_k: k\in[K]\}$.
\Ensure solution $\rho^K$ in the $K$-th iteration.
\For{$k \gets 1$ to $K$}
  \For{$j\gets 1$ to $m$}
    \State $\rho_j^k = \argmin_{\rho_j\in\ms P(\m X)} \int_{\m X}\frac{\delta\m F}{\delta\rho_j}(\rho^{k-1})\,\dd(\rho_j - \rho_j^{k-1}) + \frac{1}{\eta_k}\KL(\rho_j\,\|\,\rho_j^{k-1})$
  \EndFor
\EndFor
\end{algorithmic}
\end{algorithm}

Ideally, when the objective functional possesses certain weak notion of smoothness, the linearization and discretization errors can be effectively controlled and will not impact the algorithmic convergence rate when the step size is not too large. The following theorem formalizes this idea and establishes a convergence rate of $O(\frac{\log k}{\sqrt{k}})$ for the CKLGD algorithm when minimizing a jointly linearly convex objective functional $\m F$ with uniformly bounded first variation (i.e., a Lipschitz condition in the Fr\'echet sense). This convergence result parallels classical optimization results for minimizing non-smooth convex functions in the Euclidean space using mirror gradient descent.
\begin{theorem}[Convergence rate of CKLGD for minimizing non-smooth convex functionals]\label{thm: algo_convergence_general}
Assume $\m F$ is jointly linearly convex, and the solution of~\eqref{eqn: CKLGD} exists for every $k\in[K]$ and $j\in[m]$.  
If $\m F$ has the uniformly bounded first variation, i.e.,
\begin{align}\label{eqn:bounded_FV}
\sup_{\rho\in\ms P^r(\m X)^{\otimes m}}\Big\|\frac{\delta\m F}{\delta\rho_j}(\rho)\Big\|_{L^\infty(\m X)}\leq L_j
\end{align}
for some constants $L_j \geq 0$, then for any $\rho\in\ms P_2^r(\m X)^{\otimes m}$ we have
\begin{align}\label{eqn: algo_converge}
\min_{0\leq k\leq K-1}\m F(\rho^k) - \m F(\rho) \leq \frac{\KL(\rho\,\|\,\rho^0)}{\eta_1 + \dots + \eta_K} + \frac{\sum_{k=1}^K\eta_k^2\sum_{j=1}^m L_j^2}{2(\eta_1 + \dots + \eta_K)}.
\end{align}
\end{theorem}

We remark that the first term in~\eqref{eqn: algo_converge} arises from the inherent properties of using KL divergence gradient descent to minimize a convex functional, while the second term in~\eqref{eqn: algo_converge} represents the discretization error, which can be further reduced if $\m F$ possesses stronger smoothness beyond having a bounded first variation. Note that the second term involves a quadratic bias that is due to linearization in the explicit scheme~\eqref{eqn: CKLGD}. Minimizing the bound of~\eqref{eqn: algo_converge} with the step size $\eta_k = k^{-1/2}$, the rate of convergence for CKLGD becomes
\begin{equation}\label{eqn:CKLGD_rate}
    \min_{0\leq k\leq K-1}\m F(\rho^k) - \m F(\rho) = O\Big(\frac{\log K}{\sqrt{K}}\Big),
\end{equation}
which aligns with classical convergence results for minimizing non-smooth convex functions using subgradient descent or mirror descent in the Euclidean space~\citep[Theorem 3.1, Theorem 3.5,][]{lan2020first}. 
While the implicit discretization scheme can improve the algorithmic convergence rate to $O(K^{-1})$ as shown by~\citet{yao2024minimizing} under a univariate setting,
computing each iterate is however often much more challenging~\citep{yao2024minimizing}.



\begin{rem}
The assumption of uniformly bounded first variation~\eqref{eqn:bounded_FV} is analogous to the uniform Lipschitz condition in convex optimization in the Euclidean space, which is commonly made when the objective function lacks stronger smoothness. In our proof, we only require the first variation to be uniformly bounded at all iterates $\rho^k$ for $k=1 \dots, K$. In later sections, we will see that if there exists a functional $\m G$ such that $\m F(\rho) = \m G(\rho) +\tau\int\rho\log\rho$, only the uniformly bounded first variation of $\m G$ is required.
\end{rem}

\section{Computing E-NPMLE via Inexact CKLGD Algorithm}\label{sec: main_results}
In Section~\ref{sec: inexact CKLGD}, we first tailor the CKLGD algorithm to compute the E-NPMLE estimator $\wht{R}_{t_1}, \dots, \wht{R}_{t_m}$ defined by~\eqref{eqn: E-NPMLE} and~\eqref{eqn: obj_func_traj}. The algorithm is called the inexact CKLGD algorithm because each subproblem~\eqref{eqn: CKLGD} can be efficiently approximated using sampling methods, leveraging the special structure of the objective functional $\mathcal F$ in our context. We will derive the algorithmic convergence rate in Section~\ref{sec: algo_converge}.

\subsection{Solving E-NPMLE via inexact CKLGD}\label{sec: inexact CKLGD}
In this section, we tailor the CKLGD algorithm to compute the E-NPMLE by minimizing the objective functional in~\eqref{eqn: obj_func_traj} while carefully accounting for the computational error arising in approximately solving the subproblem~\eqref{eqn: CKLGD} in each iteration. 
Unlike MFLD, which suffers from slow convergence, our CKLGD algorithm better aligns with the joint linear convexity of $\mathcal{F}_{N, m}$ by discretizing the KL divergence gradient flow.


Now, we describe our inexact CKLGD to solve the E-NPMLE problem. With the explicit expression of the first variation of $\m F_{N, m}$, the solution $\rho^k = (\rho_1^k, \dots, \rho_m^k)$ of the subproblem~\eqref{eqn: CKLGD} to minimize $\m F_{N, m}$ using CKLGD satisfies
\begin{align}
&\qquad \qquad \qquad \rho_j^k(y_j) \propto \big[\rho_j^{k-1}(y_j)\big]^{1-\tau\eta_k}\exp\Big\{-V_j\big(y_j; \rho^{k-1}\big)\Big\},\label{eqn: CKLGD_traj}\\
&\mx{where}\qquad V_j\big(y_j; \rho^{k-1}\big) := -\frac{t_{j+1}-t_j}{N\lambda}\sum_{i=1}^N\frac{\m K_\sigma(X_{t_j}^i - y_j)}{\m K_\sigma\ast\rho_{j}^{k-1}(X_{t_j}^i)} + \frac{\varphi_{j, j+1}^{k-1}(y_j)}{t_{j+1}-t_j} +  \frac{\psi_{j, j-1}^{k-1}(y_j)}{t_j - t_{j-1}} \label{eqn: Vj_def}.
\end{align}
Here, $\varphi_{j, j+1}^{k-1}$ and $\psi_{j+1, j}^{k-1}$ are the Schr\"{o}dinger potential functions introduced in Section~\ref{subsec: EOT} for solving the entropic optimal transport problem $\eot_{{\tau^j}}^{c_j}(\rho_j^{k-1}, \rho_{j+1}^{k-1})$.

Computing the density function of $\rho_j^k$ in~\eqref{eqn: CKLGD_traj} is challenging in practice due to the computationally intractable normalizing constant. An alternative approach is to sample from the distribution $\rho_j^k$. Assuming $\rho_j^0$ is the uniform distribution over $\m X$, iterative application of the updating formula~\eqref{eqn: CKLGD_traj} yields
\begin{align}\label{eqn: exact_sampling}
\rho_j^k(y_j) \propto \exp\bigg\{-\sum_{l=1}^k\Big[\eta_l\prod_{l < l' \leq k}(1-\tau\eta_{l'})\Big]V_j\big(y_j; \rho^{l-1}\big)\bigg\}.
\end{align}
To sample from such a distribution, the unadjusted Langevin algorithm~\citep[ULA,][]{dalalyan2017theoretical, wibisono2018sampling} is a popular choice. It is well established that ULA is \emph{biased} and converges exponentially fast when the target measure is log-concave or satisfies a log-Sobolev inequality~\citep{dalalyan2017theoretical, chewi2024analysis, vempala2019rapid}. In particular, given that only a finite number of sampling iterations of ULA can be performed, and the fact that ULA introduces bias even with an infinite number of iterations, the algorithm may sample from a different probability distribution $\wht\rho_j^k$ that is close to $\wt\rho_j^k$ in the sense that $\KL(\wht\rho_j^k\,\|\,\wt\rho_j^k)$ is small.
Accordingly, we modify the updating formula~\eqref{eqn: exact_sampling} by incorporating an extra quadratic term, which ensures the log-Sobolev inequality, leading to contraction in each iteration. To be precise, let $\wht\rho^1, \dots, \wht\rho^{k-1}\in\ms P^r(\m X)^{\otimes m}$ be the iterates derived from the previous $(k-1)$ iterations. In the $k$-th iteration,
we aim to sample from the distribution
\begin{align}
&\wt\rho_j^k(y_j) \propto \exp\bigg\{-\sum_{l=1}^k\Big[\eta_l\prod_{l<l'\leq k}(1-\tau\eta_{l'})\Big]\big[V_j\big(y_j; \wht\rho^{l-1}\big) + \alpha_l\|y_j\|^2\big]\bigg\}\label{eqn: explicit_quad_anneal_update},
\end{align}
where $\{\alpha_l \geq 0: l\in\mb Z_+\}$ are the coefficients of the quadratic terms. 
The introduced quadratic term in~\eqref{eqn: explicit_quad_anneal_update} can reduce the inner iteration sampling complexity while maintaining the same accuracy compared with directly sampling from $\rho_j^k$ in~\eqref{eqn: exact_sampling}. As will be demonstrated shortly (Theorem~\ref{thm: traj_CKLGD} and Remark~\ref{rmk:extra_quad_term}), this quadratic term does not compromise the convergence of outer iterations for minimizing the objective functional $\m F_{N,m}$ when the coefficients $\{\alpha_l\}_{l\in\mb Z_+}$ are appropriately selected.

\begin{figure}[ht]
\centering
\scalebox{1}{\subfloat[Mean-field Langevin dynamics.]{\includegraphics[width=.5\columnwidth]{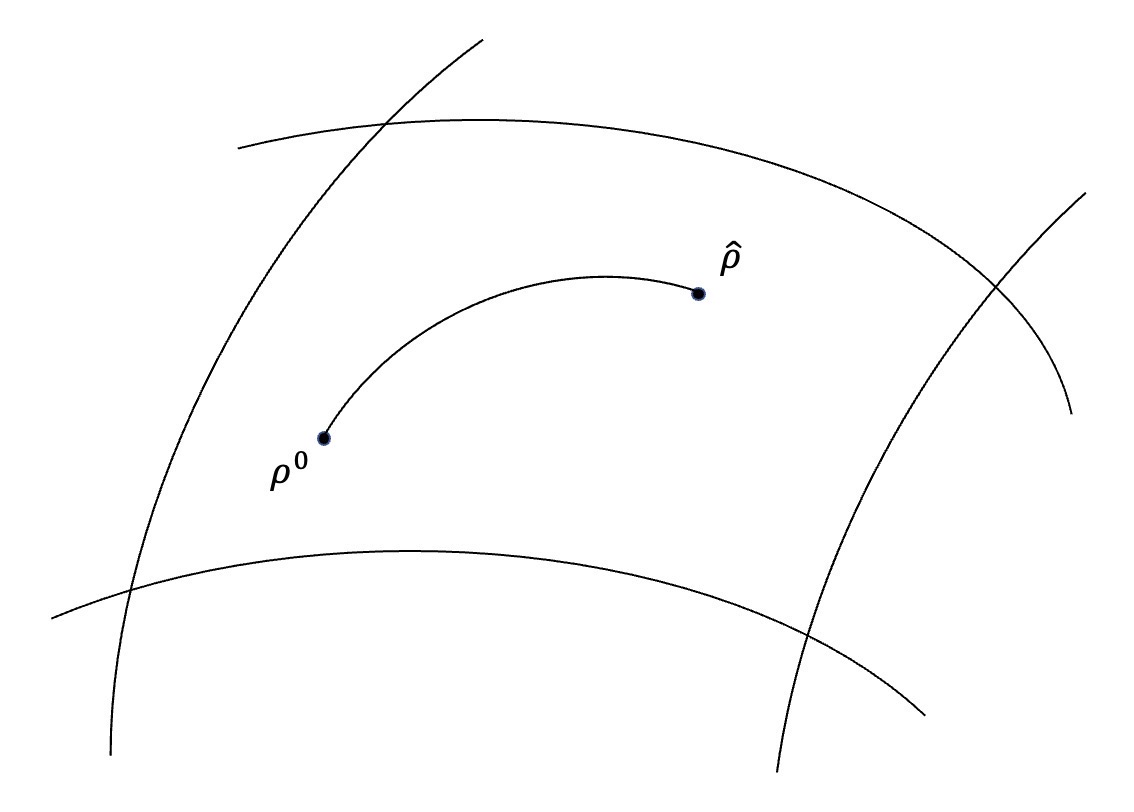}\label{fig: MFLD}}
\subfloat[Inexact coordinate KL divergence gradient descent.]{\includegraphics[width=.5\columnwidth]{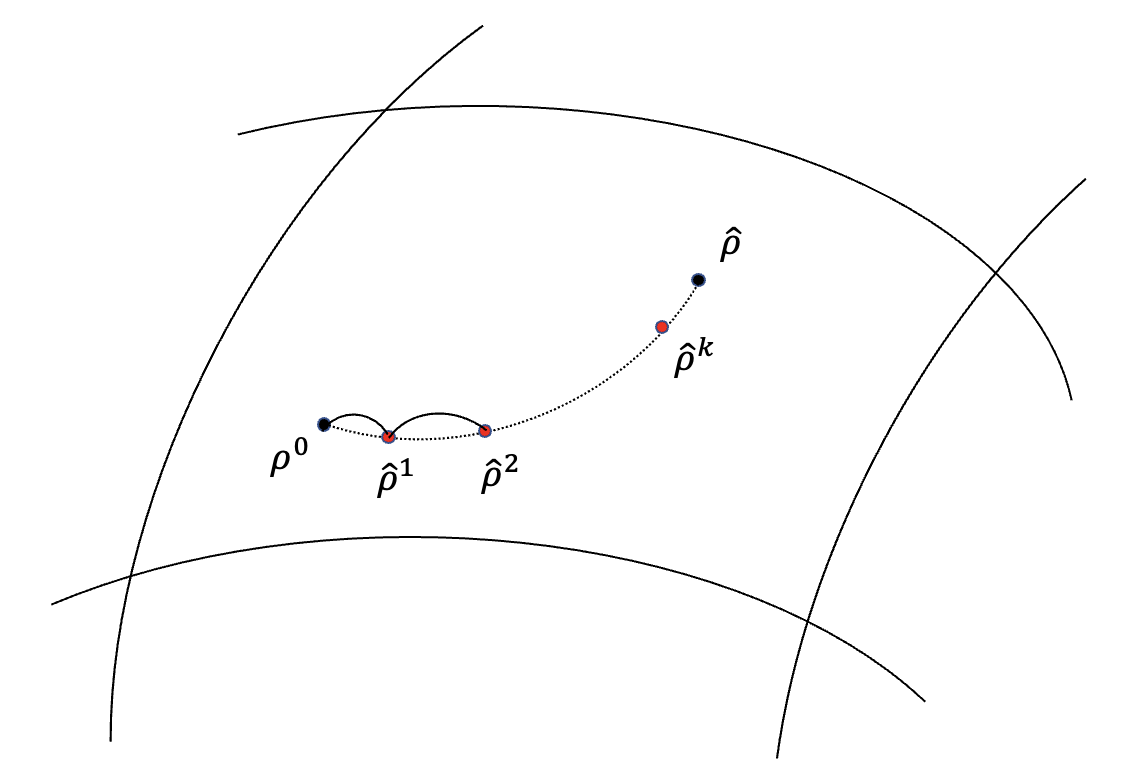}\label{fig: CKLGD}}}
\caption{High-level comparison of using the MFLD algorithm and the inexact CKLGD algorithm to minimize $\m F_{N,m}$. All solid lines represent mean-field Langevin dynamics and dotted lines represent KL divergence gradient flow. (a) The MFLD algorithm directly applies the mean-field Langevin dynamics (solid line) to compute the global minimum of $\m F_{N, m}$. Due to the nonconvexity along geodesics, MFLD with annealing requires $O(e^{\frac{C}{\varepsilon}})$ total iterations to achieve $\varepsilon$-accuracy. (b) Inexact CKLGD discretizes the KL divergence gradient flow (dotted line) and uses MFLD (solid line) to compute each iterate. Inexact CKLGD only requires polynomial total iterations to achieve $\varepsilon$-accuracy (Remark~\ref{rmk:total_iteration_complexity}).}
\label{fig: algo_illustrate}
\end{figure}
In a nutshell, our method can be interpreted as a hybridization of CKLGD and the ULA (without annealing), in contrast to MFLD which simply applies the ULA (with annealing) to minimize the reduced objective functional $\m F_{N, m}$. Figure~\ref{fig: algo_illustrate} illustrates the main difference between these two algorithms. The corresponding pseudocode is presented in Algorithm~\ref{alg: inexact_CKLGD}, which we refer as \emph{inexact} CKLGD due to the potential numerical errors that may arise during the inner loop sampling procedure.  
\begin{algorithm}[ht]
\caption{Inexact CKLGD for minimizing the reduced objective functional $\m F_{N, m}$}\label{alg: inexact_CKLGD}
\begin{algorithmic}
\Require observations $\{X_{t_j}^i: i\in[N], j\in[m]\}$; number of particles $B$; number of iterations $K$; number of iterations for sampling $\{n_{k}: k\in[K]\}$; a sequence of step sizes $\{\eta_k: k\in[K]\}$; a sequence of annealing coefficients $\{\alpha_k: k\in[K]\}$; a sequence of learning rate for sampling $\{h_k: k\in[K]\}$. 
\Ensure A set of particles $\{Y_{j, b}^K: j\in[m], b\in[B]\}$ followed the distribution $\wht\rho^K = (\wht\rho_1^K, \dots, \wht\rho_m^K)$.
\For{$j\gets 1$ to $m$}
  \State Uniformly sample $Y_{j, 1}^0, \dots, Y_{j, B}^0$ in $\m X$. \Comment{can also choose other initial distributions}
\EndFor
\For{$k \gets 1$ to $K$}
  \State Take $\wht\rho_j^{k-1} := \frac{1}{B}\sum_{b=1}^B\delta_{Y_{j, b}^{k-1}}$ and $\wht\rho^{k-1} = (\wht\rho_1^{k-1}, \dots, \wht\rho_m^{k-1})$.
  \For{$j\gets 1$ to $m$}
    \State Take $Z_{j, b}^0 = Y_{j, b}^{k-1}$ for all $b\in[B]$.
    \For{$s\gets 1$ to $n_k$} \Comment{ULA for inner sampling}
      \State $\!Z_{j, b}^{s} = Z_{j, b}^{s-1} - h_k\sum_{l=1}^k\big[\eta_l\prod_{l<l'\leq k}(1-\tau\eta_{l'})\big]\big[\nabla V_j(Z_{j, b}^{s-1}; \wht\rho^{l-1}) + 2\alpha_l Z_{j, b}^{s-1}\big] + \m N(0, 2h_k I_d)$, $\forall\,b\in[B]$. 
    \EndFor
    \State Take $Y_{j, b}^{k} = Z_{j, b}^{n_k}$ for all $b\in[B]$.
  \EndFor
\EndFor
\end{algorithmic}
\end{algorithm}


\subsection{Algorithmic convergence of inexact CKLGD}\label{sec: algo_converge}
We will first examine the convergence of the inexact CKLGD algorithm. The following result demonstrates how the step size, the coefficient of the quadratic terms in~\eqref{eqn: explicit_quad_anneal_update}, and the sampling error within each outer iteration influence the algorithmic convergence rate. 
\begin{theorem}[Convergence rate of inexact CKLGD]\label{thm: traj_CKLGD}
Assume that the step size of CKLGD $\{\eta_k\}_{k=1}^\infty$ and the coefficients of the extra quadratic terms $\{\alpha_k\}_{k=1}^\infty$
are positive and satisfy
\begin{itemize}
\item $\eta_k$ is decreasing to $0$ and $\sum_k\eta_k = \infty$;
\item $\lim_{k\to\infty}\alpha_k = \lim_{k\to\infty}\frac{\alpha_{k-1}-\alpha_k}{\eta_k\alpha_k} = 0$;
\item $\{\alpha_ke^{\tau(\eta_1 + \dots + \eta_k)}\}_{k=1}^\infty$ converges to $\infty$ and is increasing when $k$ is large enough.
\end{itemize}
Let $\{\delta_k\}_{k=1}^\infty$ be the tolerance of numerical error such that $\KL(\wht\rho^k\,\|\,\wt\rho^k) \leq \delta_k$. 
Then, we have
\begin{align*}
\min_{1\leq k\leq K}\m F_{N, m}(\wht\rho^k) - \m F_{N, m}(\rho) \lesssim \bigg[\sum_{k=1}^K\eta_{k+1}\bigg]^{-1}\bigg[\sum_{k=2}^K\frac{\eta_{k+1}(\alpha_k-\alpha_{k+1})}{\alpha_{k+1}} + \sum_{k=1}^K \eta_{k+1}^2 + \sum_{k=1}^K\eta_{k+1}\sqrt{\frac{\delta_k}{\alpha_k}} + \sum_{k=1}^{K+1}\alpha_k\eta_k\bigg],
\end{align*}
where the constant of $\lesssim$ does not depend on iteration number $K$.
\end{theorem}
\begin{rem}[Outer iteration complexity]
When we select $\eta_k = \alpha_k = k^{-\frac{1}{2}}$ and $\delta_k = k^{-\frac{3}{2}}$, we dereive the same convergence rate $O(\frac{\log K}{\sqrt{K}})$ as the one in CKLPD as demonstrated in Theorem~\ref{thm: algo_convergence_general}. This rate indicates that a carefully chosen coefficient of the quadratic terms $\{\alpha_k\}_{k=1}^\infty$ will not compromise the convergence rate when the numerical error $\KL(\wht\rho^k\,\|\,\wt\rho^k) \leq \delta_k$ is well controlled. 
\end{rem}

\begin{rem}[Total iteration complexity using ULA]\label{rmk:total_iteration_complexity}
By applying Theorem 2 from~\cite{vempala2019rapid}, we demonstrate that $O\big(\frac{1}{\alpha_k^2\delta_k}\log\frac{1}{\delta_k}\big)$ inner iterations of sampling are required in the $k$-th iteration by using ULA with the step size $h_k = O(\alpha_k\delta_k)$ to achieve $\delta_k$ accuracy (see Lemma~\ref{lem: rho_LSI} and Lemma~\ref{lem: L-smooth-rho} in Appendix). 
When we select $\eta_k=\alpha_k=k^{-\frac{1}{2}}$ and $\delta_k = k^{-\frac{3}{2}}$, the inner iteration complexity for approximating the solution in~\eqref{eqn: explicit_quad_anneal_update} via sampling is $O(k^{\frac{5}{2}}\log k)$ in the $k$-th iteration. Combining this inner iteration complexity with the outer iteration complexity reveals that at most $O\big(\frac{1}{\varepsilon^7}\log\frac{1}{\varepsilon}\big)$ number of total iterations are required to achieve $\varepsilon$-accuracy. This iteration complexity is significantly smaller than the $O(e^{\frac{C}{\varepsilon}})$ complexity by using the MFLD algorithm~\citep{chizat2022trajectory}.
\end{rem}

\begin{rem}[Choice of sampling algorithms]
We adpot ULA for sampling from a target distribution primarily mainly because of its simplicity. More efficient sampling algorithms, such as the Metropolis-adjusted Langevin algorithm~\citep{bou2013nonasymptotic}, could be employed to reduce the inner iteration complexity of sampling, potentially yielding a smaller total iteration complexity.
\end{rem}

\begin{rem}[Extra quadratic term]\label{rmk:extra_quad_term}
The extra quadratic terms introduced in the algorithm serve two purposes. First, these terms ensure that the target distribution $\wt\rho_j^k$ in the $k$-th iteration satisfies the log-Sobolev inequality, enabling sampling from $\wt\rho_j^k$ with $\delta_k$-accuracy using ULA within a polynomial number of iterations.
Second, the quadratic terms provide a lower bound to $H(\wt\rho^k)$, which is needed for controlling the difference $|H(\wht\rho^k) - H(\wt\rho^{k+1})|$ in our analysis (see Lemma~\ref{lem: diff_entropy} and its proof).
\end{rem}

\begin{rem}[Uniformly bounded first variation]
Due to the presence of the negative self-entropy term in the reduced objective functional $\m F_{N, m}$, its first variation $\frac{\delta\m F_{N,m}}{\delta\rho_j}$ cannot be uniformly bounded. Therefore, Theorem~\ref{thm: algo_convergence_general} for optimizing a generic jointly linearly convex functional is not directly applicable due to the violation of condition~\eqref{eqn:bounded_FV}. Our proof of Theorem~\ref{thm: traj_CKLGD} for the inexact CKLGD circumvents this condition by leveraging two key observations: (1) the first variation, excluding this entropy term, is uniformly bounded, and (2) the entropy term can be suitably controlled thanks to the additional quadratic term (see the previous remark). A sketch of the proof will be provided in Section~\ref{sec: sketch_algo}.
\end{rem}

\section{Simulation Results}\label{sec: numerical_results}
We demonstrate the advantages of using the inexact CKLGD algorithm to minimize the reduced objective functional $\m F_{N, m}$ defined in~\eqref{eqn: obj_func_traj} through a simulation study. Consider an SDE 
\begin{align}\label{eqn: SDE}
\dd Z_t = \nabla\Psi(t, Z_t)\,\dd t + \frac{1}{\sqrt{
2}}\,\dd W_t
\end{align}
evolving in the state space $\m X = \mb R^2$ with the potential function $\Psi(t, x) = 0.5(x_1-1.5)^2(x_1 + 1.5)^2 + 10(x_2+t)^2$. We assume the SDE evolves from $t = 0$ to $t = 1.25$. We select $m = 8$ time points with equal separation. At each time point $t_j$, $N = 64$ samples are uniformly drawn from the marginal distributions of the SDE at time $t_j$. The observations at $t_j$ consist of these samples with additional Gaussian noise of variance $\sigma^2 = 0.25$.

\begin{figure}[ht]
    \centering
    \includegraphics[width=0.75\linewidth]{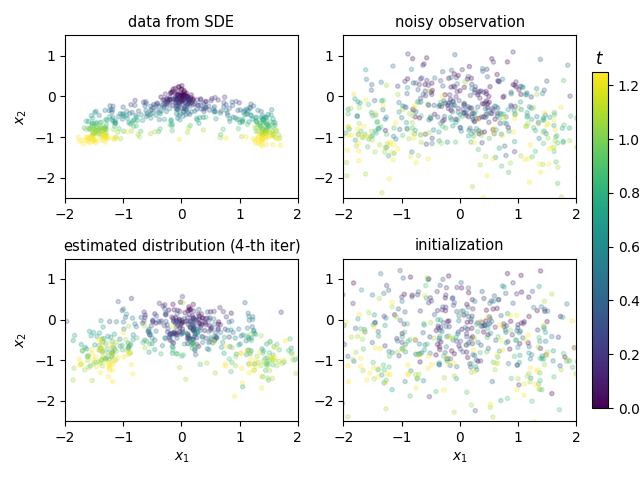}
    \caption{Scatter plot of the noiseless data generated from the SDE~\eqref{eqn: SDE} (upper left), noisy observations (upper right), initialization of the CKLGD algorithm and the baseline MFLD algorithm (lower right), and the estimated marginal distributions derived by applying the CKLGD algorithm (lower left).}
    \label{fig: scatter_plot}
\end{figure}
Figure~\ref{fig: scatter_plot} illustrates the scatter plots of data sampled from the SDE~\eqref{eqn: SDE}, the noisy observations, algorithm initialization, and the distribution estimated by the inexact CKLPD algorithm.
Specifically, the \textbf{upper left} figure displays all samples from the underlying SDE~\eqref{eqn: SDE} at different time points, starting with $Z_0\sim\m N(0, 0.01)$. Points with different colors represent samples from distinct time points. In the final time point $t_8 = 1.25$, the samples are distributed around two modes located at $(-1.5, -1.25)$ and $(1.5, -1.25)$. This bimodal phenomenon stems directly from the potential function $\Psi$, where these two points represent the function's minima at $t = 1.25$.
The \textbf{upper right} figure shows the noisy observations created by adding Gaussian noise with variance $\sigma^2 = 0.25$ to the data sampled from the SDE. As illustrated in the introduction, such Gaussian noise represents measurement uncertainty during data collection.
The \textbf{lower right} figure presents the initialization for both our inexact CKLGD algorithm and the mean-field Langevin dynamics (MFLD) algorithm proposed by~\citet{chizat2022trajectory}, serving as a comparative baseline for computing the E-NPMLE estimator defined through~\eqref{eqn: E-NPMLE} and~\eqref{eqn: obj_func_traj}. Both algorithms are initialized from the same group of particles generated by adding Gaussian noise to the noisy observations.
The \textbf{lower left} figure shows the estimated distribution at all time points after running the inexact CKLGD algorithm with 4 outer iterations, where each outer iteration comprises 500 inner sampling iterations to approximate the distribution $\wt\rho^k$ in~\eqref{eqn: explicit_quad_anneal_update}. As seen in the figure, the estimated distributions have the same pattern as the data derived from the SDE.

\begin{figure}[ht]
    \centering
    \includegraphics[width=0.8\linewidth]{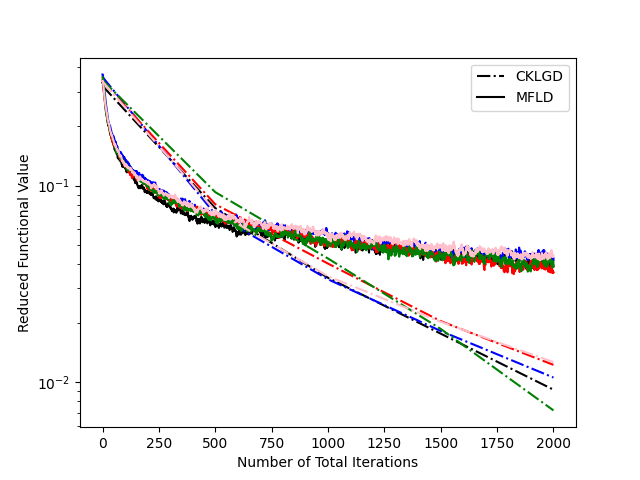}
    \caption{Reduced objective functional value $\m F_{N, m}(\rho) - \m F_{N,m}(\wht \rho)$ in the log scale versus the total number of iterations. The experiment is conducted five times independently with different observations and initializations. The MFLD algorithm exhibits a slower decay rate of reduced objective functional values (solid lines) compared to our CKLGD algorithm (dotted lines), which can be attributed to the presence of the annealing term.}
    \label{fig: CKLGDvsMFLD}
\end{figure}
Figure~\ref{fig: CKLGDvsMFLD} presents the loss of the reduced functional $\m F_{N, m}(\rho) - \m F_{N, m}(\wht\rho)$ versus the total number of iterations. Since the global minima $\wht\rho$ of $\m F_{N, m}$ is unknown, we use the solution obtained by running the inexact CKLGD algorithm for 8 outer iterations as a proxy for the global minima. We emphasize that we simply connect the loss with \emph{straight lines} at iterations 500, 1000, 1500, and 2000 in the CKLGD, corresponding to the 1st, 2nd, 3rd, and 4th outer iterations, and the loss between these points does \emph{not} reflect the true reduced functional values.
As illustrated in the figure, the loss decay rate of the MFLD algorithm (solid curves) decelerates after approximately 500 iterations due to the annealing term. We conducted five independent experiments with varying observations and initializations, with each color representing a distinct experiment.

\section{Sketch of Proofs}\label{sec: proof}
In this section, we summarize the main ideas of the proofs of the statistical convergence in Theorem~\ref{thm: stat_rate} and the algorithmic convergence of inexact CKLGD in Theorem~\ref{thm: traj_CKLGD}, and highlight the technical difficulties and contributions. Detailed proofs are provided in Appendix.

\subsection{Analysis of statistical convergence rate}
The proof is involved and can be decomposed into several steps. We summarize the main idea of the proof, while leaving the details of each step in the appendix.

\paragraph{Notation.}
We begin with introducing several useful notations. Recall that  $\m P(\Omega)$ denotes the set of all probability distributions over the path space $\Omega=\m C([0,T];\m X)$. For any $R\in\ms P(\Omega)$, define
\begin{align*}
g_R(t, x) := \sqrt{\frac{\m K_\sigma\ast R_t(x) + \m K_\sigma\ast R_t^\ast(x)}{2\m K_\sigma\ast R_t^\ast(x)}}.
\end{align*}
Let $C_\sigma > 0$ be a constant such that $C_\sigma^{-1} \leq \m K_\sigma R_t(x) \leq C_\sigma$ uniformly holds for all $t\in[0, 1]$ and $x\in\m X$. The existence of this $C_\sigma$ is guaranteed by~\citep[Proposition B.4,][]{lavenant2024toward}. Then, it is easy to check that
\begin{align*}
\frac{1}{\sqrt{2}} \leq g_{R}(t, x) \leq \sqrt{\frac{C_\sigma^2+1}{2}},\quad\forall\, (t, x)\in[0, 1]\times\m X.
\end{align*}
Also, note that for any $R$,
\begin{align*}
\sum_{j=1}^m\sum_{i=1}^N\frac{t_{j+1}-t_j}{N} \mb E\log g_{R}(t_j, X_{t_j}^i)
= -\sum_{j=1}^m (t_{j+1} - t_j)\KL\Big(\m K_\sigma\ast R_{t_j}^\ast\,\Big\|\,\frac{\m K_\sigma\ast R_{t_j}^\ast + \m K_\sigma\ast R_{t_j}}{2}\Big).
\end{align*}
For any function $f:[0, 1] \times \m X\to\mb R$, define the $L_m^2$-norm by
\begin{align*}
\|f\|_{L_m^2}^2 \coloneqq \sum_{j=1}^m (t_{j+1}-t_j)\|f(t_j, \cdot)\|_{L^2(\m K_\sigma\ast R_{t_j}^\ast)}^2.
\end{align*}
With this definition, we can apply the Hellinger-KL inequality~\citep[Equation 14.57b,][]{wainwright2019high} to obtain
\begin{align*}
\|g_{R} - g_{R^\ast}\|_{L_m^2}^2 = \sum_{j=1}^m(t_{j+1}-t_j)d_{\rm H}^2\Big(\m K_\sigma\ast R_{t_j}^\ast\,\Big\|\,\frac{\m K_\sigma\ast R_{t_j}^\ast + \m K_\sigma\ast R_{t_j}}{2}\Big)
\leq -\sum_{j=1}^m\sum_{i=1}^N\frac{t_{j+1}-t_j}{N} \mb E\log g_{R}(t_j, X_{t_j}^i).
\end{align*}
Furthermore, define the subset
\begin{align}\label{eqn:GR}
\m GR(r)\coloneqq \big\{R\in\ms P(\Omega): \|g_R - g_{R^\ast}\|_{L_m^2} \leq r
\,\,\mx{and}\,\,
\tau\KL(R\,\|\,W^\tau) \leq 2E\big\}.
\end{align}
Due to the fact that $\|g_R - g_{R'}\|_{L_m^2} \leq \sqrt{2}$ holds for all $R, R'\in\ms P(\Omega)$, we know $\m GR(\infty) = \m GR(\sqrt{2})$.

\paragraph{Proof of Theorem~\ref{thm: stat_rate}.}
By the optimality of $\wht R$, one can prove a modified basic inequality (see \emph{step 1} in Appendix~\ref{app: pf_stat_conv_discrete}):
\begin{align*}
-\sum_{j=1}^m\frac{t_{j+1}-t_j}{N}\sum_{i=1}^N\log g_{\wht R}(t_j, X_{t_j}^i)
\leq \frac{\lambda\tau}{4}\big[\KL(R^\ast\,\|\,W^\tau) - \KL(\wht R\,\|\,W^\tau)\big].
\end{align*}
We expect that the left-hand side of the above inequality is close to its expected value when $N$ and $m$ are large enough. This idea can be rigorously summarized by a uniform laws of large number in the following lemma. A proof of this lemma is deferred to Appendix~\ref{app: pf_of_emp_proc}. We highlight that the proof is highly nontrivial, with additional discussion provided at the end of this subsection.

\begin{lemma}\label{lem: tail bound}
Let $C_{\rm HP} := 12 + 34.5\log\frac{C_\sigma^2+1}{2}$ and define the event
\begin{align*}
\ms A \coloneqq\bigg\{\sup_{R\in\m GR(\infty)} \frac{\big|\sum_{j=1}^m\sum_{i=1}^N\frac{t_{j+1}-t_j}{N}[\log g_R(t_j, X_{t_j}^i) - \mb E\log g_R(t_j, X_{t_j}^i)]\big|}{\delta_{N, m} + \|g_R - g_{R^\ast}\|_{L_m^2}} \leq C_{\rm HP}\delta_{N, m}\bigg\}.
\end{align*}
Then, we have
\begin{align*}
\mb P(\ms A) \geq 1 - 2e^{-\frac{N\delta_{N, m}^2}{2\Delta_m}}.
\end{align*}
\end{lemma}
Now, let us prove the result with Lemma~\ref{lem: tail bound} by considering two different cases (see \emph{step 2} in Appendix~\ref{app: pf_stat_conv_discrete}).

\medskip
\noindent {\bf Case 1:} When $\tau\KL(\wht R\,\|\,W^\tau) \leq 2\tau\KL(R^\ast\,\|\,W^\tau)$, we have $\wht R\in\m GR(\infty)$ and therefore
\begin{align*}
C_{\rm HP}\delta_{N, m}\big(\delta_{N, m} + \|g_{\wht R} - g_{R^\ast}\|_{L_m^2}\big)
&\geq \sum_{j=1}^m \frac{t_{j+1}-t_j}{N}\sum_{i=1}^N\big[\log g_{\wht R}(t_j, X_{t_j}^i) - \mb E\log g_{\wht R}(t_j, X_{t_j}^i)\big]\\
&\geq -\frac{\lambda\tau}{4}\KL(R^\ast\,\|\,W^\tau) + \|g_{\wht R} - g_{R^\ast}\|_{L_m^2}^2.
\end{align*}
Using the facts $\lambda_{N, m} = C_\lambda\delta_{N, m}^2$ and $\tau\KL(R^\ast\,\|\,W^\tau) \leq E$, the above inequality implies
\begin{align*}
\|g_{\wht R} - g_{R^\ast}\|_{L_m^2} \leq \bigg(C_{\rm HP} + \sqrt{C_{\rm HP}} + \frac{\sqrt{C_\lambda E}}{2}\bigg)\delta_{N, m}.
\end{align*}

\medskip
\noindent {\bf Case 2:}
When $\tau\KL(\wht R\,\|\,W^\tau) > 2\tau\KL(R^\ast\,\|\,W^\tau)$, by taking $\varepsilon = \frac{\tau\KL(\wht R\,\|\,W^\tau) - \frac{3}{2}\tau\KL( R^\ast\,\|\,W^\tau)}{\tau\KL(\wht R\,\|\,W^\tau) - \tau\KL(R^\ast\,\|\,W^\tau)}\in(0, 1)$ and letting $\wt R = (1-\varepsilon)\wht R + \varepsilon R^\ast \in\ms P(\Omega)$, one can show that
\begin{align*}
\KL(\wt R\,\|\,W^\tau)
\leq \frac{3}{2}\KL(R^\ast\,\|\,W^\tau)
\quad\mx{and}\quad
\sum_{j=1}^m \sum_{i=1}^N\frac{t_{j+1}-t_j}{N}\log g_{\wt R}(t_j, X_{t_j}^i) 
\geq \frac{\lambda}{8}\tau\KL(R^\ast\,\|\,W^\tau).
\end{align*}
Therefore, $\wt R\in\m GR(\infty)$ (recall the definition of $\m GR$ in equation~\eqref{eqn:GR}) and we have
\begin{align*}
C_{\rm HP}\delta_{N, m}\big(\delta_{N, m} + \|g_{\wt R} - g_{R^\ast}\|_{L_m^2}\big)
&\geq \sum_{j=1}^m \frac{t_{j+1}-t_j}{N}\sum_{i=1}^N\big[\log g_{\wt R}(t_j, X_{t_j}^i) - \mb E\log g_{\wt R}(t_j, X_{t_j}^i)\big]\\
&\geq \frac{\lambda}{8}\tau\KL(R^\ast\,\|\,W^\tau) + \|g_{\wt R} - g_{R^\ast}\|_{L_m^2}^2\\
&\geq \frac{C_\lambda E^{-1}}{8}\delta_{N, m}^2 + \|g_{\wt R} - g_{R^\ast}\|_{L_m^2}^2.
\end{align*}
However, if a sufficiently large constant $C_\lambda$ is chosen so that $C_\lambda > (8C_{\rm HP} + 2C_{\rm HP}^2)E$, then one can obtain by using the AM–GM inequality that
\begin{align*}
\frac{C_\lambda E^{-1}}{8}\delta_{N,m}^2 + \|g_{\wt R} - g_{R^\ast}\|_{L_m^2}^2 
> C_{\rm HP}\delta_{N, m}\big(\delta_{N, m} + \|g_{\wt R} - g_{R^\ast}\|_{L_m^2}\big),
\end{align*}
which is a contradiction.
Therefore, this second case of $\tau\KL(\wht R\,\|\,W^\tau) > 2\tau\KL(R^\ast\,\|\,W^\tau)$ cannot hold under event $\ms A$ under the condition of the theorem.

To summarize, we have shown that
\begin{align*}
\mb P\bigg(\|g_{\wht R} - g_{R^\ast}\|_{L_m^2} \leq \Big(C_{\rm HP} + \sqrt{C_{\rm HP}} + \frac{\sqrt{C_\lambda E}}{2}\Big)\delta_{N, m}\bigg)
\geq \mb P(\ms A) \geq 1 - 2e^{-\frac{N\delta_{N,m}^2}{2\Delta_m}}.
\end{align*}
Finally, our desired bound~\eqref{eqn: finite_marginal_converge} follows from applying Lemma~\ref{lem: Hellinger} in Appendix~\ref{sec: tech_lem}. 

\paragraph{Discussion of Lemma~\ref{lem: tail bound}.}
We highlight several key techniques used in the proof of the finite-sample uniform law of large numbers in Lemma~\ref{lem: tail bound}.

Firstly, in most existing literature, the desired function class is often assumed to be ``star-shaped", meaning that if both a function and the ground-truth function belong to this function class, their convex combination also belongs to the same function class. This assumption directly implies that $\mb ES_{N,m}(r) / r$ is non-increasing. However, we note that this assumption does not hold in our problem.  In general, there exists no $R'\in\ms P(\Omega)$ such that $g_{R'} = \frac{g_{R} + g_{R^\ast}}{2}$, as an additional scaling factor is required. To address this issue, we first upper bound $\mb ES_{N, m}(r)$ with respect to $r$ using a chaining argument~\citep{geer2000empirical, wainwright2019high} from empirical process theory and then demonstrate that this upper bound, when divided by $r$, is non-increasing.

Secondly, when using the chaining technique to control $\mb ES_{N, m}(r)$, traditional approaches typically condition on all samples, resulting in a sub-Gaussian conditioned empirical process due to the Rademacher random variables. However, this standard approach leads to convergence of the estimator with respect to a sample-based norm, requiring additional analysis to establish its equivalence to a sample-independent norm. Instead, we find that the empirical process in our context is sub-exponential and therefore choose to apply the chaining technique simultaneously with respect to both $L_m^2$-norm and $L_m^\infty$-norm, following the approaches of~\citet{baraud2010bernstein} and~\citet{yao2022mean}. This method effectively captures the local sub-Gaussian behavior of sub-exponential random variables, leading to a sharper convergence rate.

Lastly, in order to derive the phase transition phenomenon in Theorem~\ref{thm: stat_rate}, a careful estimation of the covering number for the involved function class is required. For our specific context, where the function $R$ depends on both spatial and temporal inputs, a key insight we utilized is that considering the covering for both the state space $\m X$ and time space $[0, 1]$ becomes advantageous only when there are sufficiently many time points.  Therefore, we employ two distinct approaches to control the covering number. For example, in cases with limited time points, such as when $m=1$, it is more effective to focus on covering only the state space and then take the union across all time points. For a detailed analysis of the covering number, we refer readers to Proposition~\ref{prop: covering_num} and its proof.

\subsection{Analysis of algorithmic convergence rate}\label{sec: sketch_algo}

Next, we provide a proof sketch for Theorem~\ref{thm: traj_CKLGD}. At each iteration $k$, the only available information about $\wht\rho^k$ is that $\KL(\wht\rho^k\,\|\,\wt\rho^k) \leq \delta_k$.  Consequently, we express the difference in the reduced objective functional as
\begin{align*}
\m F_{N, m}(\wht\rho^k) - \m F_{N, m}(\rho)
= \big[\m F_{N, m}(\wht\rho^k) - \m F_{N, m}(\wt\rho^k)\big] + \big[\m F_{N, m}(\wt\rho^k) - \m F_{N, m}(\rho)\big].
\end{align*}
Here, the two terms correspond to the approximation error and the optimization error, respectively. Throughout the proof, we use the shorthand notation of $H(\rho) = \int\rho\log\rho$.

\paragraph{Control the approximation error.}
We develop a novel technique to control the approximation error. Notably, directly applying the joint linear convexity of 
$\m F_{N, m}$ results in a term $\KL(\wt\rho^k\,\|\,\wht\rho^k)$ in the upper bound, which cannot be directly controlled. Some existing works address this issue by either adding an additional regularization term to the objective functional~\citep{nitanda2021particle, oko2022particle} or adopting alternative measures of numerical error that are not suitable for our context~\citep{yao2024wasserstein, cheng2024convergence}. In this work, we directly tackle the approximation error without introducing any extra regularization term or switching to a different error measure, by employing a divide-and-conquer approach.
Specifically, a key innovation in our proof is the application of the joint linear convexity of  $\m F_{N, m}$ along an interpolation between $\wt\rho^k$ and $\wht\rho^k$. Specifically, we construct a sequence of probability distributions $\mu_0 ,\dots, \mu_{r+1}$ where $\mu_{0} = \wht\rho^k$ and $\mu_{r+1} = \wt\rho^k$, with $r\in\mb Z_+$ as a positive integer to be determined later. By leveraging the convexity of $\m F_{N, m}$ along these interpolations, we transform the error term $\KL(\wt\rho^k\,\|\,\wht\rho^k)$ into the summation $\sum_{s=0}^r\KL(\mu_{s+1}\,\|\,\mu_s)$. The following result establishes control over this sum of KL divergences after taking the supremum over $r\in\mb Z_+$ and the interpolations $\mu_0, \dots, \mu_{r+1}$. A proof of the lemma is provided in Appendix~\ref{sec: tech_lem}.

\begin{lemma}\label{prop: FR_dist}
For any $\rho, \rho'\in\ms P^r(\m X)$, we have
\begin{enumerate}
\item[(1)] $d(\rho, \rho') := \arccos(\int_{\m X}\sqrt{\rho\rho'}\,\dd x)$ is a distance on $\ms P^r(\m X)$ and satisfies $d(\rho, \rho') \leq \sqrt{\KL(\rho\,\|\,\rho')}$;

\item[(2)] If the density functions of $\rho$ and $\rho'$ are positive and continuous, then
\begin{align*}
\inf_{r, \mu_0, \dots, \mu_{r+1}}\Big\{\sum_{s=0}^r \KL(\mu_{s+1}\,\|\,\mu_{s}): \mu_{r+1} = \rho', \mu_{0}=\rho, \mu_s\in\ms P^r(\m X)\Big\} = 0;
\end{align*}

\item[(3)] With the same assumptions as in (2), we have
\begin{align*}
\inf_{r, \mu_0, \dots, \mu_{r+1}}\Big\{\sum_{s=0}^r \sqrt{\KL(\mu_{s+1}\,\|\,\mu_{s})}: \mu_{r+1} = \rho', \mu_{0}=\rho, \mu_s\in\ms P^r(\m X)\Big\} \leq \sqrt{2}d(\rho, \rho').
\end{align*}
\end{enumerate}
\end{lemma}
With the above lemma, one can show that
\begin{align*}
\m F_{N, m}(\wht\rho^k) - \m F_{N,m}(\wt\rho^k) \leq 2\sqrt{B_1^2 + \cdots + B_m^2}\cdot\delta_k^{\frac{1}{2}} + \tau\big[H(\wht\rho^k) - H(\wt\rho^k)\big],
\end{align*}
where $B_1, \dots, B_m \geq 0$ are some universal constants defined in Lemma~\ref{lem: uniform_bound_Vj}. We refer to \emph{step 1} in the proof provided in Appendix~\ref{app: pf_algo_conv} for more details.

\paragraph{Control the optimization error.}
We can directly apply the convexity of $\m F_{N, m}$ to derive
\begin{align*}
\m F_{N, m}(\wt\rho^k) - \m F_{N, m}(\rho)
&\leq \sum_{j=1}^m \int V_j(y_j; \wt\rho^k) + \tau\log\wt\rho_j^k(y_j)\,\dd[\wt\rho_j^k - \rho_j].
\end{align*}
By adopting a stability argument (Lemma~\ref{lem: stable_Vj}), one can show that the integration of $V_j(\cdot; \wt\rho^k)$ is closed to the integration of $V_j(\cdot; \wht\rho_j^k)$. The key is to control the integration of $\log\wt\rho_j^k$, which is also the main difference from the proof of Theorem~\ref{thm: algo_convergence_general}, which relies on the additional condition~\eqref{eqn:bounded_FV}. Note that we may use the definition~\eqref{eqn: explicit_quad_anneal_update} to rewrite $\wt\rho^k$ as the minimization of a functional $U_k$ defined by
\begin{align}\label{eqn: def_Uk}
U_k(\rho) \coloneqq \sum_{j=1}^m\int_{\m X}\sum_{l=1}^k\Big[\eta_l\prod_{l < l'\leq k}(1-\tau\eta_{l'})\Big]\big[V_j(y_j; \wht\rho^{l-1}) + \alpha_l\|y_j\|^2\big]\,\dd\rho_j + H(\rho).
\end{align}
This definition establishes a connection between $\log\wt\rho_j^k$ and $U_k(\wt\rho_j^k)$, where $U_k$ follows a recursive definition
\begin{align*}
U_k(\rho)
&= (1-\tau\eta_k)U_{k-1}(\rho) + \tau\eta_k H(\rho) + \eta_k\sum_{j=1}^m\int_{\m X}\big[V_j(y_j;\wht\rho^{k-1}) + \alpha_k\|y_j\|^2\big]\,\dd\rho_j.
\end{align*}
With these facts, we can provide an upper bound for $\sum_{k=1}^K\eta_{k+1}\big[\m F_{N, m}(\wht\rho^k) - \m F_{N,m}(\rho)\big]$. Since the exact form of this upper bound is somewhat complex, we refer to \emph{Step 2} in the proof provided in Appendix~\ref{app: pf_algo_conv} for further details.

\paragraph{Derive the algorithmic convergence rate.} Finally, in \emph{step 3} of the proof provided in Appendix~\ref{app: pf_algo_conv}, we can combine the upper bound of two terms and derive the claimed result presented in Theorem~\ref{thm: traj_CKLGD}.

\section{Summary}\label{sec: summary}
In this paper, we studied the problem of estimating the probability density evolution for a stochastic process using noisy snapshot data. Our focus is on analyzing both statistical and computational aspects of the proposed E-NPMLE.

Statistically, we conduct a non-asymptotic analysis of the estimator and reveal a phase transition phenomenon that depends on the snapshot/sample frequency. Our result demonstrates the importance of balancing the number of snapshots versus the sample size per snapshot given a fixed total sample size budget constraint. We believe that these findings provide valuable guidance for experimental design in real-world applications for learning dynamical structures from static distributional data.

Computationally, we introduced a novel CKLGD algorithm, derived from an explicit discretization of the KL divergence gradient flow. We demonstrate that this algorithm achieves a polynomial convergence rate. In the algorithmic convergence analysis, we develop a new technique for analyzing the impact of the KL-type sampling error by interpolating two distributions along the Fisher--Rao geodesics. This approach also establishes a connection between Fisher--Rao distance and KL divergence through a variational approach, which may be of independent interest in probability theory.

\bibliographystyle{plainnat}
\bibliography{ref}

\begin{thebibliography}{68}
\providecommand{\natexlab}[1]{#1}
\providecommand{\url}[1]{\texttt{#1}}
\expandafter\ifx\csname urlstyle\endcsname\relax
  \providecommand{\doi}[1]{doi: #1}\else
  \providecommand{\doi}{doi: \begingroup \urlstyle{rm}\Url}\fi

\bibitem[Aragam and Yang(2024)]{aragam2024model}
Bryon Aragam and Ruiyi Yang.
\newblock Model-free estimation of latent structure via multiscale nonparametric maximum likelihood.
\newblock \emph{arXiv preprint arXiv:2410.22248}, 2024.

\bibitem[Aubin-Frankowski et~al.(2022)Aubin-Frankowski, Korba, and L{\'e}ger]{aubin2022mirror}
Pierre-Cyril Aubin-Frankowski, Anna Korba, and Flavien L{\'e}ger.
\newblock {Mirror descent with relative smoothness in measure spaces, with application to Sinkhorn and EM}.
\newblock \emph{Advances in Neural Information Processing Systems}, 35:\penalty0 17263--17275, 2022.

\bibitem[Baraud(2010)]{baraud2010bernstein}
Yannick Baraud.
\newblock {A Bernstein-type inequality for suprema of random processes with applications to model selection in non-Gaussian regression}.
\newblock \emph{Bernoulli}, pages 1064--1085, 2010.

\bibitem[Berg(1976)]{berg1976potential}
Christian Berg.
\newblock Potential theory on the infinite dimensional torus.
\newblock \emph{Inventiones mathematicae}, 32\penalty0 (1):\penalty0 49--100, 1976.

\bibitem[Both and Kusters(2019)]{both2019temporal}
Gert-Jan Both and Remy Kusters.
\newblock Temporal normalizing flows.
\newblock \emph{arXiv preprint arXiv:1912.09092}, 2019.

\bibitem[Botvinick-Greenhouse et~al.(2023)Botvinick-Greenhouse, Yang, and Maulik]{botvinick2023generative}
Jonah Botvinick-Greenhouse, Yunan Yang, and Romit Maulik.
\newblock Generative modeling of time-dependent densities via optimal transport and projection pursuit.
\newblock \emph{Chaos: An Interdisciplinary Journal of Nonlinear Science}, 33\penalty0 (10), 2023.

\bibitem[Bou-Rabee and Hairer(2013)]{bou2013nonasymptotic}
Nawaf Bou-Rabee and Martin Hairer.
\newblock {Nonasymptotic mixing of the MALA algorithm}.
\newblock \emph{IMA Journal of Numerical Analysis}, 33\penalty0 (1):\penalty0 80--110, 2013.

\bibitem[Carlier et~al.(2022)Carlier, Chizat, and Laborde]{carlier2022lipschitz}
Guillaume Carlier, L{\'e}na{\"\i}c Chizat, and Maxime Laborde.
\newblock {Lipschitz continuity of the Schr{\"o}dinger map in entropic optimal transport}.
\newblock 2022.

\bibitem[Chen et~al.(2021{\natexlab{a}})Chen, Yang, Duan, and Karniadakis]{chen2021solving}
Xiaoli Chen, Liu Yang, Jinqiao Duan, and George~Em Karniadakis.
\newblock {Solving Inverse Stochastic Problems from Discrete Particle Observations Using the Fokker--Planck Equation and Physics-Informed Neural Networks}.
\newblock \emph{SIAM Journal on Scientific Computing}, 43\penalty0 (3):\penalty0 B811--B830, 2021{\natexlab{a}}.

\bibitem[Chen et~al.(2018)Chen, Conforti, and Georgiou]{chen2018measure}
Yongxin Chen, Giovanni Conforti, and Tryphon~T Georgiou.
\newblock Measure-valued spline curves: An optimal transport viewpoint.
\newblock \emph{SIAM Journal on Mathematical Analysis}, 50\penalty0 (6):\penalty0 5947--5968, 2018.

\bibitem[Chen et~al.(2021{\natexlab{b}})Chen, Georgiou, and Pavon]{chen2021stochastic}
Yongxin Chen, Tryphon~T Georgiou, and Michele Pavon.
\newblock {Stochastic control liaisons: Richard Sinkhorn meets Gaspard Monge on a Schr\"{o}dinger bridge}.
\newblock \emph{Siam Review}, 63\penalty0 (2):\penalty0 249--313, 2021{\natexlab{b}}.

\bibitem[Cheng et~al.(2024)Cheng, Lu, Tan, and Xie]{cheng2024convergence}
Xiuyuan Cheng, Jianfeng Lu, Yixin Tan, and Yao Xie.
\newblock {Convergence of flow-based generative models via proximal gradient descent in Wasserstein space}.
\newblock \emph{IEEE Transactions on Information Theory}, 2024.

\bibitem[Chewi et~al.(2021)Chewi, Clancy, Le~Gouic, Rigollet, Stepaniants, and Stromme]{chewi2021fast}
Sinho Chewi, Julien Clancy, Thibaut Le~Gouic, Philippe Rigollet, George Stepaniants, and Austin Stromme.
\newblock {Fast and smooth interpolation on Wasserstein space}.
\newblock In \emph{International Conference on Artificial Intelligence and Statistics}, pages 3061--3069. PMLR, 2021.

\bibitem[Chewi et~al.(2024)Chewi, Erdogdu, Li, Shen, and Zhang]{chewi2024analysis}
Sinho Chewi, Murat~A Erdogdu, Mufan Li, Ruoqi Shen, and Matthew~S Zhang.
\newblock {Analysis of Langevin Monte Carlo from Poincar\'e to log-Sobolev}.
\newblock \emph{Foundations of Computational Mathematics}, pages 1--51, 2024.

\bibitem[Chizat(2022{\natexlab{a}})]{chizat2022convergence}
L{\'e}na{\"\i}c Chizat.
\newblock Convergence rates of gradient methods for convex optimization in the space of measures.
\newblock \emph{Open Journal of Mathematical Optimization}, 3:\penalty0 1--19, 2022{\natexlab{a}}.

\bibitem[Chizat(2022{\natexlab{b}})]{chizat2022mean}
L{\'e}na{\"\i}c Chizat.
\newblock {Mean-Field Langevin Dynamics: Exponential Convergence and Annealing}.
\newblock \emph{Transactions on Machine Learning Research}, 2022{\natexlab{b}}.

\bibitem[Chizat et~al.(2022)Chizat, Zhang, Heitz, and Schiebinger]{chizat2022trajectory}
L{\'e}na{\"\i}c Chizat, Stephen Zhang, Matthieu Heitz, and Geoffrey Schiebinger.
\newblock {Trajectory inference via mean-field Langevin in path space}.
\newblock \emph{Advances in Neural Information Processing Systems}, 35:\penalty0 16731--16742, 2022.

\bibitem[Cuturi(2013)]{cuturi2013sinkhorn}
Marco Cuturi.
\newblock Sinkhorn distances: Lightspeed computation of optimal transport.
\newblock \emph{Advances in neural information processing systems}, 26, 2013.

\bibitem[Dalalyan(2017)]{dalalyan2017theoretical}
Arnak~S Dalalyan.
\newblock Theoretical guarantees for approximate sampling from smooth and log-concave densities.
\newblock \emph{Journal of the Royal Statistical Society Series B: Statistical Methodology}, 79\penalty0 (3):\penalty0 651--676, 2017.

\bibitem[Fikhtengol'ts(2014)]{fikhtengol2014fundamentals}
Grigorii~Mikhailovich Fikhtengol'ts.
\newblock \emph{{The Fundamentals of Mathematical Analysis}}.
\newblock Elsevier, 2014.

\bibitem[Holbrook et~al.(2020)Holbrook, Lan, Streets, and Shahbaba]{holbrook2020nonparametric}
Andrew Holbrook, Shiwei Lan, Jeffrey Streets, and Babak Shahbaba.
\newblock {Nonparametric Fisher geometry with application to density estimation}.
\newblock In \emph{Conference on Uncertainty in Artificial Intelligence}, pages 101--110. PMLR, 2020.

\bibitem[Holley and Stroock(1986)]{holley1986logarithmic}
Richard Holley and Daniel~W Stroock.
\newblock {Logarithmic Sobolev inequalities and stochastic Ising models}.
\newblock 1986.

\bibitem[Karimi et~al.(2016)Karimi, Nutini, and Schmidt]{karimi2016linear}
Hamed Karimi, Julie Nutini, and Mark Schmidt.
\newblock {Linear convergence of gradient and proximal-gradient methods under the Polyak-{\L}ojasiewicz condition}.
\newblock In \emph{Machine Learning and Knowledge Discovery in Databases: European Conference, ECML PKDD 2016, Riva del Garda, Italy, September 19-23, 2016, Proceedings, Part I 16}, pages 795--811. Springer, 2016.

\bibitem[Kiefer and Wolfowitz(1956)]{kiefer1956consistency}
Jack Kiefer and Jacob Wolfowitz.
\newblock Consistency of the maximum likelihood estimator in the presence of infinitely many incidental parameters.
\newblock \emph{The Annals of Mathematical Statistics}, pages 887--906, 1956.

\bibitem[Klein et~al.(2015)Klein, Mazutis, Akartuna, Tallapragada, Veres, Li, Peshkin, Weitz, and Kirschner]{klein2015droplet}
Allon~M Klein, Linas Mazutis, Ilke Akartuna, Naren Tallapragada, Adrian Veres, Victor Li, Leonid Peshkin, David~A Weitz, and Marc~W Kirschner.
\newblock Droplet barcoding for single-cell transcriptomics applied to embryonic stem cells.
\newblock \emph{Cell}, 161\penalty0 (5):\penalty0 1187--1201, 2015.

\bibitem[Koenker and Mizera(2014)]{koenker2014convex}
Roger Koenker and Ivan Mizera.
\newblock {Convex optimization, shape constraints, compound decisions, and empirical Bayes rules}.
\newblock \emph{Journal of the American Statistical Association}, 109\penalty0 (506):\penalty0 674--685, 2014.

\bibitem[Kosorok(2008)]{kosorok2008introduction}
Michael~R Kosorok.
\newblock \emph{{Introduction to Empirical Processes and Semiparametric Inference}}, volume~61.
\newblock Springer, 2008.

\bibitem[Kullback(1968)]{kullback1968probability}
Solomon Kullback.
\newblock Probability densities with given marginals.
\newblock \emph{The Annals of Mathematical Statistics}, 39\penalty0 (4):\penalty0 1236--1243, 1968.

\bibitem[Lan(2020)]{lan2020first}
Guanghui Lan.
\newblock \emph{{First-order and Stochastic Optimization Methods for Machine Learning}}, volume~1.
\newblock Springer, 2020.

\bibitem[Lavenant et~al.(2024)Lavenant, Zhang, Kim, Schiebinger, et~al.]{lavenant2024toward}
Hugo Lavenant, Stephen Zhang, Young-Heon Kim, Geoffrey Schiebinger, et~al.
\newblock Toward a mathematical theory of trajectory inference.
\newblock \emph{The Annals of Applied Probability}, 34\penalty0 (1A):\penalty0 428--500, 2024.

\bibitem[Ledoux and Talagrand(2013)]{ledoux2013probability}
Michel Ledoux and Michel Talagrand.
\newblock \emph{Probability in Banach Spaces: Isoperimetry and Processes}.
\newblock Springer Science \& Business Media, 2013.

\bibitem[L{\'e}onard(2014)]{leonard2014survey}
Christian L{\'e}onard.
\newblock {A survey of the Schr{\"o}dinger problem and some of its connections with optimal transport}.
\newblock \emph{Discrete and Continuous Dynamical Systems-Series A}, 34\penalty0 (4):\penalty0 1533--1574, 2014.

\bibitem[Leskovec et~al.(2007)Leskovec, Kleinberg, and Faloutsos]{leskovec2007graph}
Jure Leskovec, Jon Kleinberg, and Christos Faloutsos.
\newblock Graph evolution: Densification and shrinking diameters.
\newblock \emph{ACM transactions on Knowledge Discovery from Data (TKDD)}, 1\penalty0 (1):\penalty0 2--es, 2007.

\bibitem[Lu et~al.(2022)Lu, Maulik, Gao, Dietrich, Kevrekidis, and Duan]{lu2022learning}
Yubin Lu, Romit Maulik, Ting Gao, Felix Dietrich, Ioannis~G Kevrekidis, and Jinqiao Duan.
\newblock Learning the temporal evolution of multivariate densities via normalizing flows.
\newblock \emph{Chaos: An Interdisciplinary Journal of Nonlinear Science}, 32\penalty0 (3), 2022.

\bibitem[Mackey and Tyran-Kami{\'n}ska(2021)]{mackey2021can}
Michael~C Mackey and Marta Tyran-Kami{\'n}ska.
\newblock How can we describe density evolution under delayed dynamics?
\newblock \emph{Chaos: An Interdisciplinary Journal of Nonlinear Science}, 31\penalty0 (4), 2021.

\bibitem[Macosko et~al.(2015)Macosko, Basu, Satija, Nemesh, Shekhar, Goldman, Tirosh, Bialas, Kamitaki, Martersteck, et~al.]{macosko2015highly}
Evan~Z Macosko, Anindita Basu, Rahul Satija, James Nemesh, Karthik Shekhar, Melissa Goldman, Itay Tirosh, Allison~R Bialas, Nolan Kamitaki, Emily~M Martersteck, et~al.
\newblock Highly parallel genome-wide expression profiling of individual cells using nanoliter droplets.
\newblock \emph{Cell}, 161\penalty0 (5):\penalty0 1202--1214, 2015.

\bibitem[Massart(2000)]{massart2000constants}
Pascal Massart.
\newblock {About the constants in Talagrand's concentration inequalities for empirical processes}.
\newblock \emph{The Annals of Probability}, 28\penalty0 (2):\penalty0 863--884, 2000.

\bibitem[Neklyudov et~al.(2023)Neklyudov, Brekelmans, Severo, and Makhzani]{neklyudov2023action}
Kirill Neklyudov, Rob Brekelmans, Daniel Severo, and Alireza Makhzani.
\newblock Action matching: Learning stochastic dynamics from samples.
\newblock In \emph{International conference on machine learning}, pages 25858--25889. PMLR, 2023.

\bibitem[Nitanda et~al.(2021)Nitanda, Wu, and Suzuki]{nitanda2021particle}
Atsushi Nitanda, Denny Wu, and Taiji Suzuki.
\newblock Particle dual averaging: Optimization of mean field neural network with global convergence rate analysis.
\newblock \emph{Advances in Neural Information Processing Systems}, 34:\penalty0 19608--19621, 2021.

\bibitem[Nitanda et~al.(2022)Nitanda, Wu, and Suzuki]{nitanda2022convex}
Atsushi Nitanda, Denny Wu, and Taiji Suzuki.
\newblock {Convex analysis of the mean field Langevin dynamics}.
\newblock In \emph{International Conference on Artificial Intelligence and Statistics}, pages 9741--9757. PMLR, 2022.

\bibitem[Oko et~al.(2022)Oko, Suzuki, Nitanda, and Wu]{oko2022particle}
Kazusato Oko, Taiji Suzuki, Atsushi Nitanda, and Denny Wu.
\newblock Particle stochastic dual coordinate ascent: Exponential convergent algorithm for mean field neural network optimization.
\newblock In \emph{International Conference on Learning Representations}, 2022.

\bibitem[Ollivier et~al.(2014)Ollivier, Pajot, and Villani]{ollivier2014optimal}
Yann Ollivier, Herv{\'e} Pajot, and C{\'e}dric Villani.
\newblock \emph{{Optimal Transport: Theory and Applications}}, volume 413.
\newblock Cambridge University Press, 2014.

\bibitem[Otto and Villani(2000)]{otto2000generalization}
Felix Otto and C{\'e}dric Villani.
\newblock {Generalization of an inequality by Talagrand and links with the logarithmic Sobolev inequality}.
\newblock \emph{Journal of Functional Analysis}, 173\penalty0 (2):\penalty0 361--400, 2000.

\bibitem[Polyanskiy and Wu(2020)]{polyanskiy2020self}
Yury Polyanskiy and Yihong Wu.
\newblock Self-regularizing property of nonparametric maximum likelihood estimator in mixture models.
\newblock \emph{arXiv preprint arXiv:2008.08244}, 2020.

\bibitem[Ruschendorf(1995)]{ruschendorf1995convergence}
Ludger Ruschendorf.
\newblock Convergence of the iterative proportional fitting procedure.
\newblock \emph{The Annals of Statistics}, pages 1160--1174, 1995.

\bibitem[Saha and Guntuboyina(2020)]{saha2020nonparametric}
Sujayam Saha and Adityanand Guntuboyina.
\newblock {On the nonparametric maximum likelihood estimator for Gaussian location mixture densities with application to Gaussian denoising}.
\newblock \emph{The Annals of Statistics}, 48\penalty0 (2):\penalty0 738--762, 2020.

\bibitem[Schiebinger et~al.(2019)Schiebinger, Shu, Tabaka, Cleary, Subramanian, Solomon, Gould, Liu, Lin, Berube, et~al.]{schiebinger2019optimal}
Geoffrey Schiebinger, Jian Shu, Marcin Tabaka, Brian Cleary, Vidya Subramanian, Aryeh Solomon, Joshua Gould, Siyan Liu, Stacie Lin, Peter Berube, et~al.
\newblock Optimal-transport analysis of single-cell gene expression identifies developmental trajectories in reprogramming.
\newblock \emph{Cell}, 176\penalty0 (4):\penalty0 928--943, 2019.

\bibitem[Sha et~al.(2024)Sha, Qiu, Zhou, and Nie]{sha2024reconstructing}
Yutong Sha, Yuchi Qiu, Peijie Zhou, and Qing Nie.
\newblock Reconstructing growth and dynamic trajectories from single-cell transcriptomics data.
\newblock \emph{Nature Machine Intelligence}, 6\penalty0 (1):\penalty0 25--39, 2024.

\bibitem[Shen et~al.(2024)Shen, Berlinghieri, and Broderick]{shen2024learning}
Yunyi Shen, Renato Berlinghieri, and Tamara Broderick.
\newblock Learning a vector field from snapshots of unidentified particles rather than particle trajectories.
\newblock \emph{ICLR Workshop on AI4DifferentialEquations In Science}, 2024.

\bibitem[Soloff et~al.(2024)Soloff, Guntuboyina, and Sen]{soloff2024multivariate}
Jake~A Soloff, Adityanand Guntuboyina, and Bodhisattva Sen.
\newblock Multivariate, heteroscedastic empirical bayes via nonparametric maximum likelihood.
\newblock \emph{Journal of the Royal Statistical Society Series B: Statistical Methodology}, page qkae040, 2024.

\bibitem[Tong et~al.(2020)Tong, Huang, Wolf, Van~Dijk, and Krishnaswamy]{tong2020trajectorynet}
Alexander Tong, Jessie Huang, Guy Wolf, David Van~Dijk, and Smita Krishnaswamy.
\newblock Trajectorynet: A dynamic optimal transport network for modeling cellular dynamics.
\newblock In \emph{International conference on machine learning}, pages 9526--9536. PMLR, 2020.

\bibitem[van~de Geer(2000)]{geer2000empirical}
Sara~A van~de Geer.
\newblock \emph{Empirical Processes in M-estimation}, volume~6.
\newblock Cambridge University Press, 2000.

\bibitem[van~der Vaart and Wellner(2013)]{wellner2013weak}
Aad~W van~der Vaart and Jon~A Wellner.
\newblock \emph{{Weak Convergence and Empirical Processes With Applications to Statistics}}.
\newblock Springer Science \& Business Media, 2013.

\bibitem[van Handel(2014)]{van2014probability}
Ramon van Handel.
\newblock {Probability in High Dimension}.
\newblock \emph{Lecture Notes (Princeton University)}, 2\penalty0 (3):\penalty0 2--3, 2014.

\bibitem[Vempala and Wibisono(2019)]{vempala2019rapid}
Santosh Vempala and Andre Wibisono.
\newblock {Rapid convergence of the unadjusted Langevin algorithm: Isoperimetry suffices}.
\newblock \emph{Advances in neural information processing systems}, 32, 2019.

\bibitem[Vershynin(2018)]{vershynin2018high}
Roman Vershynin.
\newblock \emph{{High-Dimensional Probability: An Introduction with Applications in Data Science}}, volume~47.
\newblock Cambridge University Press, 2018.

\bibitem[Villani(2009)]{villani2009optimal}
C{\'e}dric Villani.
\newblock \emph{Optimal Transport: Old and New}, volume 338.
\newblock Springer, 2009.

\bibitem[Villani(2021)]{villani2021topics}
C{\'e}dric Villani.
\newblock \emph{{Topics in Optimal Transportation}}, volume~58.
\newblock American Mathematical Soc., 2021.

\bibitem[Wainwright(2019)]{wainwright2019high}
Martin~J Wainwright.
\newblock \emph{{High-Dimensional Statistics: A Non-Asymptotic Viewpoint}}, volume~48.
\newblock Cambridge University Press, 2019.

\bibitem[Wibisono(2018)]{wibisono2018sampling}
Andre Wibisono.
\newblock {Sampling as optimization in the space of measures: The Langevin dynamics as a composite optimization problem}.
\newblock In \emph{Conference on Learning Theory}, pages 2093--3027. PMLR, 2018.

\bibitem[Yan et~al.(2024)Yan, Wang, and Rigollet]{yan2024learning}
Yuling Yan, Kaizheng Wang, and Philippe Rigollet.
\newblock {Learning Gaussian mixtures using the Wasserstein--Fisher--Rao gradient flow}.
\newblock \emph{The Annals of Statistics}, 52\penalty0 (4):\penalty0 1774--1795, 2024.

\bibitem[Yao and Yang(2022)]{yao2022meana}
Rentian Yao and Yun Yang.
\newblock Mean-field variational inference via wasserstein gradient flow.
\newblock \emph{arXiv preprint arXiv:2207.08074}, 2022.

\bibitem[Yao et~al.(2022)Yao, Chen, and Yang]{yao2022mean}
Rentian Yao, Xiaohui Chen, and Yun Yang.
\newblock Mean-field nonparametric estimation of interacting particle systems.
\newblock In \emph{Conference on Learning Theory}, pages 2242--2275. PMLR, 2022.

\bibitem[Yao et~al.(2024{\natexlab{a}})Yao, Chen, and Yang]{yao2024wasserstein}
Rentian Yao, Xiaohui Chen, and Yun Yang.
\newblock Wasserstein proximal coordinate gradient algorithms.
\newblock \emph{Journal of Machine Learning Research}, 25\penalty0 (269):\penalty0 1--66, May 2024{\natexlab{a}}.

\bibitem[Yao et~al.(2024{\natexlab{b}})Yao, Huang, and Yang]{yao2024minimizing}
Rentian Yao, Linjun Huang, and Yun Yang.
\newblock {Minimizing Convex Functionals over Space of Probability Measures via KL Divergence Gradient Flow}.
\newblock In \emph{International Conference on Artificial Intelligence and Statistics}, pages 2530--2538. PMLR, 2024{\natexlab{b}}.

\bibitem[Yeo et~al.(2020)Yeo, Saksena, and Gifford]{yeo2020generative}
Grace~HT Yeo, Sachit~D Saksena, and David~K Gifford.
\newblock Generative modeling of single-cell population time series for inferring cell differentiation landscapes.
\newblock \emph{BioRxiv}, pages 2020--08, 2020.

\bibitem[Zhang(2009)]{zhang2009generalized}
Cun-Hui Zhang.
\newblock Generalized maximum likelihood estimation of normal mixture densities.
\newblock \emph{Statistica Sinica}, pages 1297--1318, 2009.

\bibitem[Zhu and Chen(2025)]{ZhuChen2025_convergence-nonconvex}
Shuailong Zhu and Xiaohui Chen.
\newblock Convergence analysis of the wasserstein proximal algorithm beyond geodesic convexity, January 2025.
\newblock URL \url{https://arxiv.org/abs/2501.14993}.

\end{thebibliography}
\newpage
\appendix
\begin{center}
{\bf\Large Supplementary Materials: Appendix}
\end{center}
\numberwithin{equation}{section}

\newtheorem{theoremA}{Theorem}
\renewcommand{\thetheoremA}{A\arabic{theoremA}}
\newtheorem{lemmaA}[theoremA]{Lemma}
\renewcommand{\thelemmaA}{A\arabic{lemmaA}}

This appendix provides technical details of the theoretical results presented in the main paper. The appendix is structured as follows. Appendix~\ref{app: bg} offers background knowledge on optimal transport and empirical process theory, which are essential for proving our theoretical results. Appendix~\ref{app: stat_rate} includes the proof of the finite sample analysis and its related results. Appendix~\ref{sec: pf_algo_converge} presents the proofs of algorithmic convergence results in Theorems~\ref{thm: algo_convergence_general} and~\ref{thm: traj_CKLGD}. All technical details are deferred to Appendix~\ref{sec: tech_lem}. In the appendix, we use the notation $H(\rho) := \int\rho\log\rho$.

\section{Additional Background}\label{app: bg}
\subsection{Optimal Transport and sampling}\label{app: OT}
When applying inexact CKLGD to compute the estimator~\eqref{eqn: obj_func}, additional sampling procedure are applied to approximate the $\wt\rho_j^k$ in~\eqref{eqn: explicit_quad_anneal_update} in each iteration, as shown in Algorithm~\ref{alg: inexact_CKLGD}. Therefore, to derive the total number of iterations in Algorithm~\ref{alg: inexact_CKLGD}, it is also important to see how many (inner) iterations are required during the sampling steps. This section aims to provide some background knowledge that is helpful for analyzing the inner iteration complexity.

Sampling from a distribution is strongly related to optimizing a functional on the probability space~\citep{wibisono2018sampling}. It is also known that the unadjusted Langevin algorithm (ULA) for sampling converges to a neighborhood of the target distribution exponentially fast, when the target distribution is strongly log-concave. This is analogous to using gradient descent algorithm to minimize a strongly convex function. In the Euclidean space, such strong convexity can be relaxed to the well-known Polyak--Lojasiewicz (PL) inequality without hurting the exponential convergence rate. In the literature of sampling and optimization on the probability space, the following inequality, known as log-Sobolev inequality (LSI), plays the same role as PL inequality in the Euclidean space optimization. For more details of LSI, we refer to~\citep{ollivier2014optimal, villani2021topics, villani2009optimal}.

\begin{definition}
A probability distribution $\mu$ satisfies $\LSI(\Lambda)$ if
\begin{align*}
    \int_{\m X}\Big\|\nabla\log\frac{\dd\nu}{\dd\mu}\Big\|^2\,\dd\nu \geq 2\Lambda \KL(\nu\,\|\,\mu)
\end{align*}
holds for all $\nu\in\ms P_2^r(\m X)$.
\end{definition}
It is known that a log-concave probability distribution naturally satisfies LSI. The following provides a sufficient condition of a distribution satisfying LSI, that is the perturbation of a distribution satisfying LSI still satisfies LSI, but with a different constant.

\begin{proposition}[\citet{holley1986logarithmic}]\label{prop: perturb_LSI}
Let $\mu\in\ms P_2^r(\m X)$ be a probability distribution satisfying $\LSI(\Lambda)$. For a bounded function $H$, let $\mu_H\propto \mu e^{-H}$ be a new probability distribution. Then, $\mu_H$ satisfies $\LSI(\Lambda_H)$ with $\Lambda_H = \Lambda e^{-4\|H\|_{L^\infty}}$.
\end{proposition}

In the Euclidean space, PL inequality implies that the objective function has a quadratic growth property~\citep{karimi2016linear}. On the Wasserstein space, LSI also implies the quadratic growth property of the objective functional.
\begin{proposition}[Talagrand's transportation inequality~\citep{otto2000generalization}]\label{prop: talagrand}
If $\mu\in\ms P_2^r(\m X)$ satisfies $\LSI(\Lambda)$, then for any $\nu\in\ms P_2^r(\m X)$, it holds that
\begin{align*}
\KL(\nu\,\|\,\mu) \geq \Lambda\W_2^2(\nu, \mu).
\end{align*}
\end{proposition}
\subsection{Empirical process theory}
In this subsection, we collection some results in the field of empirical process that are helpful to the proof of the statistical results. More details are referred to the monographs~\citep{wainwright2019high, vershynin2018high, van2014probability, wellner2013weak, geer2000empirical}.

The following definition of Orlicz norm characterizes the tail of a random variable. Generally, the sample mean of a group of i.i.d. random variables with finite Orlicz norm is closed to the population mean.
\begin{definition}[Orlicz norm]
For $\alpha \geq 1$, define the function $\psi_\alpha(x) = e^{x^\alpha} - 1$. Then for any random variable $X$, the Orlicz norm is defined as
\begin{align*}
\|X\|_{\psi_\alpha} \coloneqq \inf\{c > 0: \mb E\psi(|X|/c) \leq 1\}.
\end{align*}
Here, the infimum of an empty set is defined as $+\infty$.
\end{definition}

The following proposition provide an upper bound of the $L^1$-norm of random variable via its $\psi_1$-norm.
\begin{proposition}
For any random variable $X$, it holds that
\begin{align}\label{eqn: E-psi-bound}
\mb E|X| \leq \|X\|_{\psi_1}.
\end{align}   
\end{proposition}

Another key concept in non-asymptotic analysis is covering number, which is useful when applying the so-called \emph{chaining} technique. By applying the chaining technique, we can approximate an arbitrary function in certain function spaces with a finite number of functions with controllable errors.
\begin{definition}[covering number]
Let $(S, d_S)$ be a metric space and $A\subset S$ be a subset of $S$. A \emph{$\eta$-covering} of $A$ is a subset $\{a_1, \dots, a_n\}\subset A$, such that for each element $a\in A$, there exists $j\in[n]$ such that $d_S(a, a_j) \leq \eta$. The \emph{$\eta$-covering number} $N(\eta, A, d_S)$ is the cardinality of the smallest $\eta$-covering.
\end{definition}
\section{Proof of Statistical Guarantee}\label{app: stat_rate}
The goal of this section is to prove Theorem~\ref{thm: stat_rate}, which provides a finite sample analysis of the E-NPMLE estimator defined in~\eqref{eqn: E-NPMLE} and~\eqref{eqn: obj_func_traj}. The whole proof consists of several steps.
First, in Appendix~\ref{app: pf_stat_conv_discrete}, we prove the main statistical result, Theorem~\ref{thm: stat_rate}, by using Lemma~\ref{lem: tail bound} which provides a high probability bound of the empirical process. Next, in Appendix~\ref{app: covering}, we provide an estimate of the covering number of the space of Gaussian convoluted path-space distribution. Then in Appendix~\ref{app: pf_of_emp_proc}, we use the estimate of the covering number to prove the key Lemma~\ref{lem: tail bound}. At the end of this section in Appendix~\ref{app: coro_pf}, we prove Theorem~\ref{coro: cts_stat_rate} for estimating the density flow map $t\mapsto R_t^\ast$.
\subsection{Proof of Theorem~\ref{thm: stat_rate}}\label{app: pf_stat_conv_discrete}
\underline{Step 0: notations.}
For any $R\in\ms P(\Omega)$ and $t\in[0, 1]$, define
\begin{align*}
g_R(t, x) := \sqrt{\frac{\m K_\sigma\ast R_t(x) + \m K_\sigma\ast R_t^\ast(x)}{2\m K_\sigma\ast R_t^\ast(x)}} \geq \frac{1}{\sqrt{2}}.
\end{align*}
Then, we have $g_R(t, x) \geq 1/\sqrt{2}$. For any $R, R'\in\ms P(\Omega)$, it is easy to check that
\begin{align}\label{eqn: connectH}
\begin{aligned}
d_{\rm H}^2\Big(\frac{\m K_\sigma R_{t_j} + \m K_\sigma R_{t_j}^\ast}{2}, \frac{\m K_\sigma R_{t_j}' + \m K_\sigma R_{t_j}^\ast}{2}\Big) &= \|g_R(t_j, \cdot) - g_{R'}(t_j, \cdot)\|_{L^2(\m K_\sigma R^\ast_{t_j})}^2\\
&= \mb E\big[g_R(t_j, X_{t_j}) - g_{R'}(t_j, X_{t_j})\big]^2.
\end{aligned}
\end{align}
Furthermore, we define the $\|\cdot\|_{L_m^2}$-norm by
\begin{align}\label{eqn: newL2norm}
\|g_R - g_{R'}\|_{L^2_m}^2 := \sum_{j=1}^m(t_{j+1}-t_j)\big\|g_{R}(t_j,\cdot) - g_{R'}(t_j, \cdot)\big\|_{L^2(\m K_\sigma\ast R_{t_j}^\ast)}^2,
\end{align}
and $\|\cdot\|_{L_m^\infty}$-norm by
\begin{align*}
\|g_R - g_{R'}\|_{L_m^\infty} := \sup_{t\in\{t_1, \dots, t_m\}, x\in\m X} |g_R(t, x) - g_{R'}(t, x)|.
\end{align*}
It is easy to check that
\begin{align*}
\|g_R - g_{R'}\|_{L_m^\infty} \leq \sup_{t\in\{t_1, \dots, t_m\}, x\in\m X}\frac{|\m K_\sigma\ast R_t(x) - \m K_\sigma\ast R_t'(x)|}{\sqrt{2}} \leq \sqrt{2}C_\sigma.
\end{align*}

\noindent\underline{Step 1: proof of modified basic inequality.}
By the optimality of $\wht R$, we have
\begin{align*}
\sum_{j=1}^m\frac{t_{j+1}-t_j}{N}\sum_{i=1}^N\log\frac{\m K_\sigma\ast \wht R_{t_j}(X_{t_j}^i)}{\m K_\sigma\ast R^\ast_{t_j}(X_{t_j}^i)}
\geq \lambda\big[\tau\KL(\wht R\,\|\,W^\tau) - \tau\KL(R^\ast\,\|\,W^\tau)\big].
\end{align*}
By Jensen's inequality, we have
\begin{align*}
\log g_R(t_j, X_{t_j}^i) = \frac{1}{2}\log\Big(\frac{1}{2}\cdot\frac{\m K_\sigma\ast R_{t_j}(X_{t_j}^i)}{\m K_\sigma\ast R_t^\ast(X_{t_j}^i)} + \frac{1}{2}\Big) \geq \frac{1}{4}\log\frac{\m K_\sigma\ast R_{t_j}(X_{t_j}^i)}{\m K_\sigma\ast R_{t_j}^\ast(X_{t_j}^i)}.
\end{align*}
Therefore, we have
\begin{align*}
-\sum_{j=1}^m\frac{t_{j+1}-t_j}{N}\sum_{i=1}^N\log g_{\wht R}(t_j, X_{t_j}^i)
&\leq -\sum_{j=1}^m\frac{t_{j+1}-t_j}{N}\sum_{i=1}^N \frac{1}{4}\log\frac{\m K_\sigma\ast\wht R_{t_j}(X_{t_j}^i)}{\m K_\sigma\ast R_{t_j}^\ast(X_{t_j}^i)}\\
&\leq \frac{\lambda\tau}{4}\big[\KL(R^\ast\,\|\,W^\tau) - \KL(\wht R\,\|\,W^\tau)\big].
\end{align*}

\noindent\underline{Step 2: convergence of the M-estimator.} 
Note that for any $R$, we have
\begin{align*}
\sum_{j=1}^m\sum_{i=1}^N\frac{t_{j+1}-t_j}{N}\mb E\log g_R(t_j, X_{t_j}^i)
&= -\sum_{j=1}^m (t_{j+1}-t_j) \KL\Big(\m K_\sigma\ast R_{t_j}^\ast\,\Big\|\, \frac{\m K_\sigma\ast R_{t_j}^\ast + \m K_\sigma\ast R_{t_j}}{2}\Big)\\
&\leq -\sum_{j=1}^m (t_{j+1}-t_j) d_{\rm H}^2\Big(\m K_\sigma\ast R_{t_j}^\ast\,\Big\|\, \frac{\m K_\sigma\ast R_{t_j}^\ast + \m K_\sigma\ast R_{t_j}}{2}\Big)\\
&= -\|g_R - g_{R^\ast}\|_{L_m^2}^2.
\end{align*}
Recall that we have the modified basic inequality
\begin{align*}
-\sum_{j=1}^m\frac{t_{j+1}-t_j}{N}\sum_{i=1}^N\log g_{\wht R}(t_j, X_{t_j}^i)
&\leq \frac{\lambda\tau}{4}\big[\KL(R^\ast\,\|\,W^\tau) - \KL(\wht R\,\|\,W^\tau)\big].
\end{align*}

\underline{\underline{Case 2.1: $\tau\KL(\wht R\,\|\, W^\tau) \leq 2\tau\KL(R^\ast\,\|\,W^\tau)$.}} In this case, we have
\begin{align*}
-\sum_{j=1}^m\frac{t_{j+1}-t_j}{N}\sum_{i=1}^N\log g_{\wht R}(t_j, X_{t_j}^i)
&\leq \frac{\lambda\tau}{4}\KL(\wht R^\ast\,\|\,W^\tau).
\end{align*}
Therefore, we have
\begin{align*}
\sum_{j=1}^m \frac{t_{j+1}-t_j}{N}\sum_{i=1}^N\big[\log g_{\wht R}(t_j, X_{t_j}^i) - \mb E\log g_{\wht R}(t_j, X_{t_j}^i)\big]
\geq -\frac{\lambda\tau}{4}\KL(R^\ast\,\|\,W^\tau) + \|g_{\wht R} - g_{R^\ast}\|_{L_m^2}^2.
\end{align*}
From Lemma~\ref{lem: tail bound}, we know that
\begin{align*}
C_{\rm HP}\delta_{N, m}\big(\delta_{N, m} + \|g_{\wht R} - g_{R^\ast}\|_{L_m^2}\big)
\geq \sum_{j=1}^m \frac{t_{j+1}-t_j}{N}\sum_{i=1}^N\big[\log g_{\wht R}(t_j, X_{t_j}^i) - \mb E\log g_{\wht R}(t_j, X_{t_j}^i)\big]
\end{align*}
holds with probability at least $\mb P(\ms A)$. When, the above inequality holds, we have
\begin{align*}
-\frac{\lambda\tau}{4}\KL(R^\ast\,\|\,W^\tau) + \|g_{\wht R} - g_{R^\ast}\|_{L_m^2}^2
\leq C_{\rm HP}\delta_{N, m}\big(\delta_{N, m} + \|g_{\wht R} - g_{R^\ast}\|_{L_m^2}\big),
\end{align*}
which implies
\begin{align*}
\|g_{\wht R} - g_{R^\ast}\|_{L_m^2} \leq \big(C_{\rm HP} + \sqrt{C_{\rm HP}}\big)\delta_{N, m} + \frac{1}{2}\sqrt{\lambda\tau\KL(R^\ast\,\|\,W^\tau)}.
\end{align*}

\underline{\underline{Case 2.2: $\tau\KL(\wht R\,\|\,W^\tau) \geq 2\tau\KL(R^\ast\,\|\,W^\tau)$.}}
Take $\varepsilon = \frac{\tau\KL(\wht R\,\|\,W^\tau) - \frac{3}{2}\tau\KL( R^\ast\,\|\,W^\tau)}{\tau\KL(\wht R\,\|\,W^\tau) - \tau\KL(R^\ast\,\|\,W^\tau)}\in(0, 1)$, and let $\wt R = (1-\varepsilon)\wht R + \varepsilon R^\ast \in\ms P(\Omega)$. By the convexity of KL divergence, we have
\begin{align*}
\KL(\wt R\,\|\,W^\tau)
&\leq (1-\varepsilon)\KL(\wht R\,\|\,W^\tau) + \varepsilon\KL(R^\ast\,\|\,W^\tau)
= \frac{3}{2}\KL(R^\ast\,\|\,W^\tau).
\end{align*}
Similarly, we have
\begin{align*}
&\quad\,\sum_{j=1}^m\frac{t_{j+1}-t_j}{N\lambda}\sum_{i=1}^N\log\frac{\m K_\sigma\ast \wt R_{t_j}(X_{t_j}^i)}{\m K_\sigma\ast R^\ast_{t_j}(X_{t_j}^i)}\\
&\geq (1-\varepsilon)\sum_{j=1}^m\frac{t_{j+1}-t_j}{N\lambda}\sum_{i=1}^N\log\frac{\m K_\sigma\ast \wht R_{t_j}(X_{t_j}^i)}{\m K_\sigma\ast R^\ast_{t_j}(X_{t_j}^i)} + \varepsilon \sum_{j=1}^m\frac{t_{j+1}-t_j}{N\lambda}\sum_{i=1}^N\log\frac{\m K_\sigma\ast  R^\ast_{t_j}(X_{t_j}^i)}{\m K_\sigma\ast R^\ast_{t_j}(X_{t_j}^i)}\\
&\geq (1-\varepsilon)\lambda\big[\tau\KL(\wht R\,\|\,W^\tau) - \tau\KL(R^\ast\,\|\,W^\tau)\big]\\
&= \frac{\lambda}{2}\tau\KL(R^\ast\,\|\,W^\tau).
\end{align*}
Therefore, we have
\begin{align*}
-\sum_{j=1}^m \frac{t_{j+1}-t_j}{N}\sum_{i=1}^N\log g_{\wt R}(t_j, X_{t_j}^i) 
\leq -\frac{1}{4}\sum_{j=1}^m \frac{t_{j+1}-t_j}{N}\sum_{i=1}^N\log\frac{\m K_\sigma\ast \wt R_{t_j}(X_{t_j}^i)}{\m K_\sigma\ast R_{t_j}^\ast(X_{t_j}^i)} \leq -\frac{\lambda}{8}\tau\KL(R^\ast\,\|\,W^\tau).
\end{align*}
In this case, we have
\begin{align*}
\sum_{j=1}^m \sum_{i=1}^N\frac{t_{j+1}-t_j}{N} \big[\log g_{\wt R}(t_j, X_{t_j}^i) -\mb E\log g_{\wt R}(t_j, X_{t_j}^i)\big] 
\geq \frac{\lambda}{8}\tau\KL(R^\ast\,\|\,W^\tau) + \|g_{\wt R} - g_{R^\ast}\|_{L_m^2}^2.
\end{align*}
Note that when $\lambda \geq \frac{2C_{\rm HP}^2\delta_{N,m}^2 + 8C_{\rm HP}\delta^2}{\tau\KL(R^\ast\,\|\,W^\tau)}$, it always holds that
\begin{align*}
\frac{\lambda}{8}\tau\KL(R^\ast\,\|\,W^\tau) + \|g_{\wt R} - g_{R^\ast}\|_{L_m^2}^2 \geq C_{\rm HP}\delta_{N,m}\big(\|g_{\wt R} - g_{R^\ast}\|_{L_m^2}\big).
\end{align*}
Therefore, we have
\begin{align*}
\sum_{j=1}^m \sum_{i=1}^N\frac{t_{j+1}-t_j}{N} \big[\log g_{\wt R}(t_j, X_{t_j}^i) -\mb E\log g_{\wt R}(t_j, X_{t_j}^i)\big] 
\geq C_{\rm HP}\delta_{N,m}\big(\|g_{\wt R} - g_{R^\ast}\|_{L_m^2}\big),
\end{align*}
which violates $\ms A$. Therefore, this case happens with probability at most $\mb P(\ms A)$.

To sum up, we have shown that when $\lambda \geq \frac{2C_{\rm HP}^2\delta_{N,m}^2 + 8C_{\rm HP}\delta^2}{\tau\KL(R^\ast\,\|\,W^\tau)}$, it holds
\begin{align*}
\|g_{\wht R} - g_{R^\ast}\|_{L_m^2} \leq \big(C_{\rm HP} + \sqrt{C_{\rm HP}}\big)\delta_{N, m} + \frac{1}{2}\sqrt{\lambda\tau\KL(R^\ast\,\|\,W^\tau)}
\end{align*}
with probability at least $1 - \mb P(\ms A)$. 
Therefore, we have
\begin{align*}
\sqrt{\sum_{j=1}^m(t_{j+1}-t_j)d_{\rm H}^2\big(\m K_\sigma\ast \wht R_{t_j}, \m K_\sigma\ast R_{t_j}^\ast\big)}
\leq (2+\sqrt{2})\big(C_{\rm HP} + \sqrt{C_{\rm HP}}\big)\delta_{N, m} + \frac{2+\sqrt{2}}{2}\sqrt{\lambda\tau\KL(R^\ast\,\|\,W^\tau)}
\end{align*}

\noindent\underline{Step 3: decide the order of the statistical radius $\delta_{N, m}$.}
Recall that $\delta_{N, m}$ satisfies
\begin{align*}
4\sqrt{\frac{2\Delta_m}{N}}\cdot C_{\rm MI}\Big[\sqrt{\frac{\Delta_m}{N}}\cdot\frac{\sqrt{m}}{\delta_{N, m}} + 1\Big]\cdot\min\{\delta_{N,m}\sqrt{m}, 2\sqrt{2}\}\cdot\big[\max\big\{\log \delta_{N, m}^{-1}, \log m\big\}\big]^{d+1} \lesssim \delta_{N,m}^2.
\end{align*}
We will only focus on finding $\delta_{N,m}$ with the possibly smallest order of $N$ and $m$, or equivalently, we want the order of $N$ and $m$ on the both sides match. In the following argument, we use $\approx$ for the meaning of same order.

\underline{\underline{Case 3.1: $\delta_{N, m}\sqrt{m}\leq 2\sqrt{2}$.}}
In this case, we have
\begin{align*}
\lhs
\approx \Big[\frac{m\Delta_m}{N} + \delta_{N,m}\sqrt{\frac{m\Delta_m}{N}}\Big]\Big(\log\frac{1}{\delta_{N,m}}\Big)^{d+1}
\approx \delta_{N,m}^2 = \rhs.
\end{align*}
Therefore, we have
\begin{align*}
\delta_{N,m} \approx \sqrt{\frac{m\Delta_m}{N}}\Big(\log\frac{N}{m\Delta_m}\Big)^{\frac{d+1}{2}}.
\end{align*}
When taking $\Delta_m \approx m^{-1}$, such as when all time points have equal separation, the above result implies $\delta_{N, m} = O\big(\frac{(\log N)^{\frac{d+1}{2}}}{\sqrt{N}}\big)$. Note that this case only happens when $N\gtrsim m$.

\underline{\underline{Case 3.2: $\delta_{N, m}\sqrt{m} > 2\sqrt{2}$.}}
In this case, we have
\begin{align*}
\lhs
\approx \Big[\frac{\Delta_m}{N}\cdot\frac{\sqrt{m}}{\delta_{N,m}} + \sqrt{\frac{\Delta_m}{N}}\Big](\log{m})^{d+1}
\approx \delta_{N,m}^2 = \rhs.
\end{align*}
In this case, we have
\begin{align*}
\delta_{N,m} \approx \max \Big\{\frac{\Delta_m^{1/3}m^{1/6}}{N^{1/3}}, \frac{\Delta_m^{1/4}}{N^{1/4}}\Big\}(\log m)^{\frac{d+1}{2}}.
\end{align*}
Again, when $\Delta_m \approx m^{-1}$, the above result implies
\begin{align*}
\delta_{N, m} \approx \max\Big\{\frac{1}{(Nm)^{\frac{1}{4}}}, \frac{1}{N^{\frac{1}{3}}m^\frac{1}{6}}\Big\}(\log m)^{\frac{d+1}{2}}.
\end{align*}
Note that this only happens when $N\lesssim m$. So, we finally have $\delta_{N, m} = O\big(\frac{(\log m)^{\frac{d+1}{2}}}{N^{\frac{1}{3}}m^{\frac{1}{6}}}\big)$.

\subsection{Control of covering number}\label{app: covering}
In the following proposition, we derive an upper bound of the covering number of Gaussian convoluted path measures with bounded KL divergence with respect to $W^\tau$. 
\begin{proposition}\label{prop: covering_num}
Let $\m X = \mb T^d = [-\pi, \pi]^d$ and $K$ be a subset of $[0, 1]$, we have
\begin{align*}
\log N\big(\eta,\big\{\m K_\sigma R_\cdot(\cdot): R\in\ms P(\Omega), \tau\KL(R\,\|\,W^\tau)\leq 2E\big\} , \|\cdot\|_{L^\infty(K\times\m X)}\big) \lesssim \min\big\{\eta^{-2}, |K|\big\}\cdot\Big(\log\frac{1}{\eta}\Big)^{d+1}.
\end{align*}
\end{proposition}
\begin{proof}
\underline{Step 1: construction of projection map.} Let $M\in\mb Z_+$ be an integer to be decided later, and $T_1 < \dots < T_{N_I}$ be a $r_I$-covering of $K\subset I = [0, 1]$ (not necessarily in $K$). For any $j\in[N_I]$, define the map $I_M^j:\ms P(\Omega)\to\mb R^{(2M+1)^d}$ by
\begin{align*}
I_M^j(R) := \bigg(\int_{\mb T^d}x_1^{k_1}\cdots x_d^{k_d}\,\dd R_{T_j}, 0\leq k_1,\dots, k_d\leq 2M\bigg).
\end{align*}
Then, the set $E_M^j:= \{I_M^j(R): R\in\ms P(\Omega)\}$ is a convex subset in $\mb R^{(2M+1)^d}$; moreover, it is a convex hall of
\begin{align*}
\Big\{(x_1^{k_1}\cdots x_d^{k_d}, 0\leq k_1,\dots, k_d\leq 2M): (x_1, \dots, x_d)\in\mb T^d\Big\} \subset \mb R^{(2M+1)^d}.
\end{align*}
By Caratheodory's theorem, every element in $E_M^j$ is a convex combination of at most $l:= (2M+1)^d + 1$ elements of the above set. Therefore, we know $E_M^j$ is a subset of the set
\begin{align*}
D_M := \bigg\{\bigg(\int_{\mb T^d}x_1^{k_1}\cdots x_m^{k_m}\,&\dd R_{T_j}, 0\leq k_1,\dots, k_d\leq 2M\bigg)\\
&: R_{t_j}\,\,\mx{is a discrete probability measure on $\ms P(\mb T^d)$ with at most $l$ atoms}\bigg\}.
\end{align*}
Then, for every $R\in\ms P(\Omega)$ and $j\in[N_I]$, we can define a discrete probability measure
\begin{align*}
\mu_R^j = \proj_{D_M}^j(R) := \sum_{s=1}^l w_j^s\delta_{x^{j, s}} \in\ms P(\mb T^d),
\end{align*}
with $w_j = (w_j^1, \dots, w_j^l)\in\Delta^l$ (probability simplex) and $x^{j,s} = (x_1^{j, s},\dots, x_d^{j, s})\in\mb T^d$, such that
\begin{align*}
\int_{\mb T^d}x_1^{k_1}\cdots x_d^{k_d}\,\dd R_{T_j} = \int_{\mb T^d}x_1^{k_1}\cdots x_d^{k_d}\,\dd\mu_R^j,\quad\forall\,0\leq k_1,\dots, k_d\leq 2M.
\end{align*}

Now, let $\m A_{\mb T^d} = \{v_1, \dots, v_{N_{\mb T^d}}\}$ be a $r_{\mb T^d}$-covering for $\mb T^d$ and let $\m A_{\Delta^l} = \{\beta_1, \dots, \beta_{N_{\Delta^l}}\}$ be a $r_{\Delta^l}$-covering of the space of $l$-dimensional probability vectors $\Delta^l$ under $\ell^1$-norm. Now, define another two projection maps
\begin{align*}
\proj_{\mb T^d}(x) = \argmin_{v\in\m A_{\mb T^d}}d_{\mb T^d}(x, v),
\quad\mx{and}\quad
\proj_{\Delta^l}(w) = \argmin_{\beta\in\m A_{\Delta^l}}\|\beta-w\|_{\ell^1}.
\end{align*}
So, for any $j\in[N_I]$, there is a sequence of mapping
\begin{align*}
R \mapsto 
\begin{pmatrix}
\mu_R^1\\ \vdots\\ \mu_R^{N_I}
\end{pmatrix}
= \begin{pmatrix}
\sum_{s=1}^l w_1^s\delta_{x^{1,s}}\\
\vdots\\
\sum_{s=1}^l w_{N_I}^s\delta_{x^{N_I, s}}
\end{pmatrix}
\mapsto 
\begin{pmatrix}
\proj_{\Delta^l}(w_1) & \proj_{\mb T^d}(x^{1,1}) & \dots & \proj_{\mb T^d}(x^{1,l})\\
\vdots & \vdots & \ddots & \vdots\\
\proj_{\Delta^l}(w_{N_I}) & \proj_{\mb T^d}(x^{N_I, 1}) & \dots & \proj_{\mb T^d}(x^{N_I, l})
\end{pmatrix}
=: \proj(R).
\end{align*}
Note that, for each row, there exists at most $N_{\Delta^l}\binom{N_{\mb T^d} + l -1}{l}$ different possible values. Therefore, $\proj(R)$ can have at most $\big[N_{\Delta^l}\binom{N_{\mb T^d}+l-1}{l}\big]^{N_I}$ different matrices.

\vspace{0.5em}
\noindent\underline{Step 2: control the $L^\infty$ distance of different path measures with same projection.} Now, for any $R, \wt R\in\ms P(\Omega)$ such that $\tau\KL(R\,\|\,W^\tau)\leq S$, $\tau\KL(\wt R\,\|\,W^\tau)\leq S$, and $\proj(R) = \proj(\wt R)$, let us control $\|\m K_\sigma\ast R_t(x) - \m K_\sigma\ast\wt R_t(x)\|_{L^\infty(K\times\mb T^d)}$. Note that for every $j\in[N_I]$, we have
\begin{align}
&\quad\,\lvert\m K_\sigma\ast R_t(x) - \m K_\sigma\ast\wt R_t(x)\rvert\notag\\
&\leq \lvert\m K_\sigma\ast R_{T_j}(x) - \m K_\sigma\ast\wt R_{T_j}(x)\rvert + \lvert\m K_\sigma\ast R_t(x) - \m K_\sigma\ast R_{T_j}(x)\rvert + \lvert\m K_\sigma\ast\wt R_t(x) - \m K_\sigma\ast\wt R_{T_j}(x)\rvert\notag\\
&\leq \lvert\m K_\sigma\ast R_{T_j}(x) - \m K_\sigma\ast\wt R_{T_j}(x)\rvert + 2C_{\rm Hol}\sqrt{t - T_j} \label{eqn: gauss_conv_diff_bd}
\end{align}
for some constant $C_{\rm Hol} = C_{\rm Hol}(\sigma, \tau, S)$ by~\citep[Proposition 2.12,][]{lavenant2024toward}. Now, assume that
\begin{align*}
\proj(R) = \proj(\wt R) =
\begin{pmatrix}
\beta^1 & v^{1, 1} & \dots & v^{1,l}\\
\vdots & \vdots & \ddots & \vdots\\
\beta^{N_I} & v^{N_I, 1} & \dots & v^{N_I, l}
\end{pmatrix}.
\end{align*}
We can then define the reconstructed probability measure
\begin{align*}
R_j^{\rm rec} := \sum_{s=1}^l \beta^{j,s}\delta_{v^{j, s}},
\end{align*}
where $\beta^j = (\beta^{j, 1}, \dots, \beta^{j, l})\in\Delta^l$. Then, we have
\begin{align*}
\big\lvert \m K_\sigma\ast R_{T_j}(x) - \m K_\sigma\ast \wt R_{T_j}(x)\big\rvert
\leq \big\lvert \m K_\sigma\ast R_{T_j}(x) - \m K_\sigma\ast R^{\rm rec}_{T_j}(x)\big\rvert + \big\lvert \m K_\sigma\ast\wt R_{T_j}(x) - \m K_\sigma\ast \wt R^{\rm rec}_{T_j}(x)\big\rvert.
\end{align*}
To control the first term, we know that the probability $\mu_R^j = \sum_{s=1}^l w_j^s\delta_{x^{j, s}}\in\ms P(\mb R^d)$ satisfies
\begin{align*}
\int_{\mb T^d}x_1^{k_1}\cdots x_d^{k_d}\,\dd R_{T_j} = \int_{\mb T^d}x_1^{k_1}\cdots x_d^{k_d}\,\dd\mu_R^j = \sum_{s=1}^lw_j^s(x_1^{j, s})^{k_1}\dots (x_d^{j, s})^{k_d},
\quad\forall\, 0\leq k_1,\dots, k_d\leq 2M,
\end{align*}
and
\begin{align}\label{eqn: covering_prop}
d_{\mb T^d}(x^{j, s}, v^{j, s}) \leq r_{\mb T^d},
\qquad \|w_j - \beta^{j}\|_{\ell^1}\leq r_{\Delta^l}.
\end{align}
Let $p_\sigma(x) \propto \sum_{k\in\mb Z^d}e^{-\frac{\|x-2\pi k\|^2}{2\sigma^2}}$ be the density function of Gaussian distribution on $\mb T^d$~\citep{berg1976potential}. Note that $p_\sigma(x)$ is well defined on $\mb R^d$ since it is periodic. Therefore, we have
\begin{align*}
\big\lvert\m K_\sigma\ast R_{T_j}(x) - \m K_\sigma\ast R_j^{\rm rec}(x)\big\rvert
&= \bigg\lvert \int_{\mb T^d} \sum_{k\in\mb Z^d}e^{-\frac{\|x - y - 2\pi k\|^2}{2\sigma^2}}\big[R_{T_j}(y) - R_j^{\rm rec}(y)\big]\,\vol(\dd y)\bigg\rvert\\
&\leq \bigg\lvert \int_{\mb T^d} \sum_{k\in\mb Z^d}e^{-\frac{\|x - y - 2\pi k\|^2}{2\sigma^2}}R_{T_j}(y)\,\vol(\dd y) - \sum_{s=1}^l w_j^s\sum_{k\in\mb Z^d}e^{-\frac{\|x - x^{j, s} - 2\pi k\|^2}{2\sigma^2}}\bigg\rvert\\
&\qquad + \bigg\lvert \sum_{s=1}^l w_j^s\sum_{k\in\mb Z^d}e^{-\frac{\|x - x^{j, s} - 2\pi k\|^2}{2\sigma^2}} - \sum_{s=1}^l w_j^s\sum_{k\in\mb Z^d}e^{-\frac{\|x - v^{j, s} - 2\pi k\|^2}{2\sigma^2}}\bigg\rvert\\
&\qquad + \bigg\lvert \sum_{s=1}^l w_j^s\sum_{k\in\mb Z^d}e^{-\frac{\|x - v^{j, s} - 2\pi k\|^2}{2\sigma^2}} - \sum_{s=1}^l \beta^{j, s}\sum_{k\in\mb Z^d}e^{-\frac{\|x - v^{j, s} - 2\pi k\|^2}{2\sigma^2}}\bigg\rvert\\
&=: J_1 + J_2 + J_3.
\end{align*}
To control $J_3$, note that
\begin{align*}
J_3 \leq \sum_{s=1}^l|w_j^s - \beta^{j, s}|\sum_{k\in\mb Z^d}e^{-\frac{\|x - v^{j, s}-2\pi_k\|^2}{2\sigma^2}} 
\stackrel{\ri}{\lesssim}\sum_{s=1}^l|w_j^s - \beta^{j, s}|
\stackrel{\rii}{\leq}r_{\Delta^l}.
\end{align*}
Here, (i) is due to the arguments in Section 3.1 of~\citep{berg1976potential}, and (ii) is due to~\eqref{eqn: covering_prop}.
To control $J_2$, we have
\begin{align*}
J_2 &\leq \sum_{s=1}^l w_j^s\big\lvert p_\sigma(x - x^{j, s}) - p_\sigma(x - v^{j, s})\big\rvert
\leq \sum_{s=1}^l w_j^s \|x^{j, s} - v^{j, s}\|_{\ell^2}\cdot\sup_{x\in\mb R^d}\|\nabla p_\sigma(x)\|\\
&\stackrel{\ri}{\lesssim} \sum_{s=1}^l w_j^s\|x^{j, s} - v^{j, s}\|_{\ell^2}
\stackrel{\rii}{\lesssim} r_{\mb T^d}.
\end{align*}
Here, (i) is because $\nabla p_\sigma(x)$ is periodic on $\mb R^d$ and smooth on the compact space $\mb T^d$, indicating that $\|\nabla p_\sigma(x)\|$ is uniformly bounded; (ii) is due to~\eqref{eqn: covering_prop} again. 
To control $J_1$, we need the following lemma, of which the proof is quite involved and thus is postponed to Appendix~\ref{sec: tech_lem}.
\begin{lemma}\label{lem: control_J1}
There is a constant $C_6 = C_6(d, \sigma) > 0$, such that
\begin{align*}
J_1\leq 2\Big(\frac{C_6e^2\log M}{M+1}\Big)^{\frac{M+1}{2}}.
\end{align*}
\end{lemma}
Now, combining all pieces above yields
\begin{align*}
\big\lvert\m K_\sigma\ast R_{T_j}(x) - \m K_\sigma\ast R_j^{\rm rec}(x)\big\rvert
\lesssim \Big(\frac{C_6e^2\log M}{M+1}\Big)^{\frac{M+1}{2}} + r_{\mb T^d} + r_{\Delta^l}.
\end{align*}
The same upper bound also holds for $\big\lvert\m K_\sigma\ast \wt R_{T_j}(x) - \m K_\sigma\ast R_j^{\rm rec}(x)\big\rvert$. Thus, by~\eqref{eqn: gauss_conv_diff_bd} we have that
\begin{align*}
\lvert\m K_\sigma\ast R_t(x) - \m K_\sigma\ast\wt R_t(x)\rvert
\lesssim \Big(\frac{C_6e^2\log M}{M+1}\Big)^{\frac{M+1}{2}} + r_{\mb T^d} + r_{\Delta^l} + \sqrt{t - T_j}
\end{align*}
holds for all $j\in[N_I]$. Taking the minimum of $j\in[N_i]$ implies
\begin{align*}
\lvert\m K_\sigma\ast R_t(x) - \m K_\sigma\ast\wt R_t(x)\rvert
\lesssim \Big(\frac{C_6e^2\log M}{M+1}\Big)^{\frac{M+1}{2}} + r_{\mb T^d} + r_{\Delta^l} + \min_{j\in[N_I]}\sqrt{t - T_j}.
\end{align*}

\vspace{0.5em}
\noindent\underline{Step 3: Bound covering number.} To derive the upper bound of the $\eta$-covering number, we need
\begin{align*}
\Big(\frac{C_6e^2\log M}{M+1}\Big)^{\frac{M+1}{2}} \lesssim \eta,
\quad r_{\mb T^d}\lesssim \eta,
\quad r_{\Delta^l} \lesssim \eta,
\quad\mx{and}\quad \min_{j\in[N_I]}\sqrt{t - T_j}\lesssim\eta.
\end{align*}
This implies $M = O(\log\eta^{-1})$, and $N_I = O\big(\min\{\eta^{-2}, |K|\}\big)$, where $|K|$ is the cardinality of the set $K\subset[0, 1]$. Note that the $r_{\mb T^d}$-covering number of $\mb T^d$ with respect to $\ell^2$-norm is bounded by
\begin{align*}
\log N_{\mb T^d} \leq d\log\Big(1 + \frac{\pi\sqrt{d}}{r_{\mb T^d}}\Big)
\lesssim d\log\frac{d}{\eta},
\end{align*}
and the $r_{\Delta^l}$-covering number of $\Delta^l$ with respect to $\ell^1$-norm is bounded by
\begin{align*}
\log N_{\Delta^l} 
\stackrel{\ri}{\leq} l\log\Big(1 + \frac{2}{r_{\Delta^l}}\Big)
\lesssim l\log\frac{1}{\eta}
\stackrel{\rii}{\lesssim}\Big(\log\frac{1}{\eta}\Big)^{d+1}.
\end{align*}
Here, (i) is derived by Example 5.8 in~\citep{wainwright2019high}, and (ii) is due to the fact that $l = (2M+1)^d + 1 = O\big([\log\eta^{-1}]^d\big)$. Also note that
\begin{align*}
\log\binom{N_{\mb T^d} + l -1}{l}
\leq \log\Big[\frac{e(N_{\mb T^d} + l - 1)}{l}\Big]^l
\lesssim l\log N_{\mb T^d} \lesssim \Big(\log\frac{1}{\eta}\Big)^{d+1}.
\end{align*}
So, the logarithm number of possible outcomes of $\proj(R)$ is at most
\begin{align*}
N_I\log N_{\Delta^l} + N_I\log\binom{N_{\mb T^d}+l-1}{l} \lesssim \min\Big\{\frac{1}{\eta^2}, |K|\Big\}\cdot\Big(\log\frac{1}{\eta}\Big)^{d+1}.
\end{align*}
\end{proof}
\subsection{Proof of Lemma~\ref{lem: tail bound}}\label{app: pf_of_emp_proc}
\underline{Step 1: control the tail of the process}
\begin{align*}
S_{N, m}(r) := &\sup_{R\in\m GR(r)}\bigg|\sum_{j=1}^m\sum_{i=1}^N \frac{t_{j+1}-t_j}{N}\big[\log g_R(t_j, X_{t_j}^i) - \mb E\log g_R(t_j, X_{t_j}^i)\big]\bigg|,
\end{align*}
where $\m GR(r) \subset\ms P(\Omega)$ consists of all path-space distribution $R$ such that
\begin{align*}
\|g_R - g_{R^\ast}\|_{L^2_m}^2 = \sum_{j=1}^m(t_{j+1}-t_j)d_{\rm H}^2\Big(\frac{\m K_\sigma\ast R_{t_j} + \m K_\sigma\ast R^\ast_{t_j}}{2}, \m K_\sigma\ast R_{t_j}\Big)\leq r^2 \quad\mx{and}\quad \tau\KL(R\,\|\,W^\tau)\leq 2E.
\end{align*}
Note that by mean-value theorem, we have
\begin{align}\label{eqn: LipLoss}
\big|\log g_R(t, x) - \log g_{R'}(t, x)\big| &\leq \frac{|g_R(t, x) - g_{R'}(t, x)|}{\min\{g_R(t, x), g_{R'}(t, x)\}}\leq \sqrt{2}\cdot|g_R(t, x) - g_{R'}(t, x)|.
\end{align}
This inequality implies
\begin{align*}
&\quad\,\,\sup_{R\in\m GR(r)}\sum_{j=1}^m\sum_{i=1}^N\mb V\Big(\frac{t_{j+1}-t_j}{N}\big[\log g_R(t_j, X_{t_j}^i) - \mb E\log g_R(t_j, X_{t_j}^i)\big]\Big)\\
&\leq \sup_{R\in\m GR(r)} \sum_{j=1}^m\sum_{i=1}^N \frac{(t_{j+1}-t_j)^2}{N^2}\mb E\big[\log g_R(t_j, X_{t_j}^i)\big]^2\\
&= \sup_{R\in\m GR(r)} \sum_{j=1}^m\sum_{i=1}^N \frac{(t_{j+1}-t_j)^2}{N^2}\mb E\big[\log g_R(t_j, X_{t_j}^i) - \log g_{R^\ast}(t_j, X_{t_j}^i)\big]^2\\
&\stackrel{\ri}{\leq} \sup_{R\in\m GR(r)}\sum_{j=1}^m\sum_{i=1}^N\frac{(t_{j+1} - t_j)^2}{N^2} 2\mb E\big[g_R(t_j, X_{t_j}^i) - 1\big]^2\\
&\stackrel{\rii}{=} \sup_{R\in\m GR(r)}\sum_{j=1}^m\sum_{i=1}^N\frac{2(t_{j+1} - t_j)^2}{N^2} d_{\rm H}^2\Big(\m K_\sigma\ast R_{t_j}^\ast, \frac{\m K_\sigma\ast R_{t_j}^\ast + \m K_\sigma\ast R_{t_j}}{2}\Big)\\
&\stackrel{\riii}{\leq} \frac{2r^2\Delta_m}{N}.
\end{align*}
Here, (i) is due to the inequality~\eqref{eqn: LipLoss}; (ii) is due to Equation~\eqref{eqn: connectH} and $g_{R^\ast} = 1$; (iii) is by the definition of $\m GR(r)$. 
By~\citep[Proposition B.4,][]{lavenant2024toward}, there exists a constant $C_\sigma > 0$ such that
\begin{align*}
C_\sigma^{-1} \leq \sup_{t\in[0, 1], x\in\m X}\m K_\sigma\ast R_t(x) < C_\sigma
\end{align*}
for every $R\in\ms P(\Omega)$. So, we have
\begin{align*}
&\quad\,\,\sup_{i\in[N], j\in[m], R\in\m GR(r)}\Big|\frac{t_{j+1}-t_j}{N}\big[\log g_R(t_j, X_{t_j}^i) - \mb E\log g_R(t_j, X_{t_j}^i)\big]\Big|\\
&\leq \frac{2\|\Delta_m\|}{N} \sup_{t\in[0, 1], x\in\m X}|\log g_R(t, x)|
= \frac{\|\Delta_m\|}{N}\log\Big(\sup_{t\in[0, 1], x\in\m X}\frac{\m K_\sigma \ast R_t(x)}{2\m K_\sigma\ast R_t^\ast(x)} + \frac{1}{2}\Big)\\
&\leq \frac{\|\Delta_m\|}{N}\log\frac{C_\sigma^2+1}{2}.
\end{align*}
Then by Talagrand's inequality~\citep[Theorem 3,][]{massart2000constants}, we have
\begin{align}\label{eqn: talagrand}
\mb P\bigg(S_{N, m}(r) \geq 2\mb ES_{N, m}(r) + 4r\sqrt{\frac{s\|\Delta_m\|}{N}} + 34.5\frac{\|\Delta_m\|s}{N}\log\frac{C_\sigma^2+1}{2}\bigg) \leq e^{-s}.
\end{align}
\underline{Step 2: control $\mb ES_{N, m}$.} Bounding the expectation follows a standard symmetrization argument. To be precise, let $\overline X_{t_j}^i$ be an i.i.d. copy of $X_{t_j}^i$, i.e. they are independent with the same distribution. We have
\begin{align*}
\mb E S_{N,m}(r) 
&= \mb E_X \sup_{R\in\m GR(r)}\bigg|\sum_{j=1}^m\sum_{i=1}^N \frac{t_{j+1}-t_j}{N}\big[\log g_R(t_j, X_{t_j}^i) - \mb E_{\overline X}\log g_R(t_j, \overline X_{t_j}^i)\big]\bigg|\\
&\leq \mb E_{X, \overline X} \sup_{R\in\m GR(r)}\bigg|\sum_{j=1}^m\sum_{i=1}^N \frac{t_{j+1}-t_j}{N}\big[\log g_R(t_j, X_{t_j}^i) - \log g_R(t_j, \overline X_{t_j}^i)\big]\bigg|\\
&= \mb E_{X,\overline X, \epsilon} \sup_{R\in\m GR(r)}\bigg|\sum_{j=1}^m\sum_{i=1}^N \frac{t_{j+1}-t_j}{N}\epsilon_{ij}\big[\log g_R(t_j, X_{t_j}^i) - \log g_R(t_j, \overline X_{t_j}^i)\big]\bigg|\\
&\leq 2\mb E_{X,\epsilon}\sup_{R\in\m GR(r)}\bigg|\sum_{j=1}^m\sum_{i=1}^N\frac{t_{j+1}-t_j}{N}\epsilon_{ij}\log g_R(t_j, X_{t_j}^i)\bigg|,
\end{align*}
where $\{\epsilon_{ij}: i\in[N], j\in[m]\}$ are i.i.d. Rademacher random variables. Then by a modified version of Ledoux--Talagrand's inequality (Proposition~\ref{prop: Ledoux-Talagrand} and Equation~\ref{eqn: Ledoux-Talagrand}), we have
\begin{align*}
&\quad\,\,\mb E_{X,\epsilon}\sup_{R\in\m GR(r)}\bigg|\sum_{j=1}^m\sum_{i=1}^N\frac{t_{j+1}-t_j}{N}\epsilon_{ij}\log g_R(t_j, X_{t_j}^i)\bigg|\\ 
&= \mb E_{X,\epsilon}\sup_{R\in\m GR(r)}\bigg|\sum_{j=1}^m\sum_{i=1}^N\frac{t_{j+1}-t_j}{N}\epsilon_{ij}\big[\log g_R(t_j, X_{t_j}^i) - \log g_{R^\ast}(t_j, X_{t_j}^i)\big]\bigg| \\
&\leq 2\sqrt{2}\mb E_{X,\epsilon}\sup_{R\in\m GR(r)}\bigg|\sum_{j=1}^m\sum_{i=1}^N\frac{t_{j+1}-t_j}{N}\epsilon_{ij} \big[g_R(t_j, X_{t_j}^i)-g_{R^\ast}(t_j, X_{t_j}^i)\big]\bigg|
\end{align*}
due to the Lipschitz property~\eqref{eqn: LipLoss} and the fact that $g_{R^\ast} = 1$. Therefore, we only need to control the expected value on the right-hand side.
For this purpose, define the process
\begin{align*}
Y_R := \sum_{j=1}^m\sum_{i=1}^N \sqrt{\frac{t_{j+1}-t_j}{N}}\epsilon_{ij}g_R(t_j, X_{t_j}^i) .
\end{align*}
It is clear that $Y_R$ is centered, i.e. $\mb E_{X,\epsilon}Y_R = 0$ for every fixed $R\in\m GR(r)$, and previous argument shows that
\begin{align*}
\mb E S_{N, m}(r) \leq 2\sqrt{\frac{2\Delta_m}{N}}\mb E_{X,\varepsilon}\sup_{R\in\m GR(r)}|Y_R - Y_{R^\ast}|.
\end{align*}
Now, we only need to control the right-hand side, which can be decomposed into several steps.

\underline{Step 2.1: sub-exponential increments of $Y_R$ and Bernstein-type bound.} For every $\lambda > 0$ and $R, R'\in\m GR(r)$, we have
\begin{align*}
\mb E_{X,\epsilon} e^{\lambda(Y_R - Y_{R'})}
&= \mb E_{X,\epsilon} \exp\bigg\{\lambda\sum_{j=1}^m\sum_{i=1}^N\sqrt{\frac{t_{j+1}-t_j}{N}}\epsilon_{ij}\big[g_R(t_j, X_{t_j}^i) - g_{R'}(t_j, X_{t_j}^i)\big]\bigg\}\\
&= \prod_{j=1}^m\prod_{i=1}^N \mb E_{X,\epsilon}\exp\bigg\{\epsilon_{ij}\cdot \lambda\sqrt{\frac{t_{j+1}-t_j}{N}}\big[g_R(t_j, X_{t_j}^i) - g_{R'}(t_j, X_{t_j}^i)\big]\bigg\}\\
&\stackrel{}{\leq}\prod_{j=1}^m\prod_{i=1}^N \mb E_{X}\exp\bigg\{\frac{(t_{j+1}-t_j)\lambda^2}{2N} \big[g_R(t_j, X_{t_j}^i) - g_{R'}(t_j, X_{t_j}^i)\big]^2\bigg\}.
\end{align*}
Here, the last inequality is due to the fact that $\mb Ee^{\epsilon s} \leq e^{\frac{s^2}{2}}$ for every $s\in\mb R$. Note that by Taylor's expansion, we have
\begin{align*}
&\quad\,\mb E_{X}\exp\bigg\{\frac{(t_{j+1}-t_j)\lambda^2}{2N} \big[g_R(t_j, X_{t_j}^i) - g_{R'}(t_j, X_{t_j}^i)\big]^2\bigg\}\\
&= 1 + \sum_{l=1}^\infty \frac{1}{l!}\mb E_X\Big[\frac{(t_{j+1}-t_j)\lambda^2}{2N}\big[g_R(t_j, X_{t_j}^i) - g_{R'}(t_j, X_{t_j}^i)\big]^2\Big]^l\\
&\stackrel{\ri}{\leq} 1 + \sum_{l=1}^\infty \Big[\frac{\lambda^2(t_{j+1}-t_j)}{2N}\Big]^l \|g_R - g_{R'}\|_{L_m^\infty}^{2l-2}\cdot\|g_R(t_j,\cdot) - g_{R'}(t_j,\cdot)\|_{L^2(\m K_\sigma\ast R_{t_j}^\ast)}^2\\
&\stackrel{\rii}{=} 1 + \frac{\frac{(t_{j+1}-t_j)\lambda^2}{2N}\|g_{R}(t_j,\cdot) - g_{R'}(t_j, \cdot)\|_{L^2(\m K_\sigma\ast R_{t_j}^\ast)}^2}{1 - \frac{(t_{j+1}-t_j)\lambda^2}{2N}\|g_R - g_{R'}\|_{L_m^\infty}^2}\\
&\stackrel{\riii}{\leq} \exp\bigg\{\frac{\frac{(t_{j+1}-t_j)\lambda^2}{2N}\|g_{R}(t_j,\cdot) - g_{R'}(t_j, \cdot)\|_{L^2(\m K_\sigma\ast R_{t_j}^\ast)}^2}{1 - \sqrt{\frac{\Delta_m}{2N}}\lambda\|g_R - g_{R'}\|_{L_m^\infty}}\bigg\}.
\end{align*}
Here, (i) is due to $l! \geq 1$ and
\begin{align*}
|g_R(t_j, \cdot) - g_{R'}(t_j, \cdot)| \leq \sup_{t\in\{t_1, \dots, t_m\}, x\in\m X} |g_R(t, x) - g_{R'}(t, x)| = \|g_R - g_{R'}\|_{L_m^\infty};
\end{align*}
(ii) holds when $\lambda < \|g_{R} - g_{R'}\|_{L_m^\infty}^{-1}\cdot\sqrt{\frac{2N}{\Delta_m}}$; (iii) is due to the fact that $1+s \leq e^s$ for all $s\geq 0$, and that $1 - s \geq 1-\sqrt{s}$ for all $s\in[0, 1]$. Thus, we have shown that
\begin{align*}
\mb E_{X,\epsilon}e^{\lambda(Y_R - Y_{R'})}
&\leq \prod_{j=1}^m\prod_{i=1}^N \exp\bigg\{\frac{\frac{(t_{j+1}-t_j)\lambda^2}{2N}\|g_{R}(t_j,\cdot) - g_{R'}(t_j, \cdot)\|_{L^2(\m K_\sigma\ast R_{t_j}^\ast)}^2}{1 - \sqrt{\frac{\Delta_m}{2N}}\lambda\|g_R - g_{R'}\|_{L_m^\infty}}\bigg\}\\
&= \exp\bigg\{\frac{\lambda^2\|g_R - g_{R'}\|_{L^2_m}^2/2}{1 - \sqrt{\frac{\Delta_m}{2N}}\lambda\|g_R - g_{R'}\|_{L_m^\infty}}\bigg\}
\end{align*}
for all $\lambda < \|g_{R} - g_{R'}\|_{L_m^\infty}^{-1}\cdot\sqrt{\frac{2N}{\Delta_m}}$. Then, by~\citep[Proposition 2.10,][]{wainwright2019high}, we have the Bernstein-type tail bound
\begin{align}\label{eqn: Berstein}
\mb P\big(|Y_R - Y_{R'}| \geq s\big) \leq 2\exp\bigg\{-\frac{1}{2}\cdot\frac{s^2}{\|g_R - g_{R'}\|_{L_m^2}^2 + s \sqrt{\frac{\Delta_m}{2N}}\|g_R - g_{R'}\|_{L_m^\infty}}\bigg\}.
\end{align}

\underline{Step 2.2: $\psi_1$-chaining for maximal inequality.} Let $N_2(s; \m GR(r)) := N_2\big(s, \{g_R - g_{R^\ast}: R\in\m GR(r)\}\big)$ be the cardinality of the smallest set $G\subset\m GR(r)$, such that for every $R\in\m GR(r)$, there exists a $R'\in G$ satisfying
\begin{align*}
\|g_R - g_{R'}\|_{L^2_m} \leq sr
\quad \mx{and}\quad 
\|g_R - g_{R'}\|_{L^\infty_m} \leq s\cdot\sqrt{2}C_\sigma.
\end{align*}
Now, for every $k\in \mb N$, let $G_k\subset\m GR(r)$ such that its cardinality $|G_k| = N_2(2^{-k}, \m GR(r))$; we specify $G_0 = \{R^\ast\}$, which is possible when $\tau\KL(R^\ast\,\|\,W^\tau) \leq u$.
Furthermore, define $\pi_k(R)\in G_k$ such that
\begin{align*}
\|g_R - g_{\pi_k(R)}\|_{L_m^2} \leq 2^{-k}r
\quad\mx{and}\quad
\|g_R - g_{\pi_k(R)}\|_{L_m^\infty} \leq 2^{-k}\cdot \sqrt{2}C_\sigma. 
\end{align*}
Now, for a fixed $K\in\mb Z_+$ to be decided later and $R\in\m GR(r)$, define $R^K = \pi_K(R)$, and $R^{k-1} = \pi_{k-1}(R^k)$ for $k = K, K-1, \dots, 0$. Then, we have
\begin{align*}
\sup_{R\in\m  GR(r, u)}|Y_R - Y_{R^\ast}|
&\leq\sup_{R\in\m GR(r)}|Y_R - Y_{R^K}| + \sup_{R\in\m GR(r)}|Y_{R^K} - Y_{R^0}|\\ 
&\leq \sup_{R\in\m GR(r)}|Y_R - Y_{R^K}| + \sum_{k=1}^K \sup_{R\in\m GR(r)}\big|Y_{R^k} - Y_{R^{k-1}}\big|\\ 
&\leq \sup_{R\in\m GR(r)}|Y_R - Y_{R^K}| + \sum_{k=1}^K \sup_{R\in G_k}\big|Y_R - Y_{\pi_{k-1}(R)}\big|\\
&\leq \sqrt{N}\sum_{j=1}^m\sqrt{t_{j+1}-t_j}\cdot 2^{-K}\sqrt{2}C_\sigma + \sum_{k=1}^K \sup_{R\in G_k}\big|Y_R - Y_{\pi_{k-1}(R)}\big|.
\end{align*}
Since this inequality holds for all $K\in\mb N^\ast$, taking $K\to\infty$ implies
\begin{align*}
\sup_{R\in\m  GR(r, u)}|Y_R - Y_{R^\ast}| \leq \sum_{k=1}^\infty \sup_{R\in G_k}\big|Y_R - Y_{\pi_{k-1}(R)}\big|.
\end{align*}
Therefore, we have
\begin{align*}
&\quad\,\mb E\sup_{R\in\m  GR(r, u)}|Y_R - Y_{R^\ast}| 
\leq \mb E\sum_{k=1}^\infty \sup_{R\in G_k}\big|Y_R - Y_{\pi_{k-1}(R)}\big|\\
&\stackrel{\ri}{\leq} \Big\|\sum_{k=1}^\infty \sup_{R\in G_k}\big|Y_R - Y_{\pi_{k-1}(R)}\big|\Big\|_{\psi_1}
\leq \sum_{k=1}^\infty \Big\|\sup_{R\in G_k}\big|Y_R - Y_{\pi_{k-1}(R)}\big|\Big\|_{\psi_1}\\
&\stackrel{\rii}{\leq} \sum_{k=1}^\infty C_{\psi_1}\bigg[\sqrt{\frac{\Delta_m}{2N}}2^{-(k-1)}\cdot\sqrt{2}C_\sigma\log(1+|G_k|) + 2^{-(k-1)}r\cdot\sqrt{\log(1+|G_k|)}\bigg]\\
&= 4C_{\psi_1}\sum_{k=1}^\infty\int_{2^{-k-1}}^{2^{-k}}\bigg[\sqrt{\frac{\Delta_m}{N}}C_\sigma\log\big[1 + N_2\big(2^{-k};\m GR(r)\big)\big] + r\sqrt{\log\big[1 + N_2\big(2^{-k};\m GR(r)\big)\big]}\bigg]\,\dd s\\
&\leq4C_{\psi_1}\int_{0}^{\frac{1}{2}}\bigg[\sqrt{\frac{\Delta_m}{N}}C_\sigma\log\big[1 + N_2\big(s;\m GR(r)\big)\big] + r\sqrt{\log\big[1 + N_2\big(s;\m GR(r)\big)\big]}\bigg]\,\dd s.
\end{align*}
Here, (i) is due to the inequality~\eqref{eqn: E-psi-bound}; (ii) is by the Bernstein-type bound~\eqref{eqn: Berstein} and~\citep[Lemma 8.3,][]{kosorok2008introduction}.

\underline{Step 2.3: control the covering number $N_2(s; \m GR(r))$.} Note that for every $R, R'\in\m GR(r)$,
\begin{align*}
\|g_R - g_{R'}\|_{L_m^2}^2 
&\leq \sum_{j=1}^m (t_{j+1}-t_j)\|g_R - g_{R'}\|_{L_m^\infty}^2\\
&\leq \|g_R - g_{R'}\|_{L_m^\infty}^2
\leq \frac{1}{2}\big\|\m K_\sigma\ast R - \m K_\sigma\ast R'\big\|_{L_m^\infty}^2.
\end{align*}
Now, let us prove that
\begin{align*}
N_2(s; \m GR(r))\leq N\big(sr/\sqrt{2},\big\{\m K_\sigma\ast R: R\in\ms P(\Omega), \tau\KL(R\,\|\,W^\tau)\leq 2E\big\} , \|\cdot\|_{L_m^\infty}\big).
\end{align*}
In fact, assume $\m K_\sigma\ast R^1, \dots, \m K_\sigma\ast R^{N_1}$ is a $\frac{sr}{\sqrt{2}}$-covering of the set $\big\{\m K_\sigma\ast R: R\in\ms P(\Omega), \tau\KL(R\,\|\,W^\tau)\leq 2E\big\}$ with respect to the $\|\cdot\|_{L_m^\infty}$-norm. Then, for every $R\in\m GR(r)$, there exists $j\in[N_1]$ such that
\begin{align*}
\|g_{R} - g_{R^j}\|_{L_m^2} \leq \|g_R - g_{R^j}\|_{L_m^\infty} \leq \frac{\|\m K_\sigma\ast R - \m K_\sigma\ast R^j\|_{L_m^\infty}}{\sqrt{2}} \leq \frac{sr}{2}.
\end{align*}
Now, let $S\subset\{1, \dots, N_1\}$ be the subset, such that every $j\in S$ if and only if there exists $\wt R^j\in\m GR(r)$ satisfying $\|\m K_\sigma\ast R^j - \m K_\sigma\ast \wt R^j\|_{L_m^\infty} \leq \frac{sr}{\sqrt{2}}$. Clearly, $\{\wt R^j: j\in S\}$ is a $sr$-covering of $\m GR(r)$, since for every $R\in\m GR(r)$
\begin{align*}
\|g_R - g_{\wt R^j}\|_{L_m^\infty} \leq \frac{\|\m K_\sigma R - \m K_\sigma R^j\|_{L_m^\infty}}{\sqrt{2}} + \frac{\|\m K_\sigma \wt R^j - \m K_\sigma R^j\|_{L_m^\infty}}{\sqrt{2}}
\leq sr.
\end{align*}
Therefore, by applying Proposition~\ref{prop: covering_num}, we have
\begin{align*}
N_2\big(s; \m GR(r)\big) 
\lesssim \min\big\{2(sr)^{-2}, m\big\}\cdot\Big(\log\frac{2}{sr}\Big)^{d+1}.
\end{align*}

\underline{Step 2.4: evaluate the upper bound of maximal inequality.} Now, we have
\begin{align}\label{eqn: maximal_ineq}
\begin{aligned}
\mb E\sup_{R\in\m GR(r)}|Y_R - Y_{R^\ast}|
&\lesssim \sqrt{\frac{\Delta_m}{N}}\int_{0}^{\frac{1}{2}}\min\big\{2(sr)^{-2}, m\big\}\cdot\Big(\log\frac{2}{sr}\Big)^{d+1}\,\dd s \\
&\qquad\qquad\qquad + r\int_{0}^{\frac{1}{2}}\sqrt{\min\big\{2(sr)^{-2}, m\big\}\cdot\Big(\log\frac{2}{sr}\Big)^{d+1}}\,\dd s.
\end{aligned}
\end{align}
Now, we bound two integrals separately. 
To begin with, first note that 
\begin{align*}
2(sr)^{-2} \leq m
\,\,\Longleftrightarrow\,\,
s \geq \frac{1}{r}\sqrt{\frac{2}{m}}.
\end{align*}
Furthermore, this threshold is smaller than $\frac{1}{2}$ when
\begin{align*}
\frac{1}{r}\sqrt{\frac{2}{m}} \leq \frac{1}{2}
\,\,\Longleftrightarrow\,\,
r \geq \frac{2\sqrt{2}}{\sqrt{m}}
\end{align*}

\underline{\underline{Case 2.4.1: $r \geq \frac{2\sqrt{2}}{\sqrt{m}}$.}} For the first term in~\eqref{eqn: maximal_ineq}, note that
\begin{align*}
&\quad\,\int_{0}^{\frac{1}{2}}\min\big\{2(sr)^{-2}, m\big\}\cdot\Big(\log\frac{2}{sr}\Big)^{d+1}\,\dd s\\
&=\int_0^{\frac{1}{r}\sqrt{\frac{2}{m}}} m \Big(\log\frac{2}{sr}\Big)^{d+1}\,\dd s + 2\int_{\frac{1}{r}\sqrt{\frac{2}{m}}}^{\frac{1}{2}}\Big(\frac{1}{sr}\Big)^2\Big(\log\frac{2}{sr}\Big)^{d+1}\,\dd s\\
&=: J_{11} + J_{12}.
\end{align*}
By using the change of variable formula with $s = \frac{1}{vr}\sqrt{\frac{2}{m}}$, we have
\begin{align*}
J_{11} &= \frac{\sqrt{2m}}{r}\int_1^\infty \big(\log\sqrt{2m} + \log v\big)^{d+1} v^{-2}\,\dd v\\
&\leq \frac{\sqrt{2m}}{r}\cdot 2^d\int_1^\infty \Big[\big(\log\sqrt{2m}\big)^{d+1} + \big(\log v\big)^{d+1}\Big] v^{-2}\,\dd v\\
&\lesssim \frac{\sqrt{m}(\log m)^{d+1}}{r}.
\end{align*}
In the last line, a finite constant only depending on the integral is omitted. Similarly, using the change of variable formula with $s = \frac{v}{r}\sqrt{\frac{2}{m}}$ yields
\begin{align*}
J_{12} &= \frac{\sqrt{2m}}{r}\int_1^{\frac{r}{2}\sqrt{\frac{m}{2}}} \Big(\log\frac{\sqrt{2m}}{v}\Big)^{d+1}v^{-2}\,\dd v\\
&\leq \frac{\sqrt{2m}}{r}\int_1^\infty \big(\log\sqrt{2m}\big)^{d+1}v^{-2}\,\dd v\\
&\lesssim \frac{\sqrt{m}(\log m)^{d+1}}{r}.
\end{align*}
Now, we have
\begin{align*}
\sqrt{\frac{\Delta_m}{N}}\int_{0}^{\frac{1}{2}}\min\big\{2(sr)^{-2}, m\big\}\cdot\Big(\log\frac{2}{sr}\Big)^{d+1}\,\dd s 
\lesssim \sqrt{\frac{\Delta_m}{N}}\cdot\frac{\sqrt{m}(\log m)^{d+1}}{r}.
\end{align*}
To bound the second term in~\eqref{eqn: maximal_ineq}, we have
\begin{align*}
&\quad\,\int_{0}^{\frac{1}{2}}\sqrt{\min\big\{2(sr)^{-2}, m\big\}\cdot\Big(\log\frac{2}{sr}\Big)^{d+1}}\,\dd s\\
&=\int_0^{\frac{1}{r}\sqrt{\frac{2}{m}}}  \sqrt{m}\Big(\log\frac{2}{sr}\Big)^{\frac{d+1}{2}}\,\dd s + \sqrt{2}\int_{\frac{1}{r}\sqrt{\frac{2}{m}}}^{\frac{1}{2}}\frac{1}{sr}\cdot\Big(\log\frac{2}{sr}\Big)^{\frac{d+1}{2}}\,\dd s\\
&=: J_{21} + J_{22}.
\end{align*}
To estimate $J_{21}$, using the change of variable formula with $s = \frac{1}{vr}\sqrt{\frac{2}{m}}$ yields
\begin{align*}
J_{21} &= \frac{\sqrt{2}}{r}\int_1^{\infty}\big(\log\sqrt{2m}+\log v\big)^{\frac{d+1}{2}}v^{-2}\,\dd v
\lesssim \frac{(\log m)^{\frac{d+1}{2}}}{r}.
\end{align*}
For $J_{22}$, the integral can be calculated explicitly
\begin{align*}
J_{22} = \frac{2\sqrt{2}}{r(d+3)}\Big[\big(\log\sqrt{2m}\big)^{\frac{d+3}{2}} - \big(\log 4r^{-1}\big)^{\frac{d+3}{2}}\Big]
\lesssim \frac{(\log m)^{\frac{d+3}{2}}}{r}.
\end{align*}
Therefore, we have
\begin{align*}
r\int_{0}^{\frac{1}{2}}\sqrt{\min\big\{2(sr)^{-2}, m\big\}\cdot\Big(\log\frac{2}{sr}\Big)^{d+1}}\,\dd s 
\lesssim (\log m)^{\frac{d+3}{2}}.
\end{align*}
Combining with the maximal inequality~\eqref{eqn: maximal_ineq}, we have
\begin{align*}
\mb E\sup_{R\in\m GR(r)}\big|Y_R - Y_{R^\ast}\big| 
&\lesssim \sqrt{\frac{\Delta_m}{N}}\cdot\frac{\sqrt{m}(\log m)^{d+1}}{r} + (\log m)^\frac{d+3}{2}\\
&\leq \Big[\sqrt{\frac{\Delta_m}{N}}\cdot\frac{\sqrt{m}}{r} + 1\Big](\log m)^{d+1}
\end{align*}

\underline{\underline{Case 2.4.2: $r < \frac{2\sqrt{2}}{\sqrt{m}}$.}} In this case, from the inequality~\eqref{eqn: maximal_ineq} we have
\begin{align*}
\mb E\sup_{R\in\m GR(r)}\big|Y_R - Y_{R^\ast}\big| 
&\lesssim m\sqrt{\frac{\Delta_m}{N}}\int_0^\frac{1}{2}\Big(\log\frac{2}{sr}\Big)^{d+1}\,\dd s + r\sqrt{m}\int_0^\frac{1}{2}\Big(\log\frac{2}{sr}\Big)^{\frac{d+1}{2}}\,\dd s\\
&\lesssim m\sqrt{\frac{\Delta_m}{N}}\Big(\log\frac{1}{r}\Big)^{d+1} + \sqrt{m}r\Big(\log\frac{1}{r}\Big)^\frac{d+1}{2}\\
&\leq \Big[\sqrt{\frac{\Delta_m}{N}}\cdot\frac{\sqrt{m}}{r} + 1\Big] \sqrt{m} r\Big(\log\frac{1}{r}\Big)^{d+1},
\end{align*}
where the second last inequality is derived similarly as in the case 3.4.1.

To sum up, we have shown that there is a universal constant $C_{\rm MI} > 0$ such that
\begin{align*}
\mb E\sup_{R\in\m GR(r)}\big|Y_R - Y_{R^\ast}\big| 
\leq C_{\rm MI}\Big[\sqrt{\frac{\Delta_m}{N}}\cdot\frac{\sqrt{m}}{r} + 1\Big]\cdot\min\{r\sqrt{m}, 2\sqrt{2}\}\cdot\big[\max\big\{\log r^{-1}, \log m\big\}\big]^{d+1}.
\end{align*}
Thus, we have
\begin{align}\label{eqn: emp_pro_bound}
\mb E S_{N, m}(r) \leq 4\sqrt{\frac{2\Delta_m}{N}}\cdot C_{\rm MI}\Big[\sqrt{\frac{\Delta_m}{N}}\cdot\frac{\sqrt{m}}{r} + 1\Big]\cdot\min\{r\sqrt{m}, 2\sqrt{2}\}\cdot\big[\max\big\{\log r^{-1}, \log m\big\}\big]^{d+1}.
\end{align}
\underline{Step 3: high probability bound of $S_{N, m}$.} For simplicity, let $J_{N, m}(r)$ be the right-hand side in~\eqref{eqn: emp_pro_bound}. It is obvious that $J_{N, m}(r) / r$ is non-increasing with respect to $r$ for any fixed $u$. Therefore, there exists $\delta_{N, m}$ such that
\begin{align*}
\frac{J_{N,m}(\delta_{N, m})}{\delta_{N, m}} \leq \delta_{N, m}.
\end{align*}
Now, for any $r \geq \delta_{N, m}$, we have
\begin{align*}
\frac{J_{N, m}(r)}{r} \leq \frac{J_{N, m}(\delta_{N, m})}{\delta_{N, m}} \leq \delta_{N, m},
\end{align*}
indicating that $J_{N, m}(r) \leq r\delta_{N, m}$. So, for every $r \geq \delta_{N, m}$, we have
\begin{align*}
\mb P\bigg(S_{N, m}(r) \geq 2r\delta_{N, m} + 4r\sqrt{\frac{s\|\Delta_m\|}{N}} + 34.5\frac{\|\Delta_m\|s}{N}\log\frac{C_\sigma^2+1}{2}\bigg) \leq e^{-s}.
\end{align*}
by Talagrand's inequality~\eqref{eqn: talagrand}.
For simplicity, define $C_{\rm HP} := 12 + 34.5\log\frac{C_\sigma^2+1}{2}$ and the event
\begin{align*}
\ms A \coloneqq\bigg\{\sup_{R\in\m GR(\infty)} \frac{\big|\sum_{j=1}^m\sum_{i=1}^N\frac{t_{j+1}-t_j}{N}[\log g_R(t_j, X_{t_j}^i) - \mb E\log g_R(t_j, X_{t_j}^i)]\big|}{\delta_{N, m} + \|g_R - g_{R^\ast}\|_{L_m^2}} \leq C_{\rm HP}\delta_{N, m}\bigg\}.
\end{align*}
We will prove that $\ms A$ holds with high probability.
To show this, define two events
\begin{align*}
\ms A_1 &\coloneqq \big\{S_{N, m}(\delta_{N, m}) \geq C_{\rm HP}\delta_{N, m}^2\big\}\\
\ms A_2 &\coloneqq \bigg\{\exists\, R\in\m GR(\infty), \,\,\mx{s.t.}\,\,
\|g_R - g_{R^\ast}\|_{L_m^2} \geq \delta_{N, m}
\quad\mx{and}\\
&\qquad\qquad\qquad
\Big|\sum_{j=1}^m\sum_{i=1}^N\frac{t_{j+1}-t_j}{N}[\log g_R(t_j, X_{t_j}^i) - \mb E\log g_R(t_j, X_{t_j}^i)]\Big| \geq C_{\rm HP}\delta_{N, m}\|g_R - g_{R^\ast}\|_{L_m^2}\bigg\}.
\end{align*}
It is obvious that $\ms A^c \subset (\ms A_1 \cup\ms A_2)$. To bound $\mb P(\ms A_1)$, simply taking $r = \delta_{N, m}$ and $s = \frac{N\delta_{N, m}^2}{\Delta_m}$ yields
\begin{align*}
\mb P(\ms A_1) \leq \mb P\Big(S_{N, m}(\delta_{N, m}) \geq \Big[6 + 34.5\log\frac{C_\sigma^2+1}{2}\Big]\delta_{N, m}^2\Big) \leq e^{-\frac{N\delta_{N, m}^2}{\Delta_m}}.
\end{align*}
To bound $\mb P(\ms A_2)$, we use the peeling technique by further decomposing $\ms A_2$ into more events. Let
\begin{align*}
\ms A_{2k} &\coloneqq \bigg\{\exists\, R\in\m GR(\infty), \,\,\mx{s.t.}\,\,
2^{k-1}\delta_{N,m}\leq \|g_R - g_{R^\ast}\|_{L_m^2} \leq 2^k\delta_{N, m}
\quad\mx{and}\\
&\qquad\qquad\qquad
\Big|\sum_{j=1}^m\sum_{i=1}^N\frac{t_{j+1}-t_j}{N}[\log g_R(t_j, X_{t_j}^i) - \mb E\log g_R(t_j, X_{t_j}^i)]\Big| \geq C_{\rm HP}\delta_{N, m}\|g_R - g_{R^\ast}\|_{L_m^2}\bigg\}\\
&\subset \bigg\{\exists\, R\in\m GR(2^k\delta_{N,m}), \,\,\mx{s.t.}\,\,
\Big|\sum_{j=1}^m\sum_{i=1}^N\frac{t_{j+1}-t_j}{N}[\log g_R(t_j, X_{t_j}^i) - \mb E\log g_R(t_j, X_{t_j}^i)]\Big| \geq C_{\rm HP}2^{k-1}\delta_{N, m}^2\bigg\}\\
&= \big\{S_{N, m}(2^k\delta_{N, m}) \geq C_{\rm HP}2^{k-1}\delta_{N, m}^2\big\}.
\end{align*}
Note that for any $R\in\ms P(\Omega)$, it holds that
\begin{align*}
\|g_R - g_{R^\ast}\|_{L_m^2} \leq \|g_R - g_{R^\ast}\|_{L_m^\infty}\leq\sqrt{\frac{C_\sigma^2+1}{2}} - \sqrt{\frac{1}{2}}.
\end{align*}
Therefore, by letting $K = \Big\lceil \log\frac{\sqrt{C_\sigma^2 + 1}-1}{\sqrt{2}\delta_{N, m}}\Big\rceil$, we have $\ms A_2\subset (\ms A_{21}\cup\dots\cup\ms A_{2K})$. Thus,
\begin{align*}
\mb P(\ms A_2) 
&\leq \sum_{k=1}^K \mb P(\ms A_{2k})
\leq \sum_{k=1}^K \mb P\big(S_{N, m}(2^k\delta_{N, m}, u) \geq C_{\rm HP}2^{k-1}\delta_{N, m}^2\big)\\
&\stackrel{\ri}{\leq} K e^{-\frac{N\delta_{N, m}^2}{\Delta_m}}
= e^{-\frac{N\delta_{N,m}^2}{\Delta_m} + \log K} 
\stackrel{\rii}{\leq}e^{-\frac{N\delta_{N, m}^2}{2\Delta_m}}.
\end{align*}
Here (i) is derived by taking $s = \frac{N\delta_{N, m}^2}{\Delta_m}$ and the fact that $12 + 2^{1-k}\cdot34.5\log\frac{C_\sigma^2+1}{2} \leq C_{\rm HP}$ for every $k\geq 1$; (ii) is due to $\log\big(1+\log\frac{\sqrt{C_\sigma^2+1}-1}{\sqrt{2}\delta_{N,m}^2}\big) \leq \frac{N\delta_{N,m}^2}{2\Delta_m}$ when at least one of $m$ and $N$ is large enough.
Combining all the above pieces yields
\begin{align*}
\mb P(\ms A^c) \leq 2e^{-\frac{N\delta_{N, m}^2}{2\Delta_m}}.
\end{align*}
\subsection{Proof of Theorem~\ref{coro: cts_stat_rate}}\label{app: coro_pf}
Note that
\begin{align*}
&\quad\,\bigg|\sum_{j=1}^m(t_{j+1}-t_j)d_{\rm H}^2\big(\m K_\sigma\ast \wht R_{t_j}, \m K_\sigma\ast R_{t_j}^\ast\big) - \int_0^1 d_{\rm H}^2\big(\m K_\sigma\ast \wht R_{t}, \m K_\sigma\ast R_{t}^\ast\big)\,\dd t\bigg|\\
&\leq \int_0^{t_1} d_{\rm H}^2\big(\m K_\sigma\ast \wht R_{t}, \m K_\sigma\ast R_{t}^\ast\big)\,\dd t + \sum_{j=1}^m \int_{t_j}^{t_{j+1}}\bigg|d_{\rm H}^2\big(\m K_\sigma\ast \wht R_{t}, \m K_\sigma\ast R_{t}^\ast\big)\,\dd t - d_{\rm H}^2\big(\m K_\sigma\ast \wht R_{t_j}, \m K_\sigma\ast R_{t_j}^\ast\big)\bigg|\,\dd t.
\end{align*}
The first term is upper bounded by $2t_1$.  To bound the second term, when $\tau\KL(\wht R\,\|\,W^\tau) \leq 2\tau\KL(R^\ast\,\|\,W^\tau)$, there are universal constants $C_{\rm Hol} > 0$ such that
\begin{align*}
\big|\m K_\sigma\ast \wht R_t(x) - \m K_\sigma\ast \wht R_{t'}(x)\big| \leq C_{\rm Hol}\sqrt{|t - t'|}
\quad\mx{and}\quad
\big|\m K_\sigma\ast  R_t^\ast(x) - \m K_\sigma\ast  R_{t'}^\ast(x)\big| \leq C_{\rm Hol}\sqrt{|t - t'|}
\end{align*}
for all $t, t'\in[0, 1]$~\citep[Proposition 2.12,][]{lavenant2024toward}. Therefore, applying Lemma~\ref{lem: diff_H2} implies
\begin{align*}
\big|d_{\rm H}^2\big(\m K_\sigma\ast \wht R_{t}, \m K_\sigma\ast R_{t}^\ast\big) &- d_{\rm H}^2\big(\m K_\sigma\ast \wht R_{t'}, \m K_\sigma\ast R_{t'}^\ast\big)\big|\\
& \leq 2\sqrt{C_\sigma \vol(\m X)}d_{\rm H}(\m K_\sigma\ast \wht R_t, \m K_\sigma\ast R_t^\ast)C_{\rm Hol}\sqrt{|t-t'|} + 2C_\sigma\vol(\m X)C_{\rm Hol}^2|t-t'|.
\end{align*}
This implies
\begin{align*}
&\quad\,\bigg|\sum_{j=1}^m(t_{j+1}-t_j)d_{\rm H}^2\big(\m K_\sigma\ast \wht R_{t_j}, \m K_\sigma\ast R_{t_j}^\ast\big) - \int_0^1 d_{\rm H}^2\big(\m K_\sigma\ast \wht R_{t}, \m K_\sigma\ast R_{t}^\ast\big)\,\dd t\bigg|\\
&\leq 2t_1 + \sum_{j=1}^m \int_{t_j}^{t_{j+1}}2\sqrt{C_\sigma \vol(\m X)}d_{\rm H}(\m K_\sigma\ast \wht R_{t}, \m K_\sigma\ast R_{t}^\ast)C_{\rm Hol}\sqrt{|t-t_j|} + 2C_\sigma\vol(\m X)C_{\rm Hol}^2|t-t_j|\,\dd t\\
&=2t_1 + 
2C_{\rm Hol}\sqrt{C_\sigma\vol(\m X)}\sum_{j=1}^m\int_{t_j}^{t_{j+1}}\sqrt{t - t_j}d_{\rm H}(\m K_\sigma\ast \wht R_{t}, \m K_\sigma\ast R_{t}^\ast)\,\dd t + 
C_\sigma\vol(\m X)C_{\rm Hol}^2\sum_{j=1}^m(t_{j+1}-t_j)^2\\
&\leq 2\Delta_m + 2C_{\rm Hol}\sqrt{C_\sigma\vol(\m X)}\cdot\sqrt{\frac{m\Delta_m^2}{2}\int_0^1 d_{\rm H}^2(\m K_\sigma\ast\wht R_t, \m K_\sigma\ast R_t^\ast)\,\dd t}
+C_\sigma\vol(\m X)C_{\rm Hol}^2\Delta_m.
\end{align*}
Therefore, we have
\begin{align*}
\int_0^1 d_{\rm H^2}(\m K_\sigma\ast\wht R_t, \m K_\sigma\ast R_t^\ast)\,\dd t
&\leq 2\sum_{j=1}^m(t_{j+1}-t_j)d_{\rm H}^2\big(\m K_\sigma\ast \wht R_{t_j}, \m K_\sigma\ast R_{t_j}^\ast\big)\\
&\qquad\quad  + 2\big[2+C_\sigma\vol(\m X)C_{\rm Hol}^2\big]\Delta_m + 2C_\sigma\vol(\m X)C_{\rm Hol}^2m\Delta_m^2.
\end{align*}
When $\Delta_m = O(m^{-1})$, under the event $\ms A$, we have
\begin{align*}
\int_0^1 d_{\rm H}^2(\m K_\sigma\ast\wht R_t, \m K_\sigma\ast R_t^\ast)\,\dd t
\lesssim \max\Big\{\delta_{N, m}^2, \frac{1}{m}\Big\}.
\end{align*}
\section{Proof of Algoithmic Convergence}\label{sec: pf_algo_converge}
In this section, we focus on the proof of algorithmic convergence of the exact CKLGD algorithm proposed in Section~\ref{sec: exact CKLGD} and the inexact CKLGD algorithm proposed in Section~\ref{sec: inexact CKLGD}.
\subsection{Proof of Theorem~\ref{thm: algo_convergence_general}}
First-order optimality condition (FOC) implies that 
\begin{align}\label{eqn: exact_FOC}
    \frac{\delta\m F}{\delta\rho_j}(\rho^{k-1})(\theta_j) + \frac{1}{\eta_k}\log\frac{\rho_j^{k}}{\rho_j^{k-1}}(\theta_j) = C_j^k
\end{align}
is a constant independent of $\theta_j$. Therefore, for any $\rho\in\ms P_2(\m X)^{\otimes m}$, applying the convexity of $\m F$ yields
\begin{align*}
\m F(\rho^{k-1}) - \m F(\rho) 
&\leq \sum_{j=1}^m \int_{\m X}\frac{\delta\m F}{\delta\rho_j}(\rho^{k-1})(\theta_j)\,\dd[\rho_j^{k-1} - \rho_j]
= \sum_{j=1}^m \int_{\m X} -\frac{1}{\eta_k}\log\frac{\rho_j^k}{\rho_j^{k-1}}(\theta_j)\,\dd[\rho_j^{k-1} - \rho_j]\\
&= \sum_{j=1}^m\frac{1}{\eta_k}\big[\KL(\rho_j^{k-1}\,\|\,\rho_j^{k}) + \KL(\rho_j\,\|\,\rho_j^{k-1}) - \KL(\rho_j\,\|\,\rho_j^k)\big].
\end{align*}
Note that
\begin{align*}
&\KL(\rho_j^{k-1}\,\|\,\rho_j^k) + \KL(\rho_j^k\,\|\,\rho_j^{k-1})
= \int\log\frac{\rho_j^{k-1}}{\rho_j^k}\,\dd[\rho_j^{k-1}-\rho_j^k] \\
&\stackrel{\ri}{=} \eta_k\int \frac{\delta\m F}{\delta\rho_j}(\rho^{k-1})(\theta_j) - C_j^k \,\dd[\rho_j^{k-1} - \rho_j^k]
\stackrel{\rii}{=} \eta_k\int \frac{\delta\m F}{\delta\rho_j}(\rho^{k-1})(\theta_j) \,\dd[\rho_j^{k-1} - \rho_j^k]\\
&\stackrel{\riii}{\leq} \eta_kL_j\|\rho_j^{k-1} - \rho_j^k\|_{L^1}.
\end{align*}
Here, (i) is by FOC~\eqref{eqn: exact_FOC}; (ii) is by the fact that $C_j^k$ is a constant; (iii) is due to the uniform bound of first variation of $\m F$. Thus, we have
\begin{align*}
\m F(\rho^{k-1}) - \m F(\rho)
&\leq \frac{1}{\eta_k}\sum_{j=1}^m \Big[\eta_k L_j\|\rho_j^{k-1} - \rho_j^k\|_{L^1} - \KL(\rho_j^k\,\|\,\rho_j^{k-1}) + \KL(\rho_j\,\|\,\rho_j^{k-1}) - \KL(\rho_j\,\|\,\rho_j^k)\Big]\\
&\stackrel{\ri}{\leq} \sum_{j=1}^m \Big[L_j\|\rho_j^{k-1} - \rho_j^k\|_{L^1} - \frac{1}{2\eta_k}\|\rho_j^{k-1} - \rho_j^k\|_{L^1}^2\Big] + \frac{1}{\eta_k}\sum_{j=1}^m \big[\KL(\rho_j\,\|\,\rho_j^{k-1}) - \KL(\rho_j\,\|\,\rho_j^k)\big]\\
&\leq \sum_{j=1}^m \frac{\eta_kL_j^2}{2} + \frac{1}{\eta_k}\sum_{j=1}^m \big[\KL(\rho_j\,\|\,\rho_j^{k-1}) - \KL(\rho_j\,\|\,\rho_j^k)\big].
\end{align*}
Here, (i) is by Pinsker's inequality. Summing up the above inequality from $k=1$ to $K$ yields
\begin{align*}
\Big(\sum_{k=1}^K \eta_k\Big) \Big[\max_{0\leq k\leq K-1}\m F(\rho^k) - \m F(\rho)\Big]
&\leq \sum_{k=1}^{K} \eta_k\big[\m F(\rho^{k-1}) - \m F(\rho)\big]\\
&\leq \frac{1}{2}\Big(\sum_{k=1}^K\eta_k^2\Big) \Big(\sum_{j=1}^m L_j^2\Big) + \KL(\rho\,\|\,\rho^0) - \KL(\rho\,\|\,\rho^K),
\end{align*}
which implies the desired result.

\subsection{Proof of Theorem~\ref{thm: traj_CKLGD}}\label{app: pf_algo_conv}
Recall that $\wt\rho^k$ is the exact solution of each iterate defined through~\eqref{eqn: explicit_quad_anneal_update}. For any $\rho\in\ms P(\m X)^{\otimes m}$, we have
\begin{align*}
\m F_{N, m}(\wht\rho^k) - \m F_{N, m}(\rho)
= \big[\m F_{N, m}(\wht\rho^k) - \m F_{N, m}(\wt\rho^k)\big] + \big[\m F_{N, m}(\wt\rho^k) - \m F_{N, m}(\wht\rho^k)\big].
\end{align*}

\vspace{0.5em}
\noindent\underline{Step 1: control $\m F_{N, m}(\wht\rho^k) - \m F_{N, m}(\wt\rho^k)$.} For any arbitrary sequence $\nu_1, \dots, \nu_r\in\ms P^r(\m X)^{\otimes m}$, by additionally defining $\nu_0 = \wht\rho^k$ and $\nu_{r+1} = \wt\rho^k$, we have
\begin{align*}
&\m F_{N, m}(\wht \rho^k) - \m F_{N, m}(\wt\rho^k)
= \sum_{s=0}^r \big[\m F_{N, m}(\nu_s) - \m F_{N, m}(\nu_{s+1})\big]
\leq \sum_{s=0}^r\sum_{j=1}^m \int_{\m X}\frac{\delta\m F_{N, m}}{\delta\rho_j}(\nu_s)\,\dd[\nu_{s, j} - \nu_{s+1,j}]\\
&= \sum_{s=0}^r\sum_{j=1}^m \int_{\m X} \underbrace{-\frac{t_{j+1}-t_j}{N\lambda}\sum_{i=1}^N\frac{\m K_\sigma(X_{t_j}^i - y_j)}{\m K_\sigma\ast\nu_{s,j}(X_{t_j}^i)} + \frac{\varphi_{j, j+1}^{\nu_s}(y_j)}{t_{j+1}-t_j} + \frac{\psi_{j, j-1}^{\nu_s}(y_j)}{t_j - t_{j-1}}}_{V_j(y_j; \nu_s)} 
+ \tau\log\nu_{s, j}(y_j)\,\dd[\nu_{s, j} - \nu_{s+1, j}],
\end{align*}
where $\varphi_{j, j+1}^{\nu_s}$ and $\psi_{j, j-1}^{\nu_s}$ are the Schr\"{o}dinger potentials associated with $\nu_s = (\nu_{s, 1}, \dots, \nu_{s, m})$. 
By Lemma~\ref{lem: uniform_bound_Vj} and Pinsker's inequality, we have
\begin{align*}
\sum_{s=0}^r\sum_{j=1}^m \int_{\m X} V_j(y_j; \nu_s)\,\dd[\nu_{s, j} - \nu_{s+1, j}] 
\leq \sum_{s=0}^r\sum_{j=1}^m B_j\|\nu_{s, j} - \nu_{s+1, j}\|_{L^1(\m X)} 
\leq \sum_{s=0}^r\sum_{j=1}^m B_j\sqrt{2\KL(\nu_{s+1, j}\,\|\,\nu_{s, j})}.
\end{align*}
Note that we also have
\begin{align*}
\sum_{s=0}^r\int_{\m X}\log\nu_{s, j}(y_j)\,\dd[\nu_{s, j} - \nu_{s+1, j}] = H(\nu_{0, j}) - H(\nu_{r+1, j}) + \sum_{s=0}^r\KL(\nu_{s+1, j}\,\|\,\nu_{s, j}),
\end{align*}
where $H(\nu_j) = \int\nu_j\log\nu_j$ is the negative self entropy of $\nu_j\in\ms P^r(\m X)$. Combining all pieces above yields
\begin{align*}
\m F_{N, m}(\wht\rho^k) - \m F_{N, m}(\wt\rho^k)
\leq \sum_{s=0}^r\sum_{j=1}^m \Big[B_j\sqrt{2\KL(\nu_{s+1, j}\,\|\,\nu_{s, j})} + \tau\KL(\nu_{s+1, j}\,\|\,\nu_{s, j})\Big] + \tau\big[H(\wht\rho^k) - H(\wt\rho^k)\big].
\end{align*}
To bound the first term, note that the KL divergence is locally quadratic. Since $\{\nu_s\}_{s=1}^r$ is an arbitrary sequence on $\ms P^r(\m X)^{\otimes m}$, by taking the infimum with respect to $\{\nu_s\}_{s=1}^r$, Lemma~\ref{prop: FR_dist} yields
\begin{align*}
\m F_{N, m}(\wht\rho^k) - \m F_{N, m}(\wt\rho^k)
& \leq \sum_{j=1}^m\sup_{r,\nu_1,\dots,\nu_{r+1}}\Big\{B_j\sum_{s=0}^r\sqrt{2\KL(\nu_{s+1, j}\,\|\,\nu_{s, j})} + \tau\sum_{s=0}^r\KL(\nu_{s+1, j}\,\|\,\nu_{s, j})\Big\} + \tau\big[H(\wht\rho^k) - H(\wt\rho^k)\big]\\
&\stackrel{\ri}{\leq} \sum_{j=1}^m 2B_j\sqrt{\KL(\wht\rho_j^k\,\|\,\wt\rho_j^k)} + \tau\big[H(\wht\rho^k) - H(\wt\rho^k)\big]\\
&\stackrel{\rii}{\leq} 2\sqrt{\Big(\sum_{j=1}^m B_j^2\Big)\Big(\sum_{j=1}^m \KL(\wht\rho_j^k\,\|\,\wt\rho_j^k)\Big)} + \tau\big[H(\wht\rho^k) - H(\wt\rho^k)\big]\\
&\stackrel{\riii}{=}2\|B\|_{\ell^2(m)}\sqrt{\KL(\wht\rho^k\,\|\,\wt\rho^k)} + \tau\big[H(\wht\rho^k) - H(\wt\rho^k)\big]\\
&\leq 2\|B\|_{\ell^2(m)}\delta_k^{\frac{1}{2}} + \tau\big[H(\wht\rho^k) - H(\wt\rho^k)\big]
\end{align*}
where $\|B\|_{\ell^2(m)} = \sqrt{B_1^2 + \dots + B_m^2}$ is the $\ell^2$-norm of $B = (B_1, \dots, B_m)$. Here, (i) is due to Lemma~\ref{prop: FR_dist}; (ii) is by Cauchy--Schwarz inequality; (ii) is due to the fact that $\sum_{j=1}^m\KL(\wht\rho_j^k\,\|\,\wt\rho_j^k) = \KL(\wht\rho\,\|\,\wt\rho^k)$.

\vspace{0.5em} 
\noindent\underline{Step 2: control $\m F_{N, m}(\wt\rho^k) - \m F_{N, m}(\rho)$.} By convexity of $\m F_{N, m}$, we have
\begin{align*}
&\m F_{N, m}(\wt\rho^k) - \m F_{N, m}(\rho)
\leq \sum_{j=1}^m \int_{\m X}\frac{\delta\m F_{N, m}}{\delta\rho_j}(\wt\rho^k)\,\dd[\wt\rho_j^k - \rho_j]
= \sum_{j=1}^m\int_{\m X} V_j(y_j; \wt\rho^k) + \tau\log\wt\rho_j^k(y_j)\,\dd[\wt\rho_j^k - \rho_j]\\
&\leq \sum_{j=1}^m \int_{\m X} V_j(y_j; \wht\rho^k) + \tau\log\wt\rho_j^k(y_j)\,\dd[\wt\rho_j^k - \rho_j] + \bigg\lvert\sum_{j=1}^m \int_{\m X} V_j(y_j; \wt\rho^k) - V_j(y_j; \wht\rho^k)\,\dd[\wt\rho_j^k - \rho_j]\bigg\rvert.
\end{align*}
So, we have
\begin{align}\label{eqn: diff_exact_iter}
\begin{aligned}
\sum_{k=1}^{K}\eta_{k+1}\big[\m F_{N, m}(\wt\rho^{k}) - \m F_{N,m}(\rho)\big]
&\leq \sum_{k=1}^{K}\sum_{j=1}^m \int_{\m X}\eta_{k+1}V_j(y_j; \wht\rho^k) + \tau\eta_{k+1}\log\wt\rho_j^k(y_j)\,\dd[\wt\rho_j^k - \rho_j]\\
&\qquad\qquad + \sum_{k=1}^{K}\eta_{k+1}\bigg\lvert\sum_{j=1}^m \int_{\m X} V_j(y_j; \wt\rho^k) - V_j(y_j; \wht\rho^k)\,\dd[\wt\rho_j^k - \rho_j]\bigg\rvert
\end{aligned}
\end{align}

\underline{Step 2.1: control the first term in~\eqref{eqn: diff_exact_iter}.} Recall that $H(\rho) = \int\rho\log\rho$ and
\begin{align*}
U_k(\rho) \coloneqq \sum_{j=1}^m\int_{\m X}\sum_{l=1}^k\Big[\eta_l\prod_{l < l'\leq k}(1-\tau\eta_{l'})\Big]\big[V_j(y_j; \wht\rho^{l-1}) + \alpha_l\|y_j\|^2\big]\,\dd\rho_j + H(\rho).
\end{align*}
By the definition~\eqref{eqn: explicit_quad_anneal_update} of $\wt\rho_j^k$, we have
\begin{align*}
U_k^\ast := \min_{\rho\in\ms P^r(\m X)^{\otimes m}} U_k(\rho) = U_k(\wt\rho^k).
\end{align*}
Some involved calculations (see Appendix~\ref{app: cal1} for more details) shows that
\begin{align}\label{eqn: cal1}
\begin{aligned}
&\quad\,\sum_{k=1}^{K}\sum_{j=1}^m \int_{\m X}\eta_{k+1}V_j(y_j; \wht\rho^k) + \tau\eta_{k+1}\log\wt\rho_j^k(y_j)\,\dd[\wt\rho_j^k - \rho_j]\\
&=
\sum_{k=1}^K\sum_{j=1}^m\eta_{k+1}\int_{\m X}V_j(y_j;\wht\rho^k)\,\dd\wt\rho_j^k - \sum_{k=1}^K\tau\eta_{k+1}U_k^\ast - U_{K+1}(\rho)\\
&\qquad\qquad  + \sum_{j=1}^m \eta_1\int_{\m X}V_j(y_j;\wht\rho^0)\,\dd\rho_j + H(\rho) + \sum_{k=1}^{K+1} \alpha_k\eta_k\int\|y\|^2\,\dd\rho +\sum_{k=1}^K\tau\eta_{k+1} H(\wt\rho^k).
\end{aligned}
\end{align}
Note that
\begin{align*}
U_k^\ast = U_k(\wt\rho^k)
&= (1-\tau\eta_k)\big[U_{k-1}(\wt\rho^k) - H(\wt\rho^k)\big] + \eta_k\sum_{j=1}^m\int_{\m X}\big[V_j(y_j;\wht\rho^{k-1}) + \alpha_k\|y_j\|^2\big]\,\dd\wt\rho^k_j + H(\wt\rho^k)\\
&= (1-\tau\eta_k)\big[U_{k-1}^\ast + \KL(\wt\rho^k\,\|\,\wt\rho^{k-1})\big] + \tau\eta_k H(\wt\rho^k) +  \eta_k\sum_{j=1}^m\int_{\m X}\big[V_j(y_j;\wht\rho^{k-1}) + \alpha_k\|y_j\|^2\big]\,\dd\wt\rho^k_j.
\end{align*}
Here, the last equality is due to Lemma~\ref{lem: diff_Uk}. Therefore, we have
\begin{align*}
&U_{K+1}(\rho) 
\geq U_{K+1}^\ast = U_1^\ast + \sum_{k=1}^K\big[U_{k+1}^\ast - U_k^\ast\big]\\
&= U_1^\ast + \sum_{k=1}^K\bigg[(1-\tau\eta_{k+1})\KL(\wt\rho^{k+1}\,\|\,\wt\rho^k)-\tau\eta_{k+1}U_k^\ast + \tau\eta_{k+1}H(\wt\rho^{k+1}) + \eta_{k+1}\sum_{j=1}^m\int_{\m X}\big[V_j(y_j;\wht\rho^{k}) + \alpha_k\|y_j\|^2\big]\,\dd\wt\rho^{k+1}_j\bigg].
\end{align*}
This implies
\begin{align*}
&\quad\,\sum_{k=1}^{K}\sum_{j=1}^m \int_{\m X}\eta_{k+1}V_j(y_j; \wht\rho^k) + \tau\eta_{k+1}\log\wt\rho_j^k(y_j)\,\dd[\wt\rho_j^k - \rho_j]\\
&=
-U_1^\ast - \sum_{k=1}^K(1-\tau\eta_{k+1})\KL(\wt\rho^{k+1}\,\|\,\wt\rho^k) - \sum_{k=1}^K\alpha_k\eta_{k+1}\int\|y\|^2\,\dd\wt\rho^{k+1} - \sum_{k=1}^K\tau\eta_{k+1} H(\wt\rho^{k+1})\\
&\qquad\qquad  + \sum_{j=1}^m \eta_1\int_{\m X}V_j(y_j;\wht\rho^0)\,\dd\rho_j + H(\rho) + \sum_{k=1}^{K+1} \alpha_k\eta_k\int\|y\|^2\,\dd\rho +\sum_{k=1}^K\tau\eta_{k+1} H(\wt\rho^k)\\
&\qquad\qquad + \sum_{k=1}^K\sum_{j=1}^m\eta_{k+1}\int_{\m X}V_j(y_j;\wht\rho^k)\,\dd[\wt\rho_j^k - \wt\rho_j^{k+1}].
\end{align*}

\underline{Step 2.2: control the second term in~\eqref{eqn: diff_exact_iter}.} To control the difference, we need the following lemma. The proof is postponed to Appendix~\ref{sec: tech_lem}.
\begin{lemma}\label{lem: stable_Vj}
For any $\rho\in\ms P_2^r(\m X)^{\otimes m}$, there is a constant $C_2 = C_2(\lambda, \sigma, \tau, t_1, \dots, t_m,\m X)$, such that
\begin{align*}
\bigg\lvert \sum_{j=1}^m\int_{\m X}V_j(y_j;\wt\rho^k) - V_j(y_j;\wht\rho^k)\,\dd[\wt\rho_j^k - \rho_j]\bigg\rvert
\leq R_1(k) := C_2\sum_{j=1}^m\big[\W_2(\wht\rho_j^k, \wt\rho_j^k) + \|\wht\rho_j^k - \wt\rho_j^k\|_{L^1(\m X)}\big].
\end{align*}
Furthermore, we have $R_1(k) \lesssim \sqrt{\delta_k / \alpha_k}$.
\end{lemma}
\noindent Therefore, in \emph{Step 2}, we have shown that
\begin{align*}
\sum_{k=1}^K\eta_{k+1}\big[\m F_{N, m}(\wt\rho^k) - \m F_{N, m}(\rho)\big]
&\leq \sum_{j=1}^m \eta_1\int_{\m X}V_j(y_j;\wht\rho^0)\,\dd\rho_j + H(\rho) + \sum_{k=1}^{K+1} \alpha_k\eta_k\int\|y\|^2\,\dd\rho +\sum_{k=1}^K\tau\eta_{k+1} H(\wt\rho^k)\\
&\qquad -U_1^\ast - \sum_{k=1}^K(1-\tau\eta_{k+1})\KL(\wt\rho^{k+1}\,\|\,\wt\rho^k) - \sum_{k=1}^K\alpha_k\eta_{k+1}\int\|y\|^2\,\dd\wt\rho^{k+1}\\
&\qquad - \sum_{k=1}^K\tau\eta_{k+1} H(\wt\rho^{k+1}) + \sum_{k=1}^K\eta_{k+1}R_1(k) + \sum_{k=1}^K\sum_{j=1}^m\eta_{k+1}\int_{\m X}V_j(y_j;\wht\rho^k)\,\dd[\wt\rho_j^k - \wt\rho_j^{k+1}].
\end{align*}

\vspace{0.5em}\noindent\underline{Step 3: combining all pieces.} Now, we have
\begin{align*}
\sum_{k=1}^K  \eta_{k+1}\big[\m F_{N, m}(\wht\rho^k) - \m F_{N, m}(\rho)\big]
&\leq \sum_{k=1}^K\tau\eta_{k+1}\big[H(\wht\rho^k) - H(\wt\rho^{k+1})\big] - \sum_{k=1}^K(1-\tau\eta_{k+1})\KL(\wt\rho^{k+1}\,\|\,\wt\rho^k)\\
&\qquad + \sum_{k=1}^K\sum_{j=1}^m\eta_{k+1}\int_{\m X}V_j(y_j;\wht\rho^k)\,\dd[\wt\rho_j^k - \wt\rho_j^{k+1}] - \sum_{k=1}^K\alpha_k\eta_{k+1}\int\|y\|^2\,\dd\wt\rho^{k+1}\\
&\qquad  + \sum_{k=1}^{K+1} \alpha_k\eta_k\int\|y\|^2\,\dd\rho + \sum_{k=1}^K\eta_{k+1}R_1(k) + 2\|B\|_{\ell^2(m)}\sum_{k=1}^K\eta_{k+1}\delta_k^{\frac{1}{2}}\\
&\qquad -U_1^\ast + H(\rho) + \sum_{j=1}^m \eta_1\int_{\m X}V_j(y_j;\wht\rho^0)\,\dd\rho_j.
\end{align*}
To control the first term, we have the following lemma. The proof is postponed to Appendix~\ref{sec: tech_lem}.
\begin{lemma}\label{lem: diff_entropy}
If $\{\eta_k\}_{k=1}^\infty$ and $\{\alpha_k\}_{k=1}^\infty$ are sequences converging to $0$, and $\{\eta_k\}_{k=1}^\infty$ is decreasing,
\begin{align*}
\sum_{k=1}^K\tau\eta_{k+1}\big[H(\wht\rho^k) - H(\wt\rho^{k+1})\big] 
&\leq \tau\eta_2\big[H(\wht\rho^1) - C_3\log\alpha_2\big] + C_3\tau\sum_{k=2}^K\eta_{k+1}\log\frac{\alpha_k}{\alpha_{k+1}}\\
&\quad + \sum_{k=2}^K\tau\eta_{k+1}\Big[\big[1 + (2+\varepsilon_k+\varepsilon_k^{-1})e^{\frac{4\|B\|_{\ell^\infty(m)}}{\tau}}\big]\delta_k + \frac{\sqrt{2\delta_k}\|B\|_{\ell^2(m)}}{\tau} + \frac{\varepsilon_k(1+\varepsilon_k)d}{2}\sum_{j=1}^me^{\frac{2B_j}{\tau}}\Big],
\end{align*}
where $C_3$ is the constant in Lemma~\ref{lem: lower_bd_entropy}.
\end{lemma}
\noindent To control the next two terms, note that
\begin{align*}
&\sum_{k=1}^K\sum_{j=1}^m\eta_{k+1}\int_{\m X}V_j(y_j;\wht\rho^k)\,\dd[\wt\rho_j^k - \wt\rho_j^{k+1}] - \sum_{k=1}^K(1-\tau\eta_{k+1})\KL(\wt\rho^{k+1}\,\|\,\wt\rho^k)\\
&\leq \sum_{k=1}^K \eta_{k+1}\sum_{j=1}^m B_j\|\wt\rho_j^k - \wt\rho_{j}^{k+1}\|_{L^1(\m X)} - \sum_{k=1}^K(1-\tau\eta_{k+1})\sum_{j=1}^m\frac{1}{2}\|\wt\rho_j^k - \wt\rho_j^{k+1}\|_{L^1(\m X)}^2\\
&\leq \sum_{k=1}^K\sum_{j=1}^m \frac{B_j^2\eta_{k+1}^2}{2(1-\tau\eta_{k+1})}
= \sum_{k=1}^K \frac{\|B\|_{\ell^2(m)}^2 \eta_{k+1}^2}{2(1-\tau\eta_{k+1})}.
\end{align*}
Combining the above two pieces together with $\varepsilon_k = \sqrt{\delta_k}$ yields
\begin{align*}
\sum_{k=1}^K\eta_{k+1}\big[\m F_{N, m}(\wht\rho^k) - \m F_{N, m}(\rho)\big]
&\leq \sum_{k=2}^K\tau\eta_{k+1}\Big[\big[1 + 4\delta_k^{-\frac{1}{2}}e^{\frac{4\|B\|_{\ell^\infty(m)}}{\tau}}\big]\delta_k + \frac{\sqrt{2\delta_k}\|B\|_{\ell^2(m)}}{\tau} + d\sqrt{\delta_k}\sum_{j=1}^me^{\frac{2B_j}{\tau}}\Big]\\
&\quad + C_3\tau\sum_{k=2}^K\eta_{k+1}\log\frac{\alpha_k}{\alpha_{k+1}} + \sum_{k=1}^K\frac{\|B\|^2_{\ell^2(m)}\eta_{k+1}^2}{2(1-\tau\eta_{k+1})} + \sum_{k=1}^K\eta_{k+1}R_1(k) + 2\|B\|_{\ell^2(m)}\sum_{k=1}^K\eta_{k+1}\delta_k^{\frac{1}{2}}\\
&\quad + \sum_{k=1}^{K+1}\alpha_k\eta_k\int\|y\|^2\,\dd\rho + \Big[H(\rho) - U_1^\ast + \sum_{j=1}^m\eta_1\int_{\m X}V_j(y_j; \wht\rho^0)\,\dd\rho_j + \tau\eta_2\big[H(\wht\rho^1)-C_3\log\alpha_2\big]\Big].
\end{align*}
Analyzing the order of right-hand side (RHS) yields
\begin{align*}
\rhs &\lesssim \sum_{k=2}^K \eta_{k+1}\sqrt{\delta_k} + \sum_{k=2}^K\eta_{k+1}\log\frac{\alpha_k}{\alpha_{k+1}} + \sum_{k=1}^K \eta_{k+1}^2 + \sum_{k=1}^K\eta_{k+1}R_1(k) + \sum_{k=1}^{K+1}\alpha_k\eta_k\\
&\lesssim \sum_{k=2}^K\frac{\eta_{k+1}(\alpha_k-\alpha_{k+1})}{\alpha_{k+1}} + \sum_{k=1}^K \eta_{k+1}^2 + \sum_{k=1}^K\eta_{k+1}\sqrt{\frac{\delta_k}{\alpha_k}} + \sum_{k=1}^{K+1}\alpha_k\eta_k.
\end{align*}
Since the left-hand side can easily be bounded as
\begin{align*}
\lhs \geq \sum_{k=1}^K \eta_{k+1}\cdot\big[\min_{1\leq k\leq K}\m F(\wht\rho^k) - \m F(\rho)\big].
\end{align*}
So, we have
\begin{align*}
\min_{1\leq k\leq K}\m F_{N,m}(\wht\rho^k) - \m F_{N,m}(\rho) \lesssim \bigg[\sum_{k=1}^K\eta_{k+1}\bigg]^{-1}\bigg[\sum_{k=2}^K\frac{\eta_{k+1}(\alpha_k-\alpha_{k+1})}{\alpha_{k+1}} + \sum_{k=1}^K \eta_{k+1}^2 + \sum_{k=1}^K\eta_{k+1}\sqrt{\frac{\delta_k}{\alpha_k}} + \sum_{k=1}^{K+1}\alpha_k\eta_k\bigg].
\end{align*}

\section{Proof of technical results}\label{sec: tech_lem}
\subsection{Some useful lemmas}
\begin{lemma}\label{lem: uniform_bound_Vj}
For all $j\in[m]$, there exists constants $A_j, B_j > 0$ such that
\begin{enumerate}
\item $\matnorm{\nabla^2 V_j(\cdot; \rho)}_{\rm op} \leq A_j$;

\item $\osc\big(V_j(\cdot; \rho)\big) = \sup_{y_j, y_j'\in\m X}\big\lvert V_j(y_j; \rho) - V_j(y_j'; \rho)\big\rvert \leq B_j$
\end{enumerate}
uniformly holds for all $\rho\in\ms P_2^r(\m X)^m$.
\end{lemma}
\begin{proof}
To prove the first argument, recall that 
\begin{align*}
V_j(y_j, \rho) = -\frac{t_{j+1} - t_j}{N\lambda}\sum_{i=1}^N\frac{\m K_\sigma(X_{t_j}^i - y_j)}{\m K_\sigma\ast\rho(X_{t_j}^i)} + \frac{\varphi_{j, j+1}^{\rho}(y_j)}{t_{j+1}-t_j} + \frac{\psi_{j, j-1}^{\rho}(y_j)}{t_j - t_{j-1}}.
\end{align*}
Then, by the definition~\eqref{eqn: schrodinger}, we know
\begin{align*}
\nabla^2\varphi_{j, j+1}^\rho(y_j) = \mb E_{\gamma^\ast_\rho|X_j = y_j}\big[\nabla_{x_j}^2 c_j(X_j, X_{j+1})\big] - \mb E_{\gamma_\rho^\ast|X_j = y_j}\big[\nabla_{x_j}c_j(X_j, X_{j+1})\big] \mb E_{\gamma_\rho^\ast|X_j = y_j}\big[\nabla_{x_j}c_j(X_j, X_{j+1})\big]^\top.
\end{align*}
So, we have
\begin{align*}
\matnorm{\nabla^2\varphi_{j, j+1}^\rho}_{\rm op} \leq \sup_{y_j, y_{j+1}\in\m X} \Big[\matnorm{\nabla_{x_j}^2 c_j(y_j, y_{j+1})}_{\rm op} + \big\|\nabla_{x_j}c_j(y_j, y_{j+1})\big\|^2\Big] < \infty
\end{align*}
due to the smoothness of $c_j$ and the compactness of $\m X$. Similarly, we have $\matnorm{\nabla^2\psi_{j, j-1}^\rho}_{\rm op} < \infty$. Since $\nabla^2\m K_\sigma$ is also uniformly bounded, we known there is a constant $A_j > 0$ such that $\matnorm{\nabla^2 V_j(\cdot;\rho)}_{\rm op} \leq A_j$ hols uniformly. The second argument exactly follows the proof in~\citep[Theorem C.1,][]{chizat2022trajectory}.

\end{proof}
\begin{lemma}\label{lem: diff_Uk}
Let $U_k$ be the functinoal defined as~\eqref{eqn: def_Uk}. For any $\rho\in\ms P^r(\m X)^{\otimes m}$, we have
\begin{align*}
U_k(\rho) - U_k^\ast = \KL(\rho\,\|\,\wt\rho^k).
\end{align*}
\end{lemma}
\begin{proof}
Just note that
\begin{align*}
H(\rho) - H(\wt\rho^k)
&= \int\rho\log\frac{\rho}{\wt\rho^k} + \int\log\wt\rho^k\,\dd[\rho - \wt\rho^k]\\
&= \KL(\rho\,\|\,\wt\rho^k) - \sum_{j=1}^m\int \sum_{l=1}^k\Big[\eta_l\prod_{l<l'\leq k}(1-\tau\eta_{l'})\Big]\big[V_j(y_j;\wht\rho^{l-1}) + \alpha_l\|y_j\|^2\big]\,\dd[\rho_j - \wt\rho_j^k]\\
&=\KL(\rho\,\|\,\wt\rho^k) - \big[U_k(\rho) - H(\rho)\big] + \big[U_k(\wt\rho^k) - H(\wt\rho^k)\big].
\end{align*}
The above equality implies the desired result.
\end{proof}

\begin{lemma}\label{lem: lower_bd_entropy}
If $\{\alpha_k\}_{k=1}^\infty$ and $\{\eta_k\}_{k=1}^\infty$ are two sequences satifying the assumptions in Lemma~\ref{lem: seq_converge}, there exists a constant $C_3 = C_3(d, \tau, m, B_1, \dots, B_m) > 0$, such that for all $k\in\mb Z_+$,
\begin{align*}
H(\wt\rho^k) \geq C_3\log\alpha_k.
\end{align*}
\end{lemma}
\begin{proof}
By~\cite[Proposition B,][]{nitanda2021particle}, we have
\begin{align*}
H(\wt\rho^k) = \sum_{j=1}^m H(\wt\rho_j) 
\geq -\sum_{j=1}^m\Big[\frac{2B_j}{\tau} + \frac{d}{2}\Big(e^{\frac{2B_j}{\tau}} + \log\pi - \log\sum_{l=1}^k\frac{\alpha_l\eta_l(1-\tau\eta_{1})\cdots(1-\tau\eta_k)}{(1-\tau\eta_1)\cdots(1-\tau\eta_l)}\Big)\Big].
\end{align*}
By Lemma~\ref{lem: seq_converge}, we know there is a constant $C'>0$ such that
\begin{align*}
\log\sum_{l=1}^k\frac{\alpha_l\eta_l(1-\tau\eta_1)\cdots(1-\tau\eta_k)}{(1-\tau\eta_1)\cdots(1-\tau\eta_l)} > C'\log\frac{\alpha_k}{\tau}.
\end{align*}
Therefore, there exists constant $C_3 = C_3(d, \tau, m, B_1, \dots, B_m) > 0$ such that
\begin{align*}
H(\wt\rho^k) \geq C_3\log\alpha_k.
\end{align*}
\end{proof}
\begin{lemma}\label{lem: seq_converge}
Assume that $\{\eta_k\}_{k=1}^\infty$ and $\{\alpha_k\}_{k=1}^\infty$ are two positive sequences that satisfy
\begin{itemize}
\item $\lim_{k\to\infty}\eta_k = 0$ and $\sum_k\eta_k = \infty$;
\item $\lim_{k\to\infty}\frac{\alpha_{k-1}-\alpha_k}{\eta_k\alpha_k} = 0$;
\item $\{\alpha_ke^{\tau(\eta_1 + \dots + \eta_k)}\}_{k=1}^\infty$ is increasing when $k$ is large enough and converge to $\infty$,
\end{itemize}
Then, there are constants $C_4, C_4' > 0$ such that
\begin{align*}
\frac{C_4'\alpha_k}{\tau} < \sum_{l=1}^k\Big[\alpha_l\eta_l\prod_{l< l'\leq k}(1-\tau\eta_{l'})\Big] \geq \frac{\sum_{l=1}^k\alpha_l\eta_le^{\alpha_k2\tau(\eta_1+\dots+\eta_l)}}{\alpha_ke^{2\tau(\eta_1+\dots+\eta_k)}} < \frac{C_4\alpha_k}{\tau}
\end{align*}
holds for all $k\in\mb Z_+$.
\end{lemma}
\begin{proof}
It is easy to see that $-2x < \log(1-x) \leq -x$ for all $0\leq x\leq 1/2$. Therefore, we have
\begin{align*}
\frac{1}{\alpha_k}\sum_{l=1}^k\Big[\alpha_l\eta_l\prod_{l< l'\leq k}(1-\tau\eta_{l'})\Big] 
\leq \frac{1}{\alpha_k}\sum_{l=1}^k\alpha_l\eta_l \exp\Big\{-\tau\sum_{l<l'\leq k}\eta_{l'}\Big\} 
=\frac{\sum_{l=1}^k\alpha_l\eta_le^{\tau(\eta_1+\dots+\eta_l)}}{\alpha_ke^{\tau(\eta_1+\dots+\eta_k)}}.
\end{align*}
By Stolz formula~\citep{fikhtengol2014fundamentals}, we have
\begin{align*}
\lim_{k\to\infty}\frac{\sum_{l=1}^k\alpha_l\eta_le^{\tau(\eta_1+\dots+\eta_l)}}{\alpha_ke^{\tau(\eta_1+\dots+\eta_k)}}
=\lim_{k\to\infty} \frac{\alpha_k\eta_ke^{\tau(\eta_1 + \dots + \eta_k)}}{\alpha_ke^{\tau(\eta_1+\dots+\eta_k)} - \alpha_{k-1}e^{\tau(\eta_1+\dots+\eta_{k-1})}}
= \lim_{k\to\infty}\frac{\eta_k}{1 - \frac{\alpha_{k-1}}{\alpha_k}e^{-\tau\eta_k}}
= \frac{1}{\tau}.
\end{align*}
Similarly, we have
\begin{align*}
\frac{1}{\alpha_k}\sum_{l=1}^k\Big[\alpha_l\eta_l\prod_{l< l'\leq k}(1-\tau\eta_{l'})\Big] \geq \frac{\sum_{l=1}^k\alpha_l\eta_le^{\alpha_k2\tau(\eta_1+\dots+\eta_l)}}{\alpha_ke^{2\tau(\eta_1+\dots+\eta_k)}} \to \frac{1}{2\tau}.
\end{align*}
Therefore, we know there are constants $C_4, C_4' > 0$ such that
\begin{align*}
\frac{C_4'}{\tau} < \frac{1}{\alpha_k}\sum_{l=1}^k\Big[\alpha_l\eta_l\prod_{l< l'\leq k}(1-\tau\eta_{l'})\Big] \geq \frac{\sum_{l=1}^k\alpha_l\eta_le^{\alpha_k2\tau(\eta_1+\dots+\eta_l)}}{\alpha_ke^{2\tau(\eta_1+\dots+\eta_k)}} < \frac{C_4}{\tau}.
\end{align*}
\end{proof}

\begin{lemma}\label{lem: rho_LSI}
If $\{\eta_k\}_{k=1}^\infty$ and $\{\eta_k\}_{k=1}^{\infty}$ satisfy the assumptions in Lemma~\ref{lem: seq_converge},
$\wt\rho^k_j$ satisfies $\LSI\big(\frac{2C_4'\alpha_k}{\tau e^{C_4B_j/\tau}}\big)$.    
\end{lemma}
\begin{proof}
Note that we have
\begin{align*}
&\bigg\lvert \sum_{l=1}^k\Big[\eta_l\prod_{l<l'\leq k}(1-\tau\eta_{l'})\Big]V_j(y_j;\wht\rho^{l-1})\bigg\rvert
\leq \frac{C_4}{\tau} \cdot \sup_{\rho\in\ms P^r(\m X)^{\otimes m}}\|V_j(\cdot; \rho)\|_{L^\infty(\m X)}
\leq \frac{C_4B_j}{\tau},
\end{align*}
where the first inequality is due to Lemma~\ref{lem: seq_converge}, and the second inequality is due to Lemma~\ref{lem: uniform_bound_Vj}. Then, by Proposition~\ref{prop: perturb_LSI}, we know $\wt\rho_j^k$ satisfies $\LSI$ with parameter
\begin{align*}
2e^{-\frac{C_4B_j}{\tau}}\sum_{l=1}^k\alpha_l\eta_l(1-\tau\eta_{l+1})\cdots(1-\tau\eta_k) > \frac{2C_4'\alpha_k}{\tau e^{C_4B_j/\tau}}.
\end{align*}
Here, the inequality is due to Lemma~\ref{lem: seq_converge} again.
\end{proof}
\begin{lemma}\label{lem: L-smooth-rho}
$\log\wt \rho_j^k$ is $\big(\frac{C_4(A_j+2\alpha_k)}{\tau}\big)$-smooth, i.e.
\begin{align*}
\matnorm{\nabla^2 \log\wt \rho_j^k}_{\rm op} \leq \frac{C_4(A_j+2\alpha_k)}{\tau}.
\end{align*}
\end{lemma}
\begin{proof}
Recall that
\begin{align*}
\nabla^2\log\wt\rho_j^k = -\sum_{l=1}^k \eta_l\Big[\prod_{l<l'\leq k}(1-\tau\eta_{l'})\Big]\nabla^2 V_j(y_j; \wht\rho^{l-1}) - 2\sum_{l=1}^k\eta_l\alpha_l\Big[\prod_{l<l'\leq k}(1-\tau\eta_{l'})\Big]I_d.
\end{align*}
Therefore, by Lemma~\ref{lem: uniform_bound_Vj} we have
\begin{align*}
\matnorm{\nabla^2V_j(y_j;\wht\rho^{l-1})}_{\rm op} 
&\leq \sum_{l=1}^k \eta_l\Big[\prod_{l<l'\leq k}(1-\tau\eta_{l'})\Big]A_j + 2\sum_{l=1}^k\eta_l\alpha_l\Big[\prod_{l<l'\leq k}(1-\tau\eta_{l'})\Big]\\
&\leq \frac{C_4A_j}{\tau} + \frac{2C_4\alpha_k}{\tau} = \frac{C_4(A_j + 2\alpha_k)}{\tau},
\end{align*}
where the last line is due to Lemma~\ref{lem: seq_converge}.
\end{proof}

\begin{proposition}[modified Ledoux--Talagrand's contraction theorem]\label{prop: Ledoux-Talagrand}
Let $F:\mb R_+\to\mb R_+$ be convex and non-decreasing. Let $\phi_i:\mb R\to\mb R$ be a $L_i$-Lipschitz function satisfying $\phi_i(0) = 0$. Let $\epsilon_i$ be independent Rademacher random variables. For any $S\subset\mb R^n$, we have
\begin{align*}
\mb EF\bigg(\frac{1}{2}\sup_{s\in S}\Big| \sum_{i=1}^n\epsilon_i\phi_i(s_i)\Big|\bigg) \leq \mb EF\bigg(\sup_{s\in S}\Big|\sum_{i=1}^n\epsilon_iL_is_i\Big|\bigg).
\end{align*}
\end{proposition}
\begin{proof}
The proof simply follows the steps in the proof of~\citep[Theorem 4.12,][]{ledoux2013probability}, but changing the universal Lipschitz constant $1$ to $L_i$ when conditioning on $\epsilon_1,\dots, \epsilon_{i-1},\epsilon_{i+1},\dots, \epsilon_n$.
\begin{rem}
For $i\in[n]$, let $a_i\in\mb R_+$ and $\wt\phi_i$ be $L$-Lipschitz functions and satisfying $\wt\phi_i(0) = 0$. For given $x_1, \dots, x_n\in\mb R^d$, and a function class $\ms F$, take $S = \{(f(x_1), \dots, f(x_n)): f\in\ms F\}$. Taking $F(x) = x$ and $\phi_i = a_i\wt\phi_i$, the above statement implies
\begin{align*}
\mb E\sup_{f\in\m F}\Big|\sum_{i=1}^n a_i \epsilon_i\wt\phi_i(f(x_i))\Big| \leq 2L\mb E\sup_{f\in\m F}\Big|\sum_{i=1}^n a_i\epsilon_i f(x_i)\Big|.
\end{align*}
Furthermore, let $g:\mb R^d\to\mb R$ be a function. If $\psi_i$ is $L$-Lipschitz (but not necessarily satisfies $\psi_i(0) = 0$), taking $\wt\phi_i(z) = \psi_i(z + g(x_i)) - \psi_i(g(x_i))$, the above inequality implies
\begin{align}\label{eqn: Ledoux-Talagrand}
\mb E\sup_{f\in\m F}\Big|\sum_{i=1}^n a_i\epsilon_i\big[\psi_i(f(x_i)) - \psi_i(g(x_i))\big]\Big| \leq 2L\mb E\sup_{f\in\m F}\Big|\sum_{i=1}^n a_i\epsilon_i[f(x_i) - g(x_i)]\Big|.
\end{align}
\end{rem}
\end{proof}

\begin{lemma}\label{lem: Hellinger}
For any any $\rho, \rho'\in\ms P^r(\m X)$, we have
\begin{align*}
d_{\rm H}\Big(\frac{\rho+\rho'}{2},\rho\Big)
&\geq \frac{1}{2+\sqrt{2}}d_{\rm H}(\rho, \rho').
\end{align*}
\end{lemma}
\begin{proof}
Simply note that
\begin{align*}
\bigg|\sqrt{\frac{\rho + \rho'}{2}}-\rho\bigg|
&= \frac{1}{2}\bigg|\frac{(\sqrt{\rho} - \sqrt{\rho'})(\sqrt{\rho}+ \sqrt{\rho'})}{\sqrt{\frac{\rho+\rho'}{2}}+\sqrt{\rho}}\bigg|
= \frac{|\sqrt{\rho}-\sqrt{\rho'}|}{2} \cdot \frac{\sqrt{\rho}+\sqrt{\rho'}}{\sqrt{\frac{\rho+\rho'}{2}}+\sqrt{\rho}} 
\geq \frac{|\sqrt{\rho}-\sqrt{\rho'}|}{2+\sqrt{2}}.
\end{align*}
Taking the square and then integrating both sides yield the result.
\end{proof}

\begin{lemma}\label{lem: diff_H2}
Assume $p, p', q, q'\in\ms P^r(\m X)$ satisfying $p, p', q, q'\geq C$ for some constant $C>0$ and 
\begin{align*}
|p(x) - p'(x)| \leq \varepsilon
\quad\mx{and}\quad
|q(x) - q'(x)| \leq \varepsilon
\end{align*}
for some constant $\varepsilon > 0$. Then, we have
\begin{align*}
d_{\rm H}^2(p, q) - d_{\rm H}^2(p', q') \leq 2\sqrt{\frac{\vol(\m X)}{C}}\varepsilon d_{\rm H}(p, q) + \frac{2\vol(\m X)}{C}\varepsilon^2.
\end{align*}
\end{lemma}
\begin{proof}
Let $\mu_s = p + s(p' - p)$ and $\nu_s = q + s(q' - q)$. Then, we have
\begin{align*}
d_{\rm H}^2(p, q) - d_{\rm H}^2(p', q')
&= 2\int\sqrt{p'q'} - \sqrt{pq}\,\dd x = 2\int h_1(x) - h_0(x)\,\dd x,
\end{align*}
where we define $h_s(x) = \sqrt{\mu_s(x)\nu_s(x)}$ for simplicity. Note that $h_s$ is smooth with respect to $s\in[0, 1]$. By mean value theorem, there exists $\xi_x$ between $h_1(x)$ and $h_0(x)$, such that
\begin{align*}
h_1(x) - h_0(x)
&= \partial_s h_0(x) + \frac{\partial_{ss}h_{\xi_x}(x)}{2}.
\end{align*}
Noting that $\partial_{ss}\mu_s = \partial_{ss}\nu_s = 0$, simple calculation shows that
\begin{align*}
\partial_s h_s(x) &= \frac{\nu_s\partial_s\mu_s + \mu_s\partial\nu_s}{2\sqrt{\mu_s\nu_s}}
=\frac{p' - p}{2}\sqrt{\frac{\nu_s}{\mu_s}} + \frac{q'-q}{2}\sqrt{\frac{\mu_s}{\nu_s}}\\
\partial_{ss} h_s(x) &= \frac{2(\partial_s\mu_s)(\partial_s\nu_s)}{2\sqrt{\mu_s\nu_s}} - \frac{(\mu_s\partial_s\nu_s + \nu_s\partial_s\mu_s)^2}{4(\mu_s\nu_s)^{3/2}} \leq \frac{(p'-p)(q'-q)}{\sqrt{\mu_s\nu_s}}.
\end{align*}
Therefore, we have
\begin{align*}
h_1(x) - h_0(x)
\leq \frac{p'-p}{2}\sqrt{\frac{q}{p}} + \frac{q'-q}{2}\sqrt{\frac{p}{q}} + \frac{\varepsilon^2}{C}.
\end{align*}
So, we have
\begin{align*}
\int h_1(x) - h_0(x)\,\dd x
&\leq \int \frac{p'-p}{2\sqrt{p}}\cdot\sqrt{p}\Big(\sqrt{\frac{q}{p}}-1\Big) + \frac{q'-q}{2\sqrt{q}}\cdot\sqrt{q}\Big(\sqrt{\frac{p}{q}}-1\Big) + \frac{\varepsilon^2}{C}\,\dd x\\
&\stackrel{\ri}{\leq}\sqrt{\int \frac{(p'-p)^2}{4p}\,\dd x}\cdot d_{\rm H}(p, q) + \sqrt{\int \frac{(q'-q)^2}{4q}\,\dd x}\cdot d_{\rm H}(p, q) + \frac{\vol(\m X)}{C}\varepsilon^2\\
&\stackrel{}{\leq} \sqrt{\frac{\vol(\m X)}{C}}\varepsilon d_{\rm H}(p, q) + \frac{\vol(\m X)}{C}\varepsilon^2.
\end{align*}
Here, (i) is by Cauchy--Schwarz inequality.
\end{proof}

\subsection{Proof of Proposition~\ref{prop: FR_dist}}
Before proving the theorem, we shall first note that for any $\rho, \rho'\in\ms P^r(\m X)$, the construction
\begin{align*}
t^\ast := d(\rho, \rho') = \arccos\Big(\int_{\m X}\sqrt{\rho\rho'}\,\dd x\Big) \in \Big[0, \frac{\pi}{2}\Big],
\quad\mx{and}\quad
f_{\rho, \rho'} = \frac{\sqrt{\rho'} - \sqrt{\rho}\cos t^\ast}{\sin t^\ast}
\end{align*}
satisfies
\begin{align*}
\int_{\m X} f_{\rho, \rho'}^2\,\dd x = 1, 
\quad
\int_{\m X}f_{\rho,\rho'}\sqrt{\rho}\,\dd x = 0,
\quad\mx{and}\quad
\sqrt{\rho'} = \sqrt{\rho}\cos t^\ast  + f_{\rho, \rho'}\sin t^\ast.
\end{align*}
Furthermore, for any $t\in[0, t^\ast]$, it is easy to verify that
\begin{align*}
\sqrt{\rho}\cos t + f_{\rho, \rho'}\sin t \geq 0,
\quad\mx{and}\quad
\int_{\m X}\big(\sqrt{\rho}\cos t + f_{\rho, \rho'}\sin t\big)^2\,\dd x = 1.
\end{align*}
Therefore, we can define a curve on $\ms P^r(\m X)$ by $\sqrt{\rho_t} := \sqrt{\rho}\cos t + f_{\rho, \rho'}\sin t$ to connect $\rho$ and $\rho'$. The above argument is the correction of the statement by~\cite{holbrook2020nonparametric}.

\vspace{0.5em}
\noindent (1) Since $0\leq \int_{\m X}\sqrt{\rho\rho'}\,\dd x\leq 1$, we know $d(\rho, \rho')$ is well defined. To show that $d$ is a distance, first note that $d(\rho, \rho') = 0$ if and only if
\begin{align*}
1 = \int_{\m X}\sqrt{\rho\rho'}\,\dd x
\leq \int_{\m X}\frac{\rho + \rho'}{2}\,\dd x = 1.
\end{align*}
Therefore, we have $\rho = \rho'$ almost surely. Next, we need to show $d$ satisfies the triangular inequality, i.e.
\begin{align*}
\arccos\Big(\int_{\m X}g_1 h\,\dd x\Big) + \arccos\Big(\int_{\m X} g_2h\,\dd x\Big) \geq \arccos\Big(\int g_1g_2\,\dd x\Big)
\end{align*}
holds for all $\|g_1\|_{L^2(\m X)} = \|g_2\|_{L^2(\m X)} = \|h\|_{L^2(\m X)} = 1$. Consider the Lagrangian multiplier
\begin{align*}
L(h, \lambda) = \arccos\Big(\int_{\m X}g_1h\,\dd x\Big) + \arccos\Big(\int_{\m X}g_2h\,\dd x\Big) + \lambda\Big(\int_{\m X}h^2\,\dd x - 1\Big).
\end{align*}
Taking Frechet derivative of $L$ with respective $h$ yields
\begin{align*}
\frac{\partial L}{\partial h} = -\frac{g_1}{\sqrt{1 + \int g_1h\,\dd x}} - \frac{g_2}{\sqrt{1 + \int g_2h\,\dd x}} + 2\lambda h = 0.
\end{align*}
This implies the optimal $h^\ast$ is a linear combination of $g_1$ and $g_2$. Let $\theta = \arccos(\int g_1g_2\,\dd x)$, and $h^\ast = a_1 g_1 + a_2g_2$ with $a_1, a_2 \geq 0$, such that
\begin{align*}
1 = \|h^\ast\|_{L^2(\m X)} = a_1^2 + a_2^2 + 2a_1a_2\cos\theta.
\end{align*}
Then, we have
\begin{align*}
\arccos\Big(\int g_1h^\ast\,\dd x\Big) + \arccos\Big(\int g_2h^\ast\,\dd x\Big)
= \arccos(a_1 + a_2\cos\theta) + \arccos(a_2 + a_1\cos\theta).
\end{align*}
It is easy to see that
\begin{align*}
\cos\big[\arccos(a_1 + a_2\cos\theta) + \arccos(a_2 + a_1\cos\theta)\big] = \cos\theta.    
\end{align*}
Therefore, we know
\begin{align*}
\arccos\Big(\int g_1h\,\dd x\Big) + \arccos\Big(\int g_2h\,\dd x\Big)
&\geq \arccos\Big(\int g_1h^\ast\,\dd x\Big) + \arccos\Big(\int g_2h^\ast\,\dd x\Big)\\
&= \theta = \arccos\Big(\int_{\m X}g_1g_2\,\dd x\Big).
\end{align*}
The above arguments imply that $d$ is a distance.

To prove $d(\rho, \rho') \leq \sqrt{\KL(\rho\,\|\, \rho')}$, it is easy to see that
\begin{align*}
(\arccos x)^2 + 2\log x \leq 0, \qquad\forall x\in(0, 1].
\end{align*}
Therefore, 
\begin{align*}
d(\rho, \rho') 
&= \Big[\arccos\Big(\int_{\m X}\sqrt{\rho\rho'}\,\dd x\Big)\Big]^2
\leq -2\log\Big(\int_{\m X}\sqrt{\rho\rho'}\,\dd x\Big)\\
&=-2\log\Big(\int_{\m X}\frac{\sqrt{\rho\rho'}}{\rho}\,\dd\rho\Big)\leq -2\int_{\m X}\log\frac{\sqrt{\rho\rho'}}{\rho}\,\dd\rho 
= \KL(\rho\,\|\,\rho').
\end{align*}
(2) We have argued that $\sqrt{\rho_t} = \sqrt{\rho}\cos t + f_{\rho, \rho'} \sin t$ defines a curve $\big\{\rho_t: t\in[0, t^\ast]\big\}$ on $\ms P^2(\m X)$. We first prove that there is a constant $C > 0$, such that
\begin{align}\label{eqn: upper_KL}
\KL(\rho_{s_2}\,\|\,\rho_{s_1}) \leq 2(s_2 - s_1)^2 + C\vol(X)(s_2 - s_1)^3.
\end{align}
Let $q_t = \sqrt{\rho_t}$. Then, we have
\begin{align*}
\int_{\m X} q_t^2 \,\dd x &= \int_{\m X}\rho_t\,\dd x = 1\\
\int_{\m X} \dot{q}_t^2\,\dd x &= \int_{\m X}\big(-\sqrt{\rho}\sin t + f_{\rho, \rho'}\cos t\big)^2\,\dd x = \int_{\m X}\rho\sin^2 t + f_{\rho,\rho'}^2\cos^2t - 2\sqrt{\rho}f_{\rho,\rho'}\sin t\cos t\,\dd x = 1\\
\ddot{q}_t &= - \sqrt{\rho}\cos t - f_{\rho,\rho'}\sin t = -q_t. 
\end{align*}
Let $f_{t, s} =   \dot{\rho}_t\log\frac{\rho_t}{\rho_s}$. By Taylor's expansion, there exists $\xi_x\in[0, t]$ such that
\begin{align*}
\KL(\rho_{s_2}\,\|\,\rho_{s_1})
&= \int_0^{s_2 - s_1}\Big[\frac{\dd}{\dd t}\int_{\m X}\rho_{s_1+t}\log\frac{\rho_{s_1+t}}{\rho_{s_1}}\,\dd x\Big]\,\dd t
= \int_0^{s_2 - s_1}\int_{\m X}f_{s_1+t, s_1}(x)\,\dd x\dd t\\
&= \int_0^{s_2-s_1}\int_{\m X} f_{s_1, s_1}(x) + \partial_tf_{s_1, s_1}(x)t + \partial_t^2 f_{s_1+\xi_x, s_1}(x)\frac{t^2}{2}\,\dd x\dd t.
\end{align*}
Direct calculation shows
\begin{align*}
f_{s_1, s_1} = \dot{\rho}_{s_1}
\quad\mx{and}\quad
\partial_t f_{s_1,s_1} = \ddot{\rho}_{s_1} + \frac{\dot\rho_{s_1}^2}{\rho_{s_1}} = 6\dot q_t^2 - 2q_t^2.
\end{align*}
To control the quadratic term, it is easy to prove that there is constant $C > 0$ independent of $x$ and $s_1$, such that $|\partial_t^2 f_{s_1+\xi_x, s_1}(x)| \leq C$. Therefore, we have
\begin{align*}
\KL(\rho_{s_2}\,\|\,\rho_{s_1})
&\leq \int_0^{s_2 - s_1}\!\!\int_{\m X} \dot\rho_{s_1} + [6\dot q_t^2 - 2q_t^2]t  + \frac{C}{2}t^2\,\dd x\dd t\\
&=\int_0^{s_2 - s_1} 4t + \frac{C\vol(\m X)}{2}t^2\,\dd t\\
&= 2(s_2 - s_1)^2 + C\vol(\m X)(s_2-s_1)^3.
\end{align*}
Here, the constant $C$ may change from lines to lines. Thus, we have shown~\eqref{eqn: upper_KL}. Now, let us take $s_i = it^\ast / (r+1)$ for $i=0, 1, \dots, r+1$, and we have
\begin{align*}
&\quad\,\inf_{r, \mu_0, \dots, \mu_{r+1}}\Big\{\sum_{s=0}^r \KL(\mu_{s+1}\,\|\,\mu_{s}): \mu_{r+1} = \rho', \mu_{0}=\rho, \mu_s\in\ms P^r(\m X)\Big\}\\
&\leq \inf_r\sum_{i=0}^r\KL(
\rho_{s_{i+1}}\,\|\,\rho_{s_i})
\leq \inf_r\sum_{i=0}^r 2(s_{i+1} - s_i)^2 + C\vol(\m X)(s_{i+1}-s_i)^3\\
&= \inf_{r}\sum_{i=0}^r \frac{2(t^\ast)^2}{(r+1)^2} + \frac{C\vol(\m X)(t^\ast)^3}{(r+1)^3} = 0.
\end{align*}
(3) Similarly, we have
\begin{align*}
&\quad\,\inf_{r, \mu_0, \dots, \mu_{r+1}}\Big\{\sum_{s=0}^r \sqrt{\KL(\mu_{s+1}\,\|\,\mu_{s})}: \mu_{r+1} = \rho', \mu_{0}=\rho, \mu_s\in\ms P^r(\m X)\Big\}\\
&\leq \inf_r\sum_{i=0}^r\sqrt{\KL(
\rho_{s_{i+1}}\,\|\,\rho_{s_i})}
\leq \inf_r\sum_{i=0}^r\sqrt{ 2(s_{i+1} - s_i)^2 + C\vol(\m X)(s_{i+1}-s_i)^3}\\
&= \inf_{r}\sum_{i=0}^r \sqrt{2}(s_{i+1}-s_i) + \sqrt{C\vol(\m X)}(s_{i+1}-s_i)^{\frac{3}{2}}= \sqrt{2}t^\ast \\
&= \sqrt{2}d(\rho, \rho').
\end{align*}

\subsection{Proof of Lemma~\ref{lem: stable_Vj}}
Recall the definition of $V_j(y_j; \rho)$ in~\eqref{eqn: Vj_def}. We have
\begin{align*}
&\quad\,\bigg\lvert \sum_{j=1}^m\int_{\m X}V_j(y_j;\wt\rho^k) - V_j(y_j;\wht\rho^k)\,\dd[\wt\rho_j^k - \rho_j]\bigg\rvert\\
&= \bigg\lvert \sum_{j=1}^m \int_{\m X}\frac{\wht\varphi_{j, j+1}^k - \wt\varphi_{j, j+1}^k}{t_{j+1} - t_j} + \frac{\wht\psi_{j, j-1}^k - \wt\psi_{j, j-1}^k}{t_j - t_{j-1}} - \frac{t_{j+1} - t_j}{N\lambda}\sum_{i=1}^N\Big[\frac{\m K_\sigma(X_{t_j}^i - y_j)}{\m K_\sigma\ast\wht\rho_{j}^k(X_{t_j}^i)} - \frac{\m K_\sigma(X_{t_j}^i - y_j)}{\m K_\sigma\ast\wt\rho_{j}^k(X_{t_j}^i)}\Big]\,\dd[\wt\rho_j^k - \rho_j]\bigg\rvert\\
&\leq I_1 + I_2,
\end{align*}
where we let
\begin{align*}
I_1 &:= \bigg\lvert\sum_{j=1}^m \int_{\m X}\frac{\wht\varphi_{j, j+1}^k - \wt\varphi_{j, j+1}^k}{t_{j+1} - t_j}\,\dd[\wt\rho_j^k - \rho_j] + \sum_{j=1}^m\int_{\m X}\frac{\wht\psi_{j, j-1}^k - \wt\psi_{j, j-1}^k}{t_j - t_{j-1}}\,\dd[\wt\rho_j^k - \rho_j]\bigg\rvert\\
I_2 &:= \sum_{j=1}^m \frac{t_{j+1} - t_j}{N\lambda}\bigg\lvert \int_{\m X}\sum_{i=1}^N\Big[\frac{\m K_\sigma(X_{t_j}^i - y_j)}{\m K_\sigma\ast\wht\rho_{j}^k(X_{t_j}^i)} - \frac{\m K_\sigma(X_{t_j}^i - y_j)}{\m K_\sigma\ast\wt\rho_{j}^k(X_{t_j}^i)}\Big]\,\dd[\wt\rho_j^k - \rho_j]\bigg\rvert.
\end{align*}
To control $I_1$, note that for every $\beta = (\beta_1, \dots, \beta_{m-1})\in\mb R^{m-1}$, we have
\begin{align*}
I_1 \leq \sum_{j=1}^{m-1} \frac{2}{t_{j+1} - t_j}\Big[\big\|\wht\varphi_{j, j+1}^k - \wt\varphi_{j, j+1}^k - \beta_j\big\|_{L^\infty(\m X)} + \big\|\wht\psi_{j+1, j}^k -\wt\psi_{j+1, j}^k\big\|_{L^\infty(\m X)}\Big].
\end{align*}
Taking infimum over $\beta\in\mb R^{m-1}$ and applying~\citep[Corollary 2.4,][]{carlier2022lipschitz} yield
\begin{align*}
I_1\leq \sum_{j=1}^{m-1} \frac{C(\tau^j, \m X)}{t_{j+1}-t_j}\big[\W_2(\wht\rho_j^k, \wt\rho_j^k) + \W_2(\wht\rho_j^k, \wt\rho_j^k)\big] \leq C\sum_{j=1}^m\W_2(\wht\rho_j^k, \wt\rho_j^k)
\end{align*}
for some constant $C = C(\tau, t_1, \dots, t_m, \m X) > 0$. To control $I_2$, just note that
\begin{align*}
&\frac{1}{N}\bigg\lvert \int_{\m X}\sum_{i=1}^N\Big[\frac{\m K_\sigma(X_{t_j}^i - y_j)}{\m K_\sigma\ast\wht\rho_{j}^k(X_{t_j}^i)} - \frac{\m K_\sigma(X_{t_j}^i - y_j)}{\m K_\sigma\ast\wt\rho_{j}^k(X_{t_j}^i)}\Big]\,\dd[\wt\rho_j^k - \rho_j]\bigg\rvert
\leq 2\Big[\frac{\max_{y_j\in\m X}\m K_\sigma(y_j)}{\min_{y_j\in\m X}\m K_\sigma(y_j)}\Big]^2\|\wht\rho_{j}^k - \wt\rho_{j}^k\|_{L^1(\m X)}.
\end{align*}
Therefore, we have
\begin{align*}
I_2 \leq \sum_{j=1}^m \frac{2(t_{j+1} - t_j)}{\lambda}\Big[\frac{\max_{y_j\in\m X}\m K_\sigma(y_j)}{\min_{y_j\in\m X}\m K_\sigma(y_j)}\Big]^2\|\wht\rho_{j}^k - \wt\rho_{j}^k\|_{L^1(\m X)}.
\end{align*}
Combining the upper bounds of $I_1$ and $I_2$ leads to the result.

To derive the order of $R_1(k)$, by Lemma~\ref{lem: rho_LSI} $\wt\rho_j^k$ satisfies $\LSI(\frac{2C_4'\alpha_k}{\tau e^{C_4B_j/\tau}})$. Therefore, by Talagrand's transportation inequality (Proposition~\ref{prop: talagrand}), we have
\begin{align*}
\W_2^2(\wt\rho_j^k, \wht\rho_j^k) \leq \frac{\tau e^{C_4B_j/\tau}}{2C_4'\alpha_k} \KL(\wht\rho_j^k\,\|\,\wt\rho_j^k).
\end{align*}
By Pinsker's inequality, we have $\|\wht\rho_j^k - \wt\rho_j^k\|_{L^1}^2 \leq 2\KL(\wht\rho_j^k\,\|\,\wt\rho_j^k)$. Thus, we have
\begin{align*}
R_1(k) &= C_2\sum_{j=1}^m\big[\W_2(\wht\rho_j^k, \wt\rho_j^k) + \|\wht\rho_j^k - \wt\rho_j^k\|_{L^1(\m X)}\big]
\leq C_2\bigg[\sqrt{\frac{m\tau e^{C_4\|B\|_{\ell^\infty}/\tau}}{2C_4'\alpha_k}} + \sqrt{2m}\bigg]\delta_k^{\frac{1}{2}}.
\end{align*}

\subsection{Proof of Lemma~\ref{lem: diff_entropy}}
Note that
\begin{align*}
&\sum_{k=1}^K\tau\eta_{k+1}\big[H(\wht\rho^k) - H(\wt\rho^{k+1})\big]
= \sum_{k=1}^K\tau\eta_{k+1}\big[H(\wht\rho^k) - C_3\log\alpha_{k+1}\big] - \sum_{k=1}^K\tau\eta_{k+1}\big[H(\wt\rho^{k+1}) - C_3\log\alpha_{k+1}\big]\\
&\stackrel{\ri}{\leq} \sum_{k=1}^K\tau\eta_{k+1}\big[H(\wht\rho^k) - C_3\log\alpha_{k+1}\big] - \sum_{k=1}^K\tau\eta_{k+2}\big[H(\wt\rho^{k+1}) - C_3\log\alpha_{k+1}\big]\\
&= \sum_{k=2}^K\tau\eta_{k+1}\big[H(\wht\rho^k) - H(\wt\rho^k) + C_3\log\frac{\alpha_k}{\alpha_{k+1}}\big] - \tau\eta_{K+2}\big[H(\wt\rho^{K+1}) - C_3\log\alpha_{K+1}\big] + \tau\eta_2\big[H(\wht\rho^1) - C_3\log\alpha_2\big]\\
&\stackrel{\rii}{\leq} \sum_{k=2}^K\tau\eta_{k+1}\big[H(\wht\rho^k) - H(\wt\rho^k)\big] + \sum_{k=2}^KC_3\tau\eta_{k+1}\log\frac{\alpha_k}{\alpha_{k+1}} + \tau\eta_2\big[H(\wht\rho^1) - C_3\log\alpha_2\big].
\end{align*}
Here, (i) is due to the monotonicity of $\{\eta_k\}_{k=1}^K$ and the fact that $H(\wt\rho^{k}) \geq C_3\log\alpha_k$ by Lemma~\ref{lem: lower_bd_entropy}; (ii) is again due to Lemma~\ref{lem: lower_bd_entropy}. To control $H(\wht\rho^k) - H(\wt\rho^k)$, we can directly apply~\cite[Proposition A,][]{nitanda2021particle} to get
\begin{align*}
\big\lvert H(\wht\rho^k) - H(\wt\rho^k)\big\rvert
&\leq \sum_{j=1}^m\big[1 + (2+\varepsilon_k + \varepsilon_k^{-1})e^{\frac{4B_j}{\tau}}\big]\KL(\wht\rho_j^k\,\|\,\wt\rho_j^k) + \frac{B_j}{\tau}\sqrt{2\KL(\wht\rho_j^k\,\|\,\wt\rho_j^k)} + \frac{\varepsilon_k(1+\varepsilon_k)de^{\frac{2B_j}{\tau}}}{2}\\
&\leq \big[1 + (2+\varepsilon_k+\varepsilon_k^{-1})e^{\frac{4\|B\|_{\ell^\infty(m)}}{\tau}}\big]\delta_k + \frac{\sqrt{2\delta_k}\|B\|_{\ell^2(m)}}{\tau} + \frac{\varepsilon_k(1+\varepsilon_k)d}{2}\sum_{j=1}^me^{\frac{2B_j}{\tau}}.
\end{align*}
where $\varepsilon_k > 0$ can be any positive value, and in the last line we use the assumption that $\KL(\wht\rho^k\,\|\,\wt\rho^k)\leq \delta_k$. Combining all pieces above yields
\begin{align*}
\sum_{k=1}^K\tau\eta_{k+1}\big[H(\wht\rho^k) - H(\wt\rho^{k+1})\big] 
&\leq \tau\eta_2\big[H(\wht\rho^1) - C_3\log\alpha_2\big] + C_3\tau\sum_{k=2}^K\eta_{k+1}\log\frac{\alpha_k}{\alpha_{k+1}}\\
&\quad + \sum_{k=2}^K\tau\eta_{k+1}\Big[\big[1 + (2+\varepsilon_k+\varepsilon_k^{-1})e^{\frac{4\|B\|_{\ell^\infty(m)}}{\tau}}\big]\delta_k + \frac{\sqrt{2\delta_k}\|B\|_{\ell^2(m)}}{\tau} + \frac{\varepsilon_k(1+\varepsilon_k)d}{2}\sum_{j=1}^me^{\frac{2B_j}{\tau}}\Big].
\end{align*}
\subsection{Proof of Lemma~\ref{lem: control_J1}}
To control $J_1$, note that
\begin{align*}
&J_1 = \vol(\mb T^d)\cdot\bigg\lvert\int_{\mb T^d}\sum_{k\in\mb Z^d}e^{-\frac{\|x-y-2\pi k\|^2}{2\sigma^2}}\,\dd\big[R_{T_j}(y) - R_j^{\rm rec}(y)\big]\bigg\rvert\\
&= \vol(\mb T^d)\cdot\bigg\lvert \sum_{k\in\mb Z^d}e^{-\frac{2\pi^2\|k\|^2}{\sigma^2}}\int_{\mb T^d}e^{-\big[\frac{\|x-y\|^2}{2\sigma^2} - \frac{2\pi}{\sigma^2}k^\top(x-y)\big]}\,\dd\big[R_{T_j}(y) - R_j^{\rm rec}(y)\big]\bigg\rvert\\
&\leq \vol(\mb T^d)\cdot\bigg\lvert \sum_{k\in\mb Z^d}e^{-\frac{2\pi^2\|k\|^2}{\sigma^2}}\int_{\mb T^d}\Big[e^{-\big[\frac{\|x-y\|^2}{2\sigma^2} - \frac{2\pi}{\sigma^2}k^\top(x-y)\big]} - \sum_{i=1}^M \frac{\big[\frac{2\pi}{\sigma^2}k^\top(x-y) - \frac{\|x-y\|^2}{2\sigma^2}\big]^i}{i!}\Big]\,\dd\big[R_{T_j}(y) - R_j^{\rm rec}(y)\big]\bigg\rvert\\
&\qquad + \vol(\mb T^d)\cdot\bigg\lvert \sum_{k\in\mb Z^d}e^{-\frac{2\pi^2\|k\|^2}{\sigma^2}}\int_{\mb T^d}\sum_{i=1}^M \frac{\big[\frac{2\pi}{\sigma^2}k^\top(x-y) - \frac{\|x-y\|^2}{2\sigma^2}\big]^i}{i!}\,\dd\big[R_{T_j}(y) - R_j^{\rm rec}(y)\big]\bigg\rvert.
\end{align*}
The second term is zero, since $\sum_{i=1}^M \frac{\big[\frac{2\pi}{\sigma^2}k^\top(x-y) - \frac{\|x-y\|^2}{2\sigma^2}\big]^i}{i!}$ is a polynomial of $y$ with degree no greater than $2M$. The first term can be bounded by
\begin{align*}
&\bigg\lvert \sum_{k\in\mb Z^d}e^{-\frac{2\pi^2\|k\|^2}{\sigma^2}}\int_{\mb T^d}\Big[e^{-\big[\frac{\|x-y\|^2}{2\sigma^2} - \frac{2\pi}{\sigma^2}k^\top(x-y)\big]} - \sum_{i=1}^M \frac{\big[\frac{2\pi}{\sigma^2}k^\top(x-y) - \frac{\|x-y\|^2}{2\sigma^2}\big]^i}{i!}\Big]\,\dd\big[R_{T_j}(y) - R_j^{\rm rec}(y)\big]\bigg\rvert\\
&\leq 2\sup_{\substack{u=x-y\\ x, y\in\mb T^d}}\sum_{k\in\mb Z^d}e^{-\frac{2\pi^2\|k\|^2}{\sigma^2}}\bigg\lvert e^{-\big[\frac{\|u\|^2}{2\sigma^2} - \frac{2\pi}{\sigma^2}k^\top u\big]} - \sum_{i=1}^M \frac{\big[\frac{2\pi}{\sigma^2}k^\top u - \frac{\|u\|^2}{2\sigma^2}\big]^i}{i!}\bigg\rvert\\
&\stackrel{\ri}{\leq} 2\sum_{k\in\mb Z^d}e^{-\frac{2\pi^2\|k\|^2}{\sigma^2}}\sup_{\substack{u\in\mb R^d\\ \|u\|_{\ell^\infty}\leq 2\pi}} \frac{\big[\frac{2\pi}{\sigma^2}k^\top u - \frac{\|u\|^2}{2\sigma^2}\big]^{M+1}}{(M+1)!} 
\leq 2 \sum_{k\in\mb Z^d}e^{-\frac{2\pi^2\|k\|^2}{\sigma^2}} \frac{\big[\frac{4\pi^2}{\sigma^2}\|k\|_{\ell^1} + \frac{2\pi^2d}{\sigma^2}\big]^{M+1}}{(M+1)!}\\
&\leq 2 \sum_{k\in\mb Z^d}e^{-\frac{2\pi^2\|k\|_{\ell^1}^2}{d\sigma^2}} \frac{\big[\frac{4\pi^2}{\sigma^2}\|k\|_{\ell^1} + \frac{2\pi^2d}{\sigma^2}\big]^{M+1}}{(M+1)!}.
\end{align*}
Here, (i) is by Taylor expansion (or mean-value theorem). To further control this upper bound, actually we can show that
\begin{align}\label{eqn: upbd_gausskernel}
\sum_{k\in\mb Z^d}e^{-\frac{2\pi^2\|k\|_{\ell^1}^2}{d\sigma^2}} \frac{\big[\frac{4\pi^2}{\sigma^2}\|k\|_{\ell^1} + \frac{2\pi^2d}{\sigma^2}\big]^{M+1}}{(M+1)!} 
\leq \frac{(C_6M\log M)^{\frac{M+1}{2}}}{(M+1)!},\qquad\forall\, M\in\mb Z_+ 
\end{align}
for some constant $C_6 = C_6(d, \sigma)$. With the above results, we get
\begin{align*}
J_1\leq \frac{2(C_6M\log M)^{\frac{M+1}{2}}}{(M+1)!} 
\leq2\Big(\frac{e}{M+1}\Big)^{M+1}(C_6M\log M)^{\frac{M+1}{2}} \leq 2\Big(\frac{C_6e^2\log M}{M+1}\Big)^{\frac{M+1}{2}}.
\end{align*}

Now, it is remained to prove the bound~\eqref{eqn: upbd_gausskernel}. Note that the left-hand side of~\eqref{eqn: upbd_gausskernel} is
\begin{align*}
\lhs &= \sum_{l=0}^\infty \sum_{k\in\mb Z^d: \|k\|_{\ell^1} = l} e^{-\frac{2\pi^2\|k\|_{\ell^1}^2}{d\sigma^2}}\frac{\big[\frac{4\pi^2\|k\|_{\ell^1}}{\sigma^2} + \frac{2\pi^2d}{\sigma^2}\big]^{M+1}}{(M+1)!}\\
&= \frac{[2\pi^2d]^{M+1}}{\sigma^{2(M+1)}(M+1)!} + \sum_{l=1}^\infty\sum_{k\in\mb Z^d: \|k\|_{\ell^1}=l}e^{-\frac{2\pi^2l^2}{d\sigma^2}}\frac{\big[\frac{4\pi^2l}{\sigma^2} + \frac{2\pi^2d}{\sigma^2}\big]^{M+1}}{(M+1)!}\\
&\leq \frac{[2\pi^2d]^{M+1}}{\sigma^{2(M+1)}(M+1)!} + \sum_{l=1}^\infty2^d\binom{l+d-1}{d-1} e^{-\frac{2\pi^2l^2}{d\sigma^2}} \frac{\big[\frac{4\pi^2l}{\sigma^2} + \frac{2\pi^2d}{\sigma^2}\big]^{M+1}}{(M+1)!}.
\end{align*}
In the last inequality, we use the fact that the equation $|k_1| + \dots + |k_d| = l$ has $\binom{l+d-1}{d-1}$ different solutions of $(|k_1|, \dots, |k_d|)\in\mb N^d$, corresponding to at most $2^d\binom{l+d-1}{d-1}$ different solutions of $(k_1, \dots, k_d)\in\mb Z^d$. By Stirling's formula, we know
\begin{align*}
\binom{l+d-1}{d-1} 
&= \frac{(l+d-1)!}{(d-1)!l!} \sim \frac{\sqrt{2\pi(l+d-1)}(\frac{l+d-1}{e})^{l+d-1}}{\sqrt{2\pi l}(\frac{l}{e})^l (d-1)!}\\
&= \sqrt{\frac{l+d-1}{l}}\Big(1 + \frac{d-1}{l}\Big)^l\Big(\frac{l+d-1}{e}\Big)^l\frac{1}{(d-1)!}\\
&\lesssim l^{d-1}.
\end{align*}
So, there are constants $C, C' > 0$ such that
\begin{align*}
\sum_{l=1}^\infty2^d\binom{l+d-1}{d-1} e^{-\frac{2\pi^2l^2}{d\sigma^2}} \frac{\big[\frac{4\pi^2l}{\sigma^2} + \frac{2\pi^2d}{\sigma^2}\big]^{M+1}}{(M+1)!}
\leq \sum_{l=1}^\infty C l^{d-1}e^{-\frac{2\pi^2 l^2}{d\sigma^2}}\frac{(C'l)^{M+1}}{(M+1)!}
\leq \frac{C(C')^{M+1}}{(M+1)!}\sum_{l=1}^\infty l^{M+d} e^{-\frac{2\pi^2l^2}{d\sigma^2}}.
\end{align*}
Note that we have
\begin{align*}
l^{M+d} \leq e^{\frac{\pi^2l^2}{d\sigma^2}} 
\Longleftrightarrow (M+d)\log l \leq \frac{\pi^2l^2}{d\sigma^2}
\Longleftarrow l \geq K\sqrt{M\log M}
\end{align*}
for some $K > 0$ independent of $M$. So, we have
\begin{align*}
\sum_{l=1}^\infty l^{M+d}e^{-\frac{2\pi^2l^2}{d\sigma^2}}
&\leq \sum_{l=1}^{K\sqrt{M\log M}}l^{M+d}e^{-\frac{2\pi^2l^2}{d\sigma^2}} + \sum_{l = K\sqrt{M\log M}}^\infty e^{-\frac{\pi^2l^2}{d\sigma^2}}\\
&\leq (K\sqrt{M\log M})^{M+d} + C''
\leq (C''' M\log M)^{\frac{M+1}{2}}
\end{align*}
for some large enough constant $C'', C'''$ independent of $M$. We finish the proof.
\subsection{Calculation of equation~\eqref{eqn: cal1}}\label{app: cal1}
By the definition~\eqref{eqn: explicit_quad_anneal_update} of $\wt\rho_j^k$, we have
\begin{align*}
&\quad\,\sum_{j=1}^m\int_{\m X} \tau\log\wt\rho_j^k(y_j)\,\dd[\wt\rho_j^k - \rho_j]
= \tau\big[H(\wt\rho^k) - U_k^\ast\big] + \sum_{j=1}^m\int_{\m X}\tau\sum_{l=1}^k\Big[\eta_l\prod_{l<l'\leq k}(1-\tau\eta_{l'})\Big]\big[V_j\big(y_j; \wht\rho^{l-1}\big) + \alpha_l\|y_j\|^2\big]\,\dd\rho_j.
\end{align*}
Therefore, we have
\begin{align*}
&\quad\,\sum_{k=1}^{K}\sum_{j=1}^m \int_{\m X}\eta_{k+1}V_j(y_j; \wht\rho^k) + \tau\eta_{k+1}\log\wt\rho_j^k(y_j)\,\dd[\wt\rho_j^k - \rho_j]\\
&= \sum_{k=1}^{K}\sum_{j=1}^m\eta_{k+1} \int_{\m X}V_j(y_j; \wht\rho^k)\,\dd\wt\rho_j^k + \sum_{k=1}^{K}\eta_{k+1}\tau\big[H(\wt\rho^k) - U_k^\ast\big]\\
&\qquad\qquad + \sum_{k=1}^{K}\sum_{j=1}^m\int_{\m X}\tau\eta_{k+1}\sum_{l=1}^k\Big[\eta_l\prod_{l<l'\leq k}(1-\tau\eta_{l'})\Big]\cdot\alpha_l\|y_j\|^2\,\dd\rho_j\\
&\qquad\qquad + \sum_{k=1}^{K}\sum_{j=1}^m\int_{\m X}\tau\eta_{k+1}\sum_{l=1}^k\Big[\eta_l\prod_{l<l'\leq k}(1-\tau\eta_{l'})\Big]V_j(y_j;\wht\rho^{l-1})\,\dd\rho_j - \sum_{k=1}^{K}\sum_{j=1}^m\int_{\m X}\eta_{k+1}V_j(y_j; \wht\rho^k)\,\dd\rho_j.
\end{align*}
Note that the last line equals to
\begin{align*}
&\quad\,\sum_{j=1}^m \bigg[\sum_{k=1}^{K}\sum_{l=1}^k\frac{\tau\eta_l\eta_{k+1}(1-\tau\eta_1)\cdots(1-\tau\eta_k)}{(1-\tau\eta_1)\cdots(1-\tau\eta_l)}\int_{\m X} V_j(y_j;\wht\rho^{l-1})\,\dd\rho_j - \sum_{k=1}^{K}\eta_{k+1}\int_{\m X} V_j(y_j;\wht\rho^k)\,\dd\rho_j\bigg]\\
&= \sum_{j=1}^m \bigg[\sum_{l=1}^{K}\sum_{k=l}^{K}\frac{\tau\eta_l\eta_{k+1}(1-\tau\eta_1)\cdots(1-\tau\eta_k)}{(1-\tau\eta_1)\cdots(1-\tau\eta_l)}\int_{\m X} V_j(y_j;\wht\rho^{l-1})\,\dd\rho_j - \sum_{k=1}^{K}\eta_{k+1}\int_{\m X} V_j(y_j;\wht\rho^k)\,\dd\rho_j\bigg]\\
&\stackrel{\ri}{=} \sum_{j=1}^m\bigg[\sum_{l=1}^K\eta_l\big[1 - (1-\tau\eta_{l+1})\cdots(1-\tau\eta_{K+1})\big]\int_{\m X}V_j(y_j;\wht\rho^{l-1})\,\dd\rho_j - \sum_{k=1}^K\int_{\m X}\eta_{k+1}V_j(y_j;\wht\rho^k)\,\dd\rho_j\bigg]\\
&= \sum_{j=1}^m\bigg[\eta_1\int_{\m X}V_j(y_j;\wht\rho^0)\,\dd\rho_j - \eta_{K+1}\int_{\m X}V_j(y_j;\wht\rho^K)\,\dd\rho_j - \sum_{l=1}^K\eta_l(1-\eta_{l+1})\cdots(1-\tau\eta_{K+1})\int_{\m X}V_j(y_j;\wht\rho^{l-1})\,\dd\rho_j\bigg]\\
&\stackrel{\rii}{=} \sum_{j=1}^m\eta_1\int_{\m X}V_j(y_j;\wht\rho^0)\,\dd\rho_j -\bigg[U_{K+1}(\rho) - H(\rho) - \sum_{k=1}^{K+1}\eta_k\alpha_k(1-\tau\eta_{k+1})\cdots(1-\tau\eta_{K+1})\int_{\m X}\|y\|^2\,\dd\rho\bigg].
\end{align*}
Here, (i) is due to the identity
\begin{align*}
\sum_{k=l}^K \tau\eta_{k+1}(1-\tau\eta_1)\cdots(1-\tau\eta_k) = (1-\tau\eta_1)\cdots(1-\tau\eta_l)\big[1 - (1-\tau\eta_{l+1})\cdots(1-\tau\eta_{K+1})\big],
\end{align*}
and (ii) is by definition of the functional $U_k$. Thus, we have
\begin{align*}
&\quad\,\sum_{k=1}^{K}\sum_{j=1}^m \int_{\m X}\eta_{k+1}V_j(y_j; \wht\rho^k) + \tau\eta_{k+1}\log\wt\rho_j^k(y_j)\,\dd[\wt\rho_j^k - \rho_j]\\
&= \sum_{k=1}^{K}\sum_{j=1}^m\eta_{k+1} \int_{\m X}V_j(y_j; \wht\rho^k)\,\dd\wt\rho_j^k + \sum_{k=1}^{K}\eta_{k+1}\tau\big[H(\wt\rho^k) - U_k^\ast\big]\\
&\qquad\qquad + \sum_{k=1}^{K}\sum_{j=1}^m\int_{\m X}\tau\eta_{k+1}\sum_{l=1}^k\Big[\eta_l\prod_{l<l'\leq k}(1-\tau\eta_{l'})\Big]\cdot\alpha_l\|y_j\|^2\,\dd\rho_j\\
&\qquad\qquad + \sum_{j=1}^m\eta_1\int_{\m X}V_j(y_j;\wht\rho^0)\,\dd\rho_j -\bigg[U_{K+1}(\rho) - H(\rho) - \sum_{k=1}^{K+1}\eta_k\alpha_k(1-\tau\eta_{k+1})\cdots(1-\tau\eta_{K+1})\int_{\m X}\|y\|^2\,\dd\rho\bigg]\\
&= \sum_{k=1}^K\sum_{j=1}^m\eta_{k+1}\int_{\m X}V_j(y_j;\wht\rho^k)\,\dd\wt\rho_j^k + \sum_{k=1}^K\tau\eta_{k+1}\big[H(\wt\rho^k) - U_k^\ast\big] + \sum_{j=1}^m \eta_1\int_{\m X}V_j(y_j;\wht\rho^0)\,\dd\rho_j - U_{K+1}(\rho) + H(\rho)\\
&\qquad\qquad + \sum_{k=1}^{K+1} \alpha_k\eta_k\int\|y\|^2\,\dd\rho.
\end{align*}
\end{document}